\documentclass[a4paper,reqno]{amsart}

\usepackage[utf8]{inputenc}
\usepackage[T1]{fontenc}
\usepackage[english]{babel}

\usepackage{amsmath}
\usepackage{amsfonts}
\usepackage{amssymb}
\usepackage{amsthm}
\usepackage[abbrev]{amsrefs}
\usepackage{enumitem}
\usepackage{cleveref}

\newcommand{\C}{\mathbb{C}}
\newcommand{\N}{\mathbb{N}}
\newcommand{\R}{\mathbb{R}}
\newcommand{\T}{\mathbb{T}}
\newcommand{\Z}{\mathbb{Z}}

\newcommand{\F}{\mathcal{F}}

\newcommand{\set}[1]{\left\{#1\right\}}
\newcommand{\paren}[1]{\left(#1\right)}
\newcommand{\norm}[1]{\|#1\|}
\newcommand{\bignorm}[1]{\left\|#1\right\|}

\newcommand{\JBX}[1][x]{\langle#1\rangle}
\newcommand{\1}[1]{\frac{1}{#1}}
\newcommand{\eps}{\varepsilon}

\let\d\relax
\DeclareMathOperator{\d}{d \!}

\DeclareMathOperator{\supp}{supp}
\DeclareMathOperator{\sech}{sech}

\newtheorem{theorem}{Theorem}
\newtheorem*{definition}{Definition}
\newtheorem{proposition}[theorem]{Proposition}
\newtheorem{lemma}[theorem]{Lemma}
\newtheorem{corollary}[theorem]{Corollary}
\newtheorem{remark}[theorem]{Remark}

\numberwithin{equation}{section}
\numberwithin{theorem}{section}

\author{Joseph Adams}
\address{Heinrich-Heine-Universität Düsseldorf, Mathematisches Institut, Universitätsstraße~1, 40225 Düsseldorf, Germany}
\email{joseph.adams@hhu.de}

\makeatletter
\@namedef{subjclassname@2020}{\textup{2020} Mathematics Subject Classification}
\makeatother
\subjclass[2020]{Primary: 35Q55. Secondary: 35Q41}
\keywords{NLS hierarchy equations --- low regularity local well-posedness --- Fourier restriction norm method}

\title{Well-posedness for the NLS hierarchy}
\begin{document}
    \begin{abstract}
    We prove well-posedness for higher-order equations in the so-called NLS hierarchy (also known as part of the AKNS hierarchy) in almost critical Fourier-Lebesgue spaces and in modulation spaces. We show the $j$th equation in the hierarchy is locally well-posed for initial data in $\hat H^s_r(\R)$ for $s \ge \frac{j-1}{r'}$ and $1 < r \le 2$ and also in $M^s_{2, p}(\R)$ for $s = \frac{j-1}{2}$ and $2 \le p < \infty$. Supplementing our results with corresponding ill-posedness results in Fourier-Lebesgue and modulation spaces shows optimality. Using the conserved quantities derived in~\cite{Koch2018} we argue that the hierarchy equations are globally well-posed for data in $H^s(\R)$ for $s \ge \frac{j-1}{2}$.

    Our arguments are based on the Fourier restriction norm method in Bourgain spaces adapted to our data spaces and bi- \& trilinear refinements of Strichartz estimates.
    \end{abstract}
	\maketitle
    \tableofcontents

\section{Introduction}

The cubic nonlinear Schrödinger (NLS) equation
\begin{equation}\label{eq:cubic-NLS}
\begin{cases}
i\partial_t u + \partial_x^2 u = \pm 2|u|^2 u \\
u(t=0) = u_0
\end{cases}
\end{equation}
with initial data $u_0$, has over the past 30 years become one of the canonical objects of study in the well-posedness theory of dispersive PDEs. We direct the interested reader to~\cites{Cazenave2003, Bourgain1993-1, TaoBook2006} and the references therein for an overview of developments in its study.

Contemporary research is more and more leaning into the fact that the NLS equation possesses a rich internal structure that may be exploited in order to prove new well-posedness results or a-priori bounds on solutions. We are, of course, referring to the fact that the NLS equation is considered to be a completely integrable system~\cites{AKNS1974, Faddeev1987, Palais1997, Koch2018, Alberty1982-i} -- an exact definition of which though escapes the literature. Usually one considers the fact that there exists an infinite sequence of non-trivial conserved quantities one of the markers of complete integrability. A fact that is also true for the NLS equation. The first few of these conserved quantities are
\begin{align*}
\text{Mass:} && H_0 = \int |u|^2 \d{x}\\
\text{Momentum:} && H_1 = -i \int u \partial_x \overline{u} \d{x} \\
\text{Energy:} && H_2 = \int |\partial_x u|^2 \pm |u|^4 \d{x}
\end{align*}
More precisely, the NLS equation is a Hamiltonian equation that is induced by its energy $H_2$. (Induced in what way we will make more precise in Section~\ref{sec:nls-hierarchy}.) This begs the question: do the higher-order conserved quantities $H_3, H_4, \ldots$ also induce any \emph{interesting} dispersive PDE\footnote{The mass $H_0$ and momentum $H_1$ also induce PDE, namely of phase shifts and of translations. Though these are not dispersive and thus are of no interest to us.}?

Yes, in fact the fourth conserved quantity
\begin{equation*}
    H_3 = i \int \partial_x u \partial_x^2 \overline{u} + 3|u|^2 u \partial_x \overline{u} \d{\lambda}
\end{equation*}
induces the also well-known modified Korteweg-de-Vries (mKdV) equation\footnote{There is a caveat to this that is discussed in Appendix~\ref{appendix:nls-hierarchy}. In short, when looking at the complex mKdV equations a slightly different nonlinearity is produced when one follows the construction of the NLS hierarchy in~\cite{Alberty1982-i}, as we do. When looking at the real valued mKdV equation there is no discrepancy.}
\begin{equation}\label{eq:mKdV}
\begin{cases}
\partial_t u + \partial_x^3 u = \pm 2\partial_x (|u|^2 u)\\
u(t=0) = u_0
\end{cases}
\end{equation}
see~\cites{FabFive-mKdV-2003, Koch2018, KPV1993} for an overview.

The next higher-order equation is not quite as well known as the NLS and/or mKdV equations, though it has also appeared independently in the literature~\cites{Feng1994, Feng1995}.

To the author's best knowledge there is no complete description of all conserved quantities of NLS available. It is though a simple, but tedious, task to calculate them. See Appendix~\ref{appendix:nls-hierarchy}, where we list more of the conserved quantities and their associated equations.

We want to mention at this point, that the pattern of even-numbered conserved quantities $H_{2k}$, $k\in\N$ inducing NLS-like equations and odd-numbered ones $H_{2k+1}$, $k\in\N$ inducing mKdV-like equations continues~\cites{Koch2018, AGTowers}. This sequence of NLS-like equations is what is referred to in the title of this paper as the NLS hierarchy\footnote{The NLS hierarchy is a part of what is often called the AKNS hierarchy, after the names of the authors that played a considerable role in developing the inverse scattering transform, see~\cite{AKNS1974}.}. We will give a more precise definition of the NLS hierarchy in Section~\ref{sec:nls-hierarchy}.

Aim of this paper is to deal with questions of low-regularity well-posedness for the NLS hierarchy equations in classical Sobolev spaces $H^s(\R)$, Fourier-Lebesgue spaces $\hat H^s_r(\R)$ (sometimes written as $\F L^{s,r'}(\R)$, where the integrability exponent is conjugated) defined by the norm
\begin{equation*}
    \norm{u}_{\hat H^s_r} = \norm{u}_{\F L^{s, r'}} = \norm{\JBX[\xi]^s \hat{u}}_{L^{r'}}
\end{equation*}
and modulation spaces $M^s_{2, p}(\R)$ defined by the norm
\begin{equation*}
    \norm{u}_{M^s_{2, p}} = \norm{\norm{\Box_n u}_{H^s}}_{\ell_n^p(\Z)}
\end{equation*}
with a family of isometric decomposition operators $(\Box_n)_{n\in\Z}$. Precise definitions of the function spaces and an overview of associated embeddings are given in Section~\ref{sec:notation}.

While we embrace the rich integrability structure of these equations for their derivation and conservation laws, we will not be making use of their integrability to argue our local well-posedness results. This has the advantage that our arguments work for a rather large class of equations, but the disadvantage that we also cannot utilise any special structure that may be present in the NLS hierarchy equations, that could aid the well-posedness.

Moreover our arguments will be based on the contraction mapping principle in versions of Bourgain spaces $X_{s,b}$ adapted to our data spaces, in combination with bi- and trilinear refinements of Strichartz estimates.

\subsection*{Acknowledgements} This work is part of the author's PhD thesis. He would like to greatly thank his advisor, Axel Grünrock, for suggesting this line of problems and his continued and ongoing support.

The author would also like to thank the second anonymous reviewer for his suggested corrections and additions. In particular asking a question about the influence of the distribution of complex conjugates in quintic and higher-order terms in the Fourier-Lebesgue space setting that lead to Remarks~\ref{rem:hat-space-complex-conj} and~\ref{rem:hat-space-complex-conj-proof}.

\subsection{Organisation of the paper}
In the next and final subsection of this introduction we will establish the general notation and function spaces that we will be using throughout the rest of this paper. The acquainted reader may skip immediately to Section~\ref{sec:nls-hierarchy}.

Following that we will define exactly what we mean by NLS hierarchy (and its generalisations) in Section~\ref{sec:nls-hierarchy}.

We give an overview of prior work related to the well-posedness study of hierarchies of PDEs, a statement of our main results and a discussion of these in Section~\ref{sec:results}.

In Section~\ref{sec:linear-est} we collect general smoothing estimates based on the dispersion present in the equations we are dealing with. This includes linear estimates we will be citing from the literature, some new bilinear estimates adapted to the case of higher-order Schrödinger equations, so-called Fefferman-Stein estimates (which generalize Strichartz estimates to the Fourier-Lebesgue spaces we will be using), and trilinear estimates. The new bi- and trilinear refinements as well as the Fefferman-Stein estimates are based on~\cite{AGTowers}.

Then in Section~\ref{sec:multilinear-estimates} we will follow up with the nonlinear estimates needed to prove Theorems~\ref{thm:wp-hat} and~\ref{thm:wp-modulation}. First we deal with estimates regarding well-posedness in Fourier-Lebesgue spaces, following up with the same for modulation spaces.

Finally in Section~\ref{sec:illposedness} we will deal with the question of ill-posedness. In this section we will see that, on the line, our methods lead to optimal results in the framework that we use. Also we will deal with the fact that a fixed-point theorem based approach cannot work in the same generality on the torus, as it does on the line.

\subsection{Notation and function spaces}\label{sec:notation}
We use the notation $A \lesssim B$ to mean $A \le CB$ for a constant $C > 0$ independent of $A$ and $B$, and $A \sim B$ denotes $A \lesssim B$ and $A \gtrsim B$, while $A \ll B$ means $A \le \eps B$ for a \emph{small} constant $\eps > 0$. For a given real number $a \in \R$ we will denote by $a+$ and $a-$ the numbers $a + \eps$ and $a - \eps$ for an arbitrarily small $\eps > 0$, respectively. The so called Japanese brackets denote the quantity $\JBX[x] = (1 + x^2)^\1{2}$.

We use the following conventions regarding the Fourier-transform: the Fourier-transform of a function $u : \R_x \times \R_t \to \C$ with respect to the space-variable $x$ is given by
\begin{equation*}
\F_x u(\xi, t) = \1{\sqrt{2\pi}} \int_\R u(x, t) e^{ix\xi} \d{x}.
\end{equation*}
The Fourier-transform with respect to the time-variable is defined analogously, though the Fourier-variable corresponding to $t$ shall be called $\tau$. We will also use the notation $\hat{u}$ to denote the Fourier-transform with respect to either one or both of those variables, but it will be clear from context which of those cases we are referring to, specifically from the use of spatial- and time-Fourier variables rather than their physical-space counterparts.

For two functions $f$ and $g$ we use the notation
\begin{equation*}
    \int_* f(\xi_1) g(\xi_2) \d{\xi_1} = \int_\R f(\xi_1) g(\xi - \xi_1) \d{\xi_1}
\end{equation*}
to represent the integral under the convolution constraint $\xi = \xi_1 + \xi_2$. This generalises naturally to an arbitrary number of functions.

Given $s \in \R$ we define the Bessel potential operator $J^s$ through its Fourier transform $\F J^s u(\xi) = \JBX[\xi]^s \hat{u}(\xi)$ for a function $u$, and similarly the Riesz potential operator $I^s$ as $\F I^s u(\xi) = |\xi|^s \hat{u}(\xi)$.

Next we define the frequency projections that we will be utilizing. Given a dyadic number $N \in 2^\N$ let $P_N$ denote the Littlewood-Paley projector onto the (spatial) frequencies $\set{\xi \in \R \mid |\xi| \sim N}$. The special case $P_1$ shall mean the projector onto the (spatial) frequencies $\set{\xi \in \R \mid |\xi| \lesssim 1}$. We direct the reader to~\cite{Grafakos-classic} for a reference on Littlewood-Paley theory.

For $n \in \Z$, let the uniform (or isometric) frequency decomposition operators $\Box_n$ be defined by
\begin{equation*}
\widehat{\Box_n f}(\xi) = \psi(\xi - n) \hat{f}(\xi),
\end{equation*}
where $\psi : \R \to \R$ is a smooth cut-off function with the properties $\supp \psi \subset [-\1{4}, \frac{5}{4}]$ and $\psi(\xi) \equiv 1$ on $[0, 1]$.

For these operators it is well known that, for any $1 \le q \le p \le \infty$, one has
\begin{equation*}
\norm{P_N f}_{L^p} \lesssim N^{\1{q} - \1{p}} \norm{f}_{L^q} \quad\text{and}\quad \norm{\Box_n f}_{L^p} \lesssim \norm{f}_{L^q}.
\end{equation*}

When dealing with estimates of products of frequency localized functions, to simplify notation, we will adhere to the following convention: for $n \in \Z$ or a dyadic number $N \in 2^\N$ we write $u_n = \Box_n u$ or $u_N = P_N u$ respectively. Complex conjugation has higher precedence than this notation, so that $\overline{u}_n = \overline{(u_{-n})}$. Different indices on different factors will not cause confusion, as we will not mix dyadic and uniform frequency localisation. Also, for ease of presentation, subscripts referring to frequency localisation may suppress other indices of functions, i.e. using $u_\ell \overline{u}_m u_n$ to refer to $(\Box_\ell u_1) (\Box_m \overline{u}_2) (\Box_n u_3)$.

Next let us define the Fourier-Lebesgue spaces $\hat{H}^s_r(\R^n)$ (also referred to as hat-spaces for obvious reasons), for $s \in \R$ and $1 \le r \le \infty$, to be the subspace of functions $u \in \mathcal{S}'(\R^n)$ such that the norm
\begin{equation*}
\norm{u}_{\hat{H}^s_r} = \norm{\JBX[\xi]^s \hat{u}}_{L^{r'}(\R^n)}
\end{equation*}
is finite. In the case $s=0$ one may resort to the slightly different notation $\hat{H}^0_r = \widehat{L^r}$. And similarly we define the modulation space $M^s_{q, p}(\R^n)$, for $s \in\R$ and $1 \le p,q \le \infty$, as the subspace of functions $u \in \mathcal{S}'(\R^n)$ such that the norm
\begin{equation*}
\norm{u}_{M^s_{q, p}} = \norm{\JBX[n]^s \norm{\Box_n u}_{L^q(\R)}}_{\ell^p_n(\Z)}
\end{equation*}
is finite.

Though we will not be exhaustive with the properties that these spaces have, we do want to emphasize an embedding connecting Fourier-Lebesgue and modulation spaces. For $p \ge 2$ one has $M^s_{2, p}(\R) \supset \hat H^s_{p'}(\R)$. This embedding can be utilised to gain a notion of criticality in modulation spaces (that are otherwise not well-behaved with respect to transformations of scale because of the isometric frequency decomposition). Also we mention, that in the periodic setting these data spaces actually coincide, i.e. $M^s_{2, p}(\T) = \hat H^s_{p'}(\T)$.

Furthermore we note, that both Fourier-Lebesgue and modulation spaces behave in a natural way with respect to complex interpolation and duality. Let $\theta \in [0,1]$, $s, s_0, s_1 \in \R$ and $1 < r, r_0, r_1, p, p_0, p_1 \le \infty$. Then for $s = (1-\theta)s_0 + \theta s_1$ one has the following interpolation identities
\begin{align*}
&\left[ \hat H^{s_0}_{r_0}, \hat H^{s_1}_{r_1} \right]_{[\theta]} = \hat H^s_r \;\;\text{for $\1{r} = \frac{1-\theta}{r_0} + \frac{\theta}{r_1}$ as well as}\\
&\left[ M^{s_0}_{2,p_0}, M^{s_1}_{2,p_1} \right]_{[\theta]} = M^s_{2,p} \;\;\text{for $\1{p} = \frac{1-\theta}{p_0} + \frac{\theta}{p_1}$}.
\end{align*}
as long as $(p_0, p_1) \not= (\infty, \infty)$.
Under the additional constraint that $p < \infty$ the following duality relationships
\begin{equation*}
    \left( \hat H^{s}_{r} \right)' \cong \hat H^{-s}_{r'} \;\text{and}\; \left( M^{s}_{2,p} \right)' \cong M^{-s}_{2,p'}
\end{equation*}
also hold. We mention~\cites{Benyi2020,Feichtinger1983} as references for embedding, duality and interpolation results regarding modulation spaces.

In order to prove local well-posedness we have a necessity for spaces that are more well-adapted to performing a contraction mapping argument. In~\cites{Bourgain1993-1, Bourgain1993-2} Bourgain introduced the now almost classical $X_{s,b}$ spaces dependent on a phase function $\varphi : \R^n \to \R$ and $s, b \in\R$, defined by the norm
\begin{equation*}
\norm{u}_{X_{s,b}} = \norm{\JBX[\xi]^s\JBX[\tau - \varphi(\xi)]^b \hat{u}}_{L^2_{xt}}.
\end{equation*}

Using these spaces to study the local well-posedness of dispersive PDEs has since become known as the Fourier restriction norm method. It was later refined and built upon in~\cites{GTV1997, KPV1996-i, KPV1996-ii} to arrive at its current use state.

In connection with the $X_{s, b}$ spaces we also define the operator $\Lambda^b$ through its Fourier transform as $\F \Lambda^b u(\xi, \tau) = \JBX[\tau - \phi(\xi)]^b \hat{u}(\xi,\tau)$ for a function $u$. The quantity $\sigma = \tau - \phi(\xi)$ is referred to as the modulation.

In the following, we define $X_{s,b}$ spaces adapted to the Fourier-Lebesgue $\hat{H}^s_r$ and modulation spaces $M^s_{2,p}$ we will be using as data spaces. For papers dealing in the same spaces see e.g.~\cites{GrünrockHerr2008, GrünrockVega2009, OhWang2021}.

For $s,b \in \R$ and $1 \le r \le \infty$, we denote the Bourgain spaces adapted to Fourier-Lebesgue spaces by $\hat{X}_{s,b}^r$. They are defined as the subspace of $\mathcal{S}'(\R^2)$ induced by the norm
\begin{equation*}
\norm{u}_{\hat{X}_{s,b}^r} = \norm{\JBX[\xi]^s \JBX[\tau - \varphi(\xi)]^b \hat{u}}_{L^{r'}_{xt}} = \norm{J^s \Lambda^b u}_{\widehat{L^r_{xt}}},
\end{equation*}
so that the classical $X_{s, b}$ spaces can be recovered by setting $r = 2$. Note the lack of inverse Fourier transformation. Recall that for $1 \le r \le \infty$ we have the following embedding:
\begin{equation*}
\hat{X}_{s,b}^r \hookrightarrow C(\R; \hat{H}^{s}_r(\R)) \quad\text{if}\; b > \1{r}.
\end{equation*}
The contraction mapping argument leading to well-posedness will be carried out in their respective time restriction norm spaces
\begin{equation*}
\hat{X}_{s,b}^r(\delta) = \set{u = \left.\tilde{u}\right|_{\R \times [-\delta, \delta]} \mid \tilde{u} \in \hat{X}_{s,b}^r}
\end{equation*}
endowed with the norm
\begin{equation*}
\norm{u}_{\hat{X}_{s,b}^r(\delta)} = \inf \set{\norm{\tilde{u}}_{\hat{X}_{s,b}^r} \mid \tilde{u} \in \hat{X}_{s,b}^r, \left.\tilde{u}\right|_{\R \times [-\delta, \delta]} = u}.
\end{equation*}

Similarly, for $s, b \in \R$ and $1 \le p \le \infty$ we define the Bourgain spaces adapted to modulation spaces $X_{s,b}^p$ (note the missing circumflex compared to the Fourier-Lebesgue based spaces). In this case they are the subspace of $\mathcal{S}'(\R^2)$ induced by the norm
\begin{equation*}
\norm{u}_{X_{s,b}^p} = \bignorm{\JBX[n]^s \norm{\Box_n u}_{X_{0,b}}}_{\ell_n^p}.
\end{equation*}
Again, $p=2$ corresponds to the classical case. The embedding giving us the persistence property is paralleled by
\begin{equation*}
X_{s,b}^p \hookrightarrow C(\R; M^s_{2,p}(\R)) \quad\text{if}\; b > \1{2}.
\end{equation*}
for $1 \le p \le \infty$. In the same fashion as for the Fourier-Lebesgue adapted spaces, we have time restriction norm spaces $X_{s,b}^p(\delta)$.

\begin{remark}
We fix $q = 2$ in the modulation space setting, because of the lack of available good (i.e. time independent) linear estimates in the $q\not=2$ case, see~\cites{Klaus2023, Benyi2020}.
\end{remark}

Having defined the spaces we will be using it is time to mention some of their properties. Among other things what makes Bourgain spaces useful is the ability to transfer estimates of free solutions in (mixed) $L^p$ spaces or their Fourier-Lebesgue cousins $\widehat{L^r}$ to estimates in $\hat X_{s,b}^r$ spaces. This is commonly known as a transfer principle. For a proof in the classical spaces we direct the reader to the self-contained exposition in~\cite{AGDiss}. The arguments for transferring (multi)linear estimates to the Fourier-Lebesgue variants $\hat X_{s,b}^r$ are contained within~\cite{Grünrock2004}.

Also contained in~\cite{Grünrock2004} is a general local well-posedness theorem for $\hat X_{s, b}^r$ spaces. A similar result, though for the modulation space variants $X_{s, b}^p$, can be found in~\cite{OhWang2021}, though which can easily be derived from the classics~\cites{Ginibre1996, GTV1997}. Using these general well-posedness theorems we will establish our well-posedness theorems with mere proofs of necessary multilinear estimates.

As we will also be using complex multilinear interpolation and duality arguments we shall state the relevant properties of our solution, and data spaces. For this let $\theta \in [0,1]$, $s, s_0, s_1, b, b_0, b_1 \in \R$ and $1 < r, r_0, r_1, p, p_0, p_1 \le \infty$. Then for $s = (1-\theta)s_0 + \theta s_1$ and $b = (1-\theta)b_0 + \theta b_1$ one has the following complex interpolation relations
\begin{align*}
    &\left[\hat X_{s_0, b_0}^{r_0}, \hat X_{s_1, b_1}^{r_1} \right]_{[\theta]} = \hat X_{s, b}^{r} \quad\text{when $\1{r} = \frac{1-\theta}{r_0} + \frac{\theta}{r_1}$ and}\\
    &\left[X_{s_0, b_0}^{p_0}, X_{s_1, b_1}^{p_1} \right]_{[\theta]} = X_{s, b}^{p} \quad\text{when $\1{p} = \frac{1-\theta}{p_0} + \frac{\theta}{p_1}$,}
\end{align*}
at least if $(p_0, p_1) \not= (\infty, \infty)$.
Moreover, with respect to the $L^2$ inner-product, their dual spaces are given by
\begin{equation*}
    \left( \hat X_{s, b}^r \right)' \cong \hat X_{-s, -b}^{r'} \quad\text{and}\quad \left( X_{s, b}^p \right)' \cong X_{-s, -b}^{p'}
\end{equation*}
if one imposes the additional constraint $p < \infty$.

Finally we recall some common inequalities that will be useful in piecing together multilinear estimates that we can establish in $L^2$-based $X_{s,b}$ spaces:
\begin{equation}\label{eq:bernstein}
\norm{u_N}_{X_{s,b}^q} \lesssim N^{\max(0, \1{q} - \1{p})} \norm{u_N}_{X_{s,b}^p} \quad\text{and}\quad \sum_{N \ge 1} N^{0-} \norm{u_N}_{X_{s,b}^p} \lesssim \norm{u}_{X_{s,b}^p}.
\end{equation}

\section{The NLS hierarchy in detail}\label{sec:nls-hierarchy}

In describing what we refer to as the NLS hierarchy we most closely follow~\cite{Alberty1982-i}, where the general structure of nonlinear evolution equations that arise as zero-curvature conditions is described. Though there are many more good references for this topic (see for example~\cites{Faddeev1987, Palais1997}), the chosen work~\cite{Alberty1982-i} concisely contains all the details we need about the NLS hierarchy.

\subsection{From linear scattering to NLS}

We start out in a geometric context, where we have an $N \times N$ matrix of differential one-forms $\Omega$ depending on a so-called spectral parameter $\zeta \in\C$. For this matrix one can express a linear scattering problem~\cite{Alberty1982-i}*{eq.~(1.1)}
\begin{equation}\label{eq:linear-scattering}
\d{v} = \Omega v.
\end{equation}
Associated with this scattering problem is the zero-curvature (or integrability) condition~\cite{Alberty1982-i}*{eq.~(1.2)}
\begin{equation}\label{eq:zcc}
0 = \d{\Omega} - \Omega \wedge \Omega,
\end{equation}
which for the right choice of $\Omega$ will result in the NLS hierarchy equations (and many other classical dispersive PDE).

In particular, as in~\cite{Alberty1982-i}*{eq.~(1.3)}, we will use the Ansatz $\Omega = (\zeta R_0 + P)\d{x} + Q(\zeta) \d{t}$ with
\begin{equation*}
R_0 = \begin{pmatrix}
-i & 0 \\ 0 & i
\end{pmatrix}
\quad\text{and}\quad
P = \begin{pmatrix}
0 & q \\ r & 0
\end{pmatrix}.
\end{equation*}
We leave the choice of $Q$ open for now, but will refer back to it at a later point.

After a lengthy calculation, that we will not reproduce for brevities sake, it is established that the zero-curvature condition~\eqref{eq:zcc} can under our Ansatz be equivalently expressed as~\cite{Alberty1982-i}*{eq.~(2.3.5)}
\begin{equation}\label{eq:total-hamiltonian-flow}
    \frac{\d}{\d{t}} u = J \frac{\delta}{\delta u} \mathcal{H},
\end{equation}
where $u = \left(\begin{smallmatrix} r\\q \end{smallmatrix}\right)$ is a vector of the ``potentials''\footnote{Potentials are what we would usually refer to as the solution of, say, NLS. In the context of NLS we have the additional assumption $r = \pm\overline{q}$. They are referred to as potentials in~\cite{Alberty1982-i}, as they are the objects along which scattering happens in~\eqref{eq:linear-scattering}.}, $J = -2\left(\begin{smallmatrix} 0 & -i \\ i & 0\end{smallmatrix}\right)$, $\frac{\delta}{\delta r}$ is a functional derivative and $\mathcal{H}$ is \emph{the} Hamiltonian of the system, defined by
\begin{equation}\label{eq:general-hamiltonian}
    \mathcal{H} = 2 \sum_{n=0}^{\infty} \alpha_n(t) I_{n+1}.
\end{equation}
In this sum the $I_{n+1}$ represent the sequence of conserved quantities of our system, i.e. up to constant factor, what was referred to in the introduction as $H_{n}$. With~\cite{Alberty1982-i}*{eqns.~(3.1.6) and~(3.1.7)} we are given explicit expressions for calculating these conserved quantities recursively
\begin{equation}\label{eq:recursion-flux}
    I_n = \int_\R q Y_n \d{x} \quad\text{and}\quad Y_{n+1} = \1{2i}\left[ \partial_x Y_n - r\delta_{0, n} + q \sum_{k=1}^{n-1} Y_{n-k} Y_k \right]
\end{equation}
with $Y_0 = 0$.

The $\alpha_n(t)$ are the choice of $Q$ we left open previously. Referring again to~\cite{Alberty1982-i}, the $\alpha_n$ control the weight of each individual flow (induced each by $I_{n+1}$) in the overall equation~\eqref{eq:total-hamiltonian-flow}. Thus by choosing the coefficients $\alpha_n(t)$ appropriately we will be able to recover NLS and the other equations that are part of the NLS hierarchy.

It is important to mention that, as we are working under the assumption $r = +\overline{q}$ in the context of the NLS hierarchy, our choice of coefficients $\alpha_n(t)$ are subject to the constraints
\begin{equation}\label{eq:coeff-constraint}
    \alpha_{2n} = -\overline{\alpha_{2n}} \quad\text{and}\quad \alpha_{2n+1} = \overline{\alpha_{2n+1}}
\end{equation}
as layed out in~\cite{Alberty1982-i}*{Section 3.2.3}.

\subsection{Defining the NLS hierarchy}

Having established the general origin of the NLS hierarchy equations we are now ready to give an exact definition, i.e. fix a choice of $(\alpha_n)_{n\in\N_0}$. From there on we will derive the general structure of the equations in the NLS hierarchy by means of~\eqref{eq:recursion-flux}. This strictly larger class of equations will be very broad in the nonlinearities contained within, but still sufficiently small for us to be able to carry out our further analysis in this generalised context.

\begin{definition}
For $j\in\N$, we define the $j$th (defocusing) NLS hierarchy equation to be the Hamiltonian equation for the potential $q(x, t)$ in~\eqref{eq:total-hamiltonian-flow}, where we choose $\alpha_{2j} \equiv -i 2^{2j-1}$ and $\alpha_n \equiv 0$ for $n \not= 2j$ in~\eqref{eq:general-hamiltonian}. We identify occurrences of the potential $r(x, t)$ with the complex conjugate of $q(x,t)$, i.e. $r = +\overline{q}$.
\end{definition}
\begin{remark} A few remarks are in order:
\begin{enumerate}
    \item Note that our choice of $\alpha_{2j}$ aligns with the constraint in~\eqref{eq:coeff-constraint}. Since we only have a single non-zero $\alpha_n$ a simple rescaling (and possible time reversion) of the equation would lead to any arbitrary choice of $\alpha_{2j}$ that aligns with~\eqref{eq:coeff-constraint}.
    \item Since we are only interested in a single component of~\eqref{eq:total-hamiltonian-flow} we may simplify. The $j$th NLS hierarchy equation thus reads
    \begin{equation}\label{eq:general-nls-hierarchy-eqn}
    q_t = -2^{2j+1} \frac{\delta}{\delta r} \int_\R q Y_{2j+1} \d{x}
    \end{equation}
    with $Y_{2j+1}$ defined in~\eqref{eq:recursion-flux}, keeping in mind the identification $r = +\overline{q}$.
    \item The first NLS hierarchy equation ($j = 1$) corresponds to the classical defocusing cubic NLS equation. In the notation of the previous display it reads
    \begin{equation*}
    iq_t = -q_{xx} + 2q^2r = -q_{xx} + 2|q|^2q.
    \end{equation*}
    Later we will switch to the more common notation of calling the unknown function $u$ instead of $q$.
    \item Above we only defined the defocusing NLS hierarchy, corresponding to the +-sign in~\eqref{eq:cubic-NLS}. There is also an equivalent focusing NLS hierarchy (that builds on the focusing cubic NLS, corresponding to the $-$-sign in~\eqref{eq:cubic-NLS}). Its equations can be derived in the same way, though with the identification $r = -\overline{q}$. This possibility is also mentioned in~\cite{Alberty1982-i}*{Section~3.2.3}.
    \item No complete description of the NLS hierarchy, i.e. the choice of coefficients for the nonlinear terms, is known. A lengthy calculation leads to Appendix~\ref{appendix:nls-hierarchy}, where we list the first few conserved quantities and the associated equations.

    It would certainly be an interesting problem to derive a general formula describing the $j$th NLS hierarchy equation in detail.
    \item Instead of a choice of $(\alpha_k)_k$, where only even numbered $\alpha_k$ are non-zero, going the opposite route and having only a single $\alpha_k$ non-zero with $k$ uneven results in the real mKdV hierarchy.

    There is a caveat to this, that is also discussed in Appendix~\ref{appendix:nls-hierarchy}, where the identification $r = \pm \overline{q}$ does not lead to the (de)focusing complex mKdV hierarchy. Using $r = q$ (which is also a compatible choice with the model, see~\cite{Alberty1982-i}*{Section~3.2.2}) one arrives at the real mKdV hierarchy, which was discussed in~\cite{AGTowers}. This fact is also mentioned in~\cite{KochKlaus2023}*{Appendix~B}.

    \item Contained within this calculus of hierarchies is another well-known one, the KdV hierarchy. Choosing $r = 1$ (which is also a compatible choice in this model, see~\cite{Alberty1982-i}*{Section~3.2.1}) results in its equations. This is also remarked in~\cite{KochKlaus2023}*{Appendix~B}.
\end{enumerate}
\end{remark}

Having defined the NLS hierarchy equations we may now reason about their general structure. We claim the following proposition.

\begin{proposition}\label{prop:Yn-structure}
    For $n \in\N$ the terms $Y_n$ have the following properties:
    \begin{enumerate}
        \item $Y_n$ is a sum of monomials in $q$, $r$ and their derivatives.
        \item The polynomial $Y_n$ is homogeneous in the order of monomials, where the order (of a monomial) is defined as the sum of the total number of derivatives and number of factors in the monomial.
        \item In every monomial of $Y_n$ the total number of factors of $r$ and its derivatives is one greater than the total number of factors $q$ and its derivatives.
        \item The coefficients of $Y_n$ are an integer multiple of $(2i)^{-n}$.
        \item In $Y_n$ there is a single monomial with only one factor. It is $(2i)^{-n} \partial_x^{n-1} r$.
    \end{enumerate}
\end{proposition}
\begin{proof}
    All of the claims in this proposition are trivially true for $Y_1 = \frac{-1}{2i} r$. For all higher-order $Y_n$ they follow inductively using the recursion formula~\eqref{eq:recursion-flux}.
\end{proof}

It is only a small step from the polynomials $Y_n$ to the conserved quantities $I_n$ and their associated, via~\eqref{eq:general-nls-hierarchy-eqn}, evolution equations. Having derived the properties of $Y_n$ mentioned in Proposition~\ref{prop:Yn-structure} we are ready to state the general structure of the NLS hierarchy equations. In doing so we switch back to the more common notation of calling the unknown solution $u$ (instead of $q$).

\begin{theorem}
    For $j\in\N$, there exist coefficients $c_{k,\alpha} \in \mathbb{Z}$ for every $\alpha \in\N_0^{2k+1}$ with $|\alpha| = 2(j-k)$, for $1 \le k \le j$, such that the $j$th NLS hierarchy equation can be written as
    \begin{equation}\label{eq:nls-hierarchy-general-eq}
        i\partial_t u + (-1)^{j+1} \partial_x^{2j} u = \sum_{k=1}^j \sum_{\substack{\alpha \in \N_0^{2k+1} \\ |\alpha| = 2(j-k)}} c_{k, \alpha} \partial_x^{\alpha_1} u \prod_{\ell = 1}^{k} \partial_x^{\alpha_{2\ell}} \overline{u} \partial_x^{\alpha_{2\ell+1}} u.
    \end{equation}
\end{theorem}
\begin{proof}
    Of course, we heavily rely on the structure of $Y_{2j+1}$ established in the preceding proposition.

    First we deal with the linear part of equation~\eqref{eq:nls-hierarchy-general-eq}: all monomials part of $Y_{2j+1}$ have a coefficient, that is an integer multiple of $(2i)^{-(2j+1)} = -i 2^{-2j-1}$. Keeping in mind, that the ``leading term'' of $Y_{2j+1}$ is $\partial_x^{2j}r$  and reminding the reader of the formula for calculating functional derivatives: For a smooth function $f : \C^{N+1} \to \C$ and a functional
    \begin{equation}\label{eq:functional-derivative}
    F[\phi] = \int_\R f(\phi, \partial_x\phi, \partial_x^2 \phi, \ldots, \partial_x^N \phi) \d{x} \quad\text{one has}\quad \frac{\delta F}{\delta\phi} = \sum_{k=0}^N (-1)^{k}\partial_x^k \frac{\partial f}{\partial (\partial_x^k \phi)},
    \end{equation}
    we may now establish, that the linear part of the equation must read
    \begin{equation*}
        i\partial_t u + (-1)^{j+1}\partial_x^{2j} u = 0
    \end{equation*}
    and we may ignore the rest of the coefficients of the nonlinearity, as they are only integers.

    Using (2) and (3) from Proposition~\ref{prop:Yn-structure} and the fact that~\eqref{eq:functional-derivative} reduces the number of factors $\partial_x^k r$, for a $0 \le k \le 2j$, by one, it is clear that the nonlinear terms must have between three and $2j+1$ factors. Of these, now there must be one more factor $u$ (or its derivatives) compared to $\overline{u}$ (or its derivatives), as the functional derivative reduces the number of factors $\overline{u}$ (or its derivatives) by one.

    The homogeneity of these nonlinear terms fixes the number of total derivatives to $2(j -k)$, if there are $2k+1$ factors.

    Since~\eqref{eq:nls-hierarchy-general-eq} covers all possible nonlinearities that fulfil these restrictions we have established the claim of this theorem.
\end{proof}

In our later dealings we will not be relying on any more information about the structure of the NLS hierarchy equations than is given in the previous theorem. Thus it makes sense to give a name to this general class of equations.

\begin{definition}
    For $j \in \N$, we call an equation a (higher-order) NLS-like equation, if there exist coefficients $c_{k,\alpha} \in \mathbb{Z}$ for every $\alpha \in\N_0^{2k+1}$ with $|\alpha| = 2(j-k)$, for $1 \le k \le j$, such that the equation can be written as~\eqref{eq:nls-hierarchy-general-eq}.
\end{definition}

\begin{remark}
    In a previous remark we mentioned the possibility of differentiating between the defocusing and focusing NLS hierarchy. Since the difference between the two is solely in the distribution of signs in the nonlinearity, both the defocusing and focusing NLS hierarchy equations are higher-order NLS-like equations, according to the above definition.
\end{remark}

\begin{remark}
    A natural question is whether there are further hierarchies of dispersive PDE arising as zero-curvature conditions~\eqref{eq:zcc}, possibly stemming from a different Ansatz than $\Omega = (\zeta R_0 + P)\d{x} + Q(\zeta) \d{t}$, the one we used.

    Indeed, this question is discussed in a follow-up paper~\cite{Sasaki1982-ii} to~\cite{Alberty1982-i}, where the Ansatz $\Omega = (\zeta^2 R_0 + \zeta P)\d{x} + Q(\zeta) \d{t}$ is used to derive the derivative nonlinear Schrödinger (dNLS) equation
    \begin{equation*}
        i\partial_t u + \partial_x^2 u = \pm i\partial_x (|u|^2 u).
    \end{equation*}
    and more generally its associated hierarchy of PDEs.

    The dNLS equation itself is an interesting object of study in the field of dispersive PDE, see for example~\cites{KNV2023, HKV2023, KlausSchippa2022, DNY2021} for some recent results and the references therein. The additional derivative in the nonlinearity, compared to the NLS equation~\eqref{eq:cubic-NLS}, introduces considerable difficulty in its analysis.

    A paper dealing with the well-posedness theory of the dNLS hierarchy equations is in preparation by the author.
\end{remark}

\begin{remark}\label{rem:scaling-Hs}
    Having established the structure of NLS-like equations~\eqref{eq:nls-hierarchy-general-eq}, we would like to note their associated critical regularity $s_c(j)$. This will guide us as a heuristic on our investigation of the well-posedness theory of said equations.

    In line with the scaling law of NLS, a solution $u$ of an NLS-like equation is invariant under the transformation of scale $u_\lambda(x, t) = \lambda u(\lambda x, \lambda^{2j} t)$, i.e. $u_\lambda$ is a solution of the same equation, but now with initial data $u_{0,\lambda} = \lambda u_0(\lambda x)$.

    This leads to all NLS-like equations being critical in the same space $\dot H^{-\1{2}}$ in the family of $L^2$-based Sobolev spaces as NLS itself, i.e. $s_c(j) = -\1{2}$. In fact, this is also true for the mKdV hierarchy~\cite{AGTowers}.
\end{remark}

\begin{remark}\label{rem:scaling-FL}
    As it will turn out though, no positive well-posedness results will be possible using the contraction mapping principle near the critical regularity in $L^2$ based Sobolev spaces. All our well-posedness results in the scale of spaces $H^s$ will be at fairly high regularity, supplemented by corresponding ill-posedness results to show optimality.

    Thus we turn to other scales of function spaces, in which we may keep this notion of criticality, though are able to obtain positive well-posedness results for the whole sub-critical range of spaces. In particular, we turn to Fourier-Lebesgue spaces $\hat H^s_r$ and modulation spaces $M^s_{2, p}$. Utilising these spaces for initial data has become commonplace for dispersive equations, as they allow to widen the class of functions for which well-posedness may be proven, inching further towards criticality. See~\cites{AGTowers, Klaus2023, OhWang2021, Chen2020, GrünrockHerr2008, Grünrock2005, Grünrock2004, GrünrockVega2009, DNY2021} for some examples where these spaces were successfully deployed.

    Especially for Fourier-Lebesgue spaces there is a well-defined notion of homogeneous space, in which one may ask the question of critical regularity for our NLS hierarchy equations. Using the equations' invariance, mentioned in Remark~\ref{rem:scaling-Hs}, we establish that all NLS-like equations are critical in the spaces $\hat H^s_r$ for $s_c(j, r) = -\1{r'}$ for $1 \le r \le \infty$.

    For modulation spaces though there is a much less clear notion of criticality, as the spaces are not invariant under transformations of scale, due to the isometric decomposition operators $(\Box_n)_{n \in\Z}$. Often the embedding $M^s_{2, r'} \supset \hat H^s_{r}$, for $r \le 2$, is used as guidance in the absence of criticality. Even under this notion though, we are unable to establish well-posedness with our techniques in or near the space $M^0_{2, \infty}$ (which corresponds to the critical case), paralleling results already known for the mKdV equation~\cites{OhWang2021, Chen2020}.
\end{remark}

\subsection{Generalising further}

Our later well-posedness arguments sometimes do not even rely on the particular structure of the nonlinearity in~\eqref{eq:nls-hierarchy-general-eq}, regarding the complex conjugation of factors. It is only when cubic nonlinear terms are involved, or when we are in Fourier-Lebesgue spaces, that the number of complex conjugated factors in the nonlinearity is of importance for our analysis.

We thus generalise further to an even larger class of equations.

\begin{definition}
    We call an equation a generalised (higher-order) NLS-like equation, if for $j \in \N$ there exist coefficients $c_{k,\alpha,b} \in \R$ for every $\alpha\in\N_0^{2k+1}$ with $|\alpha| = 2(j-k)$ and $b \in \set{+, -}^{2k+1}$, for $1 \le k \le j$, such that it can be written as
    \begin{equation}\label{eq:nls-like}
       i\partial_t u + (-1)^{j+1} \partial_x^{2j} u = \sum_{k=1}^j \sum_{\substack{b\in\set{\pm}^{2k+1} \\ \alpha \in \N_0^{2k+1} \\ |\alpha| = 2(j-k)}} c_{k, \alpha, b} \partial_x^{\alpha_1} v_{b_1} \prod_{\ell = 1}^{k} \partial_x^{\alpha_{2\ell}} v_{b_2} \partial_x^{\alpha_{2\ell+1}} v_{b_3},
    \end{equation}
    where each $v_{\pm}$ is to be identified with $u$ or $\overline{u}$ respectively.
\end{definition}

In short, allowing arbitrary complex conjugation in the nonlinearity of a NLS-like equation leads to the definition of generalised NLS-like equation.

\begin{remark}
    Note that the behaviour of the equations under transformations of scale does not change with this generalisation. Thus we keep the previously established critical regularity $s_c(j, r) = \1{r} - 1$ in the family of Fourier-Lebesgue spaces $\hat H^s_r$ as laid out in Remarks~\ref{rem:scaling-Hs} and~\ref{rem:scaling-FL}.
\end{remark}

\section{Statement of results}\label{sec:results}

\subsection{Prior work on higher-order equations}

Before we move on to state our main results, let us review related work. We try to be brief and thus focus on results concerning only higher-order NLS/(m)KdV equations. Giving a complete account of the history of well-posedness theory for the NLS and (m)KdV equations is beyond our scope, though we will mention some important comparative results in the two sections following the current.

Already in~\cite{Saut1979} global existence of solutions to the $j$th KdV hierarchy equation was proven, with data in high regularity Sobolev spaces $H^{k}$, $k \ge j$, using a-priori estimates provided by the structure of the hierarchy equations, together with parabolic regularisation. Positive results could be achieved in both geometries $\R$ and $\T$, though due to the techniques used, full on well-posedness was not proven, as uniqueness was left unclear.

Later, in~\cites{KPV1994-i, KPV1994-ii}, well-posedness even for a more general class of higher-order KdV like equations was proven. This was still at a comparatively high level of regularity for the initial data and was achieved using a gauge-transformation combined with linear smoothing estimates. As data spaces the weighted spaces $H^k(\R) \cap H^m(|x|^2 \d{x})$, for $k, m \in \N$ large enough, were used. It was already noted in~\cite{KPV1994-ii} that one can drop the weight, if only cubic or higher-order terms appear in the nonlinearity.

The weighted spaces (or similar alternative spaces, like Fourier-Lebesgue ones) though turn out to be indispensable in the study of the KdV hierarchy using the contraction mapping principle. This was shown in~\cite{Pilod2008}, where it was established that the higher-order equations ($j \ge 2$) of the KdV hierarchy cannot have twice continuously differentiable flow\footnote{Technically this consequence for the KdV hierarchy was noted~\cite{AGTowers}, as~\cite{Pilod2008} deals only with quadratic nonlinearities. The non-quadratic nonlinear terms though are well-behaved, so failure of smoothness of the flow carries over to the KdV hierarchy.}. In the same work it was also proven (using the contraction mapping principle), that higher dispersion KdV-like equations with quadratic nonlinearities are locally well-posed in an intersection of $H^s(\R)$, for $s > 2j + \1{4}$, with a weighted Besov space.

Most closely resembling our results and techniques is~\cite{AGTowers}, where well-posedness for the mKdV hierarchy equations (mentioned above) was derived in Fourier-Lebes\-gue spaces $\hat H^s_r(\R)$, for $s(j, r) = \frac{2j-1}{2r'}$ with $1 < r \le 2$, inching right up to the critical endpoint space $\widehat{L^1}(\R)$. These results were established using a contraction mapping argument in appropriate versions of Bourgain spaces that we use too. A partial transfer of these results to higher-order KdV type equations was possible, and appears natural due to the Miura map.

A positive result, again independent of the underlying geometry, was also proven in~\cite{KenigPilod2016}. Here the authors established well-posedness for all higher-order ($j \ge 2$) KdV hierarchy equations in $H^s$ for $s > 4j - \frac{9}{2}$. The result relies on a modified energy estimate using lower order correction terms for the energy, thus it isn't susceptible to the barrier when trying to prove well-posedness using the contraction mapping principle mentioned above.

In recent years it has also become more fashionable to utilise the underlying integrability structure of the equations in order to derive a-priori estimates. We mention~\cite{Koch2018}, where a-priori estimates for solutions of both the mKdV and NLS equations were derived, building on the earlier works~\cites{KochTataru2012, KochTataru2007} by the same authors. See also~\cite{Killip2018} for a general approach to conservation laws for integrable PDE.

Most recently published was the seminal work~\cite{KochKlaus2023}, where, relying on the integrability structure of the equations, the well-posedness of the entire KdV hierarchy in the space $H^{-1}(\R)$ was proven, as well as in $H^{j-2}(\T)$ for the $j$th equation (with dispersion order $2j+1$) in the hierarchy.

Focusing on just a single equation of the NLS hierarchy (besides NLS itself), there are only few papers dealing with low regularity well-posedness. In~\cites{Feng1994, Feng1995} the author derives global well-posedness for data in $H^s(\R)$ for $s \ge 4$ an integer\footnote{We suspect there to be a typo in the the cited works as the fourth order NLS hierarchy equation given there differs slightly from the ones given by us in Appendix~\ref{appendix:nls-hierarchy}.}. More recently in~\cite{HuoJia2005} it was proven that the fourth order equation is locally well-posed in $H^s(\R)$ for $s \ge \1{2}$ under a non-resonance condition on the coefficients of the nonlinearity. Managing to improve to local well-posedness in $H^s(\R)$ for all $s \ge \1{2}$ (without a non-resonance condition) for generalised fourth-order NLS-like equations we mention~\cite{Ikeda2021}*{Theorem~1.3}. The same paper also contains some results on fourth-order dNLS-like equations.

There also exists a rich body of literature that deals with equations that are referred to as higher-order Schrödinger equations, but differ fundamentally from what we refer to as NLS-like equations. Usually only the order of dispersion is increased or one generalises to a higher power nonlinearity $|u|^{p-1} u$, $p > 3$, compared to NLS, specifically without increasing the number of derivatives in the nonlinearity. We note the introduction of an ever increasing number of derivatives in the nonlinearity makes the analysis considerably more difficult, compared to merely upping the dispersion\footnote{Increasing just the power in the nonlinearity, at constant dispersion and if one remains in the realm of algebraic nonlinearities, also leads to more well-behaved equations. This is mirrored by our Theorem~\ref{thm:wp-hat} part (2).}; this is what we focus on as our goal is covering (at least) the equations contained within the NLS hierarchy itself.

\subsection{Main results}
As we have now established, dealing with higher-order (or higher dispersion) equations is nothing new. Though what is missing from the literature is a low-regularity well-posedness theory dealing with (generalised) higher-order NLS-like equations (i.e. that mixes higher dispersion with an appropriate number of derivative in the nonlinearity).

We hope to close this gap, at least partially, with the following theorems. For this, consider a general Cauchy problem
\begin{equation}\label{eq:general-cauchy}
    \begin{cases}
    i\partial_t u + (-1)^{j+1} \partial_x^{2j} u = F(u), \\
    u(t = 0) = u_0
    \end{cases}.
\end{equation}

\begin{theorem}\label{thm:wp-hat}
    Let $j \ge 2$ and~\eqref{eq:general-cauchy} be a higher-order NLS-like equation~\eqref{eq:nls-hierarchy-general-eq}. Then
    \begin{enumerate}
        \item if $1 < r \le 2$ and $s \ge \frac{j-1}{r'}$, the Cauchy problem~\eqref{eq:general-cauchy} for $u_0 \in \hat{H}^{s}_r(\R)$ is locally well-posed in the analytic sense,
        \item if $1 < r \le 2$, $s > -\frac{1}{r'}$ and additionally $c_{1, \alpha} = 0$ for all $\alpha \in\N_0^{3}$ (i.e. the equation contains no cubic nonlinear terms), the Cauchy problem~\eqref{eq:general-cauchy} for $u_0 \in \hat H^s_{r}(\R)$ is locally well-posed in the analytic sense.
    \end{enumerate}
\end{theorem}

For $j = 1$ this result corresponds to well-posedness of NLS in Fourier-Lebesgue spaces and is already known~\cite{Grünrock2005}. The case of periodic initial data was dealt with by different authors in~\cite{OhWang2020}.

\begin{remark}
    If we restrict ourselves to the classic Sobolev spaces $H^s(\R)$ only, we can generalise further in Theorem~\ref{thm:wp-hat} part (2). Because Proposition~\ref{prop:quintic-estimate-sobolev} allows an arbitrary number of factors in the nonlinear terms to be complex conjugates, it is also true that any generalised higher-order NLS-like equation~\eqref{eq:nls-like} that contains no cubic terms in the nonlinearity is locally well-posed in $H^s(\R)$ for $s > -\1{2}$.
\end{remark}

Besides Fourier-Lebesgue spaces we are also able to prove a general well-po\-sed\-ness result for modulation spaces $M^s_{2,p}$. In the following theorem we rely less on the distribution of complex conjugates in the nonlinearity compared with Theorem~\ref{thm:wp-hat}. The attentive reader will note, that Theorem~\ref{thm:wp-modulation} deals with any generalised higher-order NLS-like equation, so long as the cubic term corresponds to the usual $u\overline{u}u$, ignoring derivatives.

\begin{theorem}\label{thm:wp-modulation}
    Let $j \ge 2$ and ~\eqref{eq:general-cauchy} be a generalised higher-order NLS-like equation~\eqref{eq:nls-like}, where $c_{1, \alpha, b} = 0$ for all $b \not= (+,-,+)$ and $\alpha \in\N_0^3$. Then for $2 \le p < \infty$ and $s = \frac{j-1}{2}$, the Cauchy problem~\eqref{eq:general-cauchy} for $u_0 \in M^{s}_{2,p}(\R)$ is locally well-posed in the analytic sense.
\end{theorem}
Again, for $j = 1$ this result is essentially\footnote{Particularly for large $p \ge 3$ the continuity of the solution is an issue. This was pointed out in~\cite{Pattakos2019}, where at least for $1 < p < 3$ continuity of the solutions was established.} already known~\cites{Guo2017, Pattakos2019, Klaus2023} and in the periodic case from~\cite{OhWang2020}.

\begin{remark}
    For Theorem~\ref{thm:wp-modulation} a similar second part as with Theorem~\ref{thm:wp-hat} could be stated, though here seems of much less value. It would be that, if~\eqref{eq:general-cauchy} is a generalised higher-order NLS-like equation, but contains no cubic terms (i.e. $c_{1, \alpha, b} = 0$ for all $b \in \set{\pm}^3$ and $\alpha \in\N_0^3$) and $s > \frac{1}{4k} - \frac{2k+1}{2k}\frac{1}{p}$, the Cauchy problem~\eqref{eq:general-cauchy} for $u_0 \in M^s_{2,p}(\R)$ is locally well-posed in the analytic sense.
\end{remark}

\begin{remark}\label{rem:hat-space-complex-conj}
    Of note is the differing influence of the distribution of complex conjugates on the well-posedness results we state in the above theorems. To quickly recap: for the cubic terms the canonical $|u|^2u$ (ignoring  derivatives) is necessary in both Fourier-Lebesgue and modulation space settings. For the higher-order nonlinear terms though the distribution of complex conjugates can be chosen arbitrarily in the modulation space setting, whereas in Fourier-Lebesgue spaces the canonical distribution was stated necessary in Theorem~\ref{thm:wp-hat}.

    This is more restrictive than would be necessary considering our proof of Theorem~\ref{thm:wp-hat}, in particular Proposition~\ref{prop:quintic-estimate-hat}. Looking into the details, one finds that in fact also in the Fourier-Lebesgue space setting an arbitrary distribution of complex conjugates for the higher-order nonlinear terms okay. More details on the necessary changes to the arguments given in the proof of Proposition~\ref{prop:quintic-estimate-hat} are given in Remark~\ref{rem:hat-space-complex-conj-proof}.
\end{remark}

For proving our well-posedness theorems we rely on multilinear estimates in $\hat X_{s,b}^r$ (see Proposition~\ref{prop:cubic-hat-estimate} and Corollary~\ref{cor:wp-hat-estimate}) and $X_{s,b}^p$ (see Proposition~\ref{prop:cubic-mod-estimate} and Corollary~\ref{cor:wp-modulation-estimate}) spaces which combined with the contraction mapping theorem lead to local well-posedness in these spaces. Using such estimates to derive local well-posedness results is a well-known technique, so we omit the specifics. They were pioneered in~\cites{Bourgain1993-1, Bourgain1993-2} and we direct the interested reader to~\cites{Grünrock2004, AGDiss} for a self-contained review of such techniques in more contemporary notation.

Contrasting the positive results above, we are also able to prove the following ill-posedness results for initial data in Fourier-Lebesgue spaces.

\begin{theorem}\label{thm:illposed-hat}
    For $j \ge 2$, $1 < r \le 2$ and $-\frac{1}{r'} < s < \frac{j-1}{r'}$ there exists a NLS-like equation (i.e. choice of coefficients $c_{k,\alpha,b} \in \R$) such that for the Cauchy problem~\eqref{eq:general-cauchy} the flow-map $S : \hat H^s_r(\R) \times (-T, T) \to \hat H^s_r(\R)$ cannot be uniformly continuous on bounded subsets.
\end{theorem}

And with initial data in modulation spaces, the situation is similar.

\begin{theorem}\label{thm:illposedness-modulation}
    For $j \ge 2$, $2 \le p \le \infty$ and $0 \le s < \frac{j-1}{2}$ there exists a NLS-like equation (i.e. choice of coefficients $c_{k,\alpha,b} \in \R$) such that for the Cauchy problem~\eqref{eq:general-cauchy} the flow-map $S : M^s_{2,p}(\R) \times (-T, T) \to M^s_{2,p}(\R)$ cannot be uniformly continuous on bounded subsets.
\end{theorem}

Thus far we have only stated results about the well-posedness theory on the line~$\R$. Regarding the torus~$\T$, it seems no positive result is possible without additional arguments, like renormalizing the equation or moving to a weaker sense of well-posedness.

\begin{theorem}\label{thm:illposedness-torus-C3}
    For any $1 \le r \le \infty$ and $s \in\R$ the flow-map $S : \hat H^s_r(\T) \times (-T, T) \to \hat H^s_r(\T)$ of the Cauchy problem for the fourth-order NLS hierarchy equation ($j = 2$) cannot be three times continuously differentiable.
\end{theorem}

Looking at lower regularities only, we may generalise to large $j$ as well. In this case the flow becomes even more irregular:

\begin{theorem}\label{thm:illposedness-torus-C0}
    For $j \ge 2$, $1 \le r \le \infty$ and $s < j - 1$ there exists a NLS-like equation (i.e. choice of coefficients $c_{k,\alpha,b} \in \R$) such that for the Cauchy problem~\eqref{eq:general-cauchy} the flow-map $S : \hat H^s_r(\T) \times (-T, T) \to \hat H^s_r(\T)$ cannot be uniformly continuous on bounded subsets.
\end{theorem}

\subsection{Global well-posedness for the NLS hierarchy}

\begin{theorem}
    The solutions constructed in Theorem~\ref{thm:wp-hat} for initial data $u_0 \in H^s(\R)$, for $s \ge \frac{j-1}{2}$, extend globally in time.
\end{theorem}
\begin{proof}
    In order to prove this theorem we rely on the scale of conserved quantities constructed in~\cite{Koch2018}. Specifically referring to Theorem 1.1 and Corollary 1.2 therein, there exist conserved quantities, for all $s > -\1{2}$, such that the norm of a solution remains bounded if the norm of the initial data was finite under the flow of NLS and complex mKdV.

    We must argue that the same holds for all flows in the NLS hierarchy. Combined with our local result in Theorem~\ref{thm:wp-hat} this will prove the Theorem. Referencing the construction of the conserved quantities in~\cite{Koch2018}*{eqns.~(2.12) and~(2.13)}, one notices that these solely rely on the so-called transmission coefficient. This quantity arises in the scattering problem we reference in~\eqref{eq:linear-scattering}, translating between two Jost solutions of $\partial_x v = (\zeta R_0 + P)v$, see~\cite{Alberty1982-i}*{eqns.~(2.1.6) and~(2.1.25)ff.} and~\cite{Koch2018}*{eqns.~(2.5)ff.}.

    Key insight is, that the transmission coefficient is always the same, independent of which equation in the hierarchy one is interested in. This is also reflected in the fact that our choice of $Q$ (see the paragraph after~\eqref{eq:recursion-flux}) does not influence the transmission coefficient. The importance of the transmission coefficient for at least polynomial conservation laws was also already recognized in~\cite{Alberty1982-i}*{eqn.~(2.1.29)}.
\end{proof}

\subsection{Discussion}

Before moving on we would like to discuss our positive and negative results laid out in the preceding subsection.


First let us mention that our Theorem~\ref{thm:illposed-hat} establishes that, within the realm of the technique we utilise, our well-posedness result in Theorem~\ref{thm:wp-hat} is optimal. In the sense that no direct application of the contraction mapping theorem will lead to well-posedness at a lower initial regularity $s \in \R$ than stated in Theorem~\ref{thm:wp-hat}, since this would lead to the flow being analytic.

This of course does not preclude the possibility of different arguments, more heavily relying on the integrability of the hierarchy, similar to~\cite{KochKlaus2023} for the KdV hierarchy, leading to well-posedness in $H^s(\R)$ for some $s < \frac{j-1}{2}$.

We extend previous results regarding the fourth order equation: in~\cite{HuoJia2005} it was shown that the fourth order NLS equation is locally well-posed in $H^s(\R)$, for $s \ge \1{2}$, under a non-resonance condition on the coefficients in the nonlinearity. This non-resonance condition could be removed by different authors in~\cite{Ikeda2021}. We are also able to remove this condition (using different underlying function spaces) and extend the well-posedness result to all higher-order Schrödinger equations. Also, the global result in~\cites{Feng1994, Feng1995} we extend all the way down to our local result using the a priori estimates from~\cite{Koch2018}.

Not included in our well-posedness result is the critical space on our scale of function spaces $\widehat{L^1}(\R)$. Though this comes at no surprise as this space contains some nasty initial data, including the Dirac delta $\delta_0$. For this, shown in~\cite{KPV2001}*{Theorem~1.5}, it is known that no suitable notion of solution may be defined in the case of NLS. We mention the ongoing effort of extending well-posedness results (under weakened continuity assumptions on the flow) to ever greater spaces comparable to the critical $H^{-\1{2}}$ or $\widehat{L^1}$. See~\cite{BanicaVega2024} for recent developments and an overview.

In connection with Theorem~\ref{thm:illposed-hat} we would also like to mention, that our arguments do not establish ill-posedness for the \emph{actual} NLS hierarchy equations. Rather looking at a set of related equations, the first of which we give in~\eqref{eq:illposedness-equation-example}.

\quad

Next we argue our interest in the other scale of function spaces that we deal with, modulation spaces. Recall that in~\cites{Guo2017, Pattakos2019} it was shown, that NLS is locally well-posed in $M^0_{2,p}(\R)$ for $2 \le p < \infty$. This exhausts the entire subcritical range suggested by the scaling heuristic (where the critical space is $H^{-\1{2}}$) and the embedding $M^s_{2, p} \supset \hat H^s_{p'}$ for $p \ge 2$. In the Fourier-Lebesgue space setting similar results were shown in~\cite{Grünrock2005}, establishing local well-posed of NLS in $\widehat{L^r}$ for $1 < r < \infty$.

For mKdV local well-posedness was also established in almost critical Fourier-Lebesgue spaces. Specifically in~\cite{GrünrockVega2009} it was shown that mKdV is locally well-posed in $\hat H^s_r(\R)$ for $s = \1{2} - \1{2r}$ and $2 \ge r > 1$. In the modulation space setting though a gap of a quarter\footnote{It is a quarter of a derivative keeping in mind we accept the embedding $M^0_{2, \infty} \supset \widehat{L^1}$ as our guidance for criticality in the modulation space setting.} derivative between the scaling heuristic and the optimal result appears. To be exact, in~\cites{OhWang2021, Chen2020}, it was shown that mKdV is locally well-posed in $M^{\1{4}}_{2,p}(\R)$ for $2 \le p \le \infty$ and that this is optimal in the sense that the flow fails to be uniformly continuous for $s < \1{4}$.

Our Theorem~\ref{thm:wp-modulation} parallels this development for the higher-order equations, i.e. for every step to the next equation in the NLS hierarchy another half-derivative regularity of the initial data is necessary for our positive result\footnote{Note that the half-derivative increase is for stepping from one NLS hierarchy equation to the next. Looking also at the mKdV hierarchy equations in modulation spaces would be an interesting feat. The author expects that well-posedness would be achieved in modulation spaces differing by a quarter derivative from the corresponding NLS hierarchy results.}.

\quad


Moving on to results for the torus $\T$, Theorem~\ref{thm:illposedness-torus-C3} establishes that no well-posedness results may be established using the contraction mapping principle directly in the data spaces we use, at least for the next higher-order equation. This is in stark contrast to NLS, where well-posedness in $L^2(\T)$ was established in~\cite{Bourgain1993-1}.

It is reasonable to believe that the further NLS hierarchy equations are ill-posed in a similar manner and do not allow direct treatment with the contraction mapping principle. Even worse though at low regularities: here Theorem~\ref{thm:illposedness-torus-C0} establishes a milder form of ill-posedness, but in this case for an NLS-like equation of arbitrarily high (dispersion) order.

In such cases, where the proper model fails to have a well-behaved local theory, it sometimes helps to look at a renormalized/gauge-transformed version of the equation. For example, with NLS below $L^2(\T)$, considering the so-called Wick ordered NLS equation
\begin{equation}\label{eq:wick-nls}
    i\partial_t u + \partial_x^2 u = \pm \left(|u|^2 - \1{\pi}\int_\T |u|^2 \d{x}\right)u
\end{equation}
has lead to some success. See~\cite{OhSulem2012} for a review. Transitioning to a renormalised equation (via a gauge-transformation) is also a common approach with the derivative NLS equation~\cites{Grünrock2005, GrünrockHerr2008, Takaoka1999, Ozawa1996, HO1994}.

For our NLS hierarchy equations such a renormalisation might also lead to positive well-posedness results. Though it is not clear if such an approach would yield well-posedness only for the NLS hierarchy equations or for a general class, like in our results on the line.

Another viable path to approaching well-posedness on the torus (but also on the line in $H^s(\R)$ for some $s < \frac{j-1}{2}$) is to rely on the integrability of the equation, as was done for the KdV hierarchy in~\cite{KochKlaus2023}. This has the disadvantage of definitely not working for similar, but non-integrable, variants of higher-order NLS-like equations.
\section{Linear and multilinear smoothing estimates}\label{sec:linear-est}

In the following section we will be collecting and proving smoothing estimates for free solutions of equation~\eqref{eq:general-cauchy}, i.e. with $F = 0$. To shorten notation, consider solutions $u(x, t) = e^{(-1)^j t \partial_x^{2j}} u_0(x)$, $v(x, t) = e^{(-1)^j t \partial_x^{2j}} v_0(x)$ and $w(x, t) = e^{(-1)^j t \partial_x^{2j}} w_0(x)$ with initial data $u_0$, $v_0$ and $w_0$ respectively when proving estimates involving free solutions. Likewise $u$, $v$ and $w$ will refer to functions in appropriate $X_{s,b}$ space variants when talking about estimates in these spaces.

\subsection{Linear estimates}

The following linear estimates are essentially known in the literature. Our proof of Proposition~\ref{prop:quintic-estimate-sobolev} relies heavily upon them.

\begin{proposition}
    Let $b > \1{2}$, then the following inequalities hold
    \begin{enumerate}
        \item for $2 \le q \le \infty$ and $\sigma > \1{2} - \frac{2j}{q}$
        \begin{equation}\label{eq:kato1}
        \norm{u}_{L^\infty_xL^q_t} \lesssim \norm{u}_{X_{\sigma, b}}
        \end{equation}
        \item for $2 \le p \le \infty$ and $\sigma = - \frac{2j - 1}{2}(1 - \frac{2}{p})$
        \begin{equation}\label{eq:kato2}
        \norm{u}_{L^p_xL^2_t} \lesssim \norm{u}_{X_{\sigma, b}}
        \end{equation}
        \item for $4 \le p \le \infty$ and $\sigma >  \1{2} - \1{p}$
        \begin{equation}\label{eq:max-fct}
        \norm{u}_{L^p_xL^\infty_t} \lesssim \norm{u}_{X_{\sigma, b}}
        \end{equation}
    \end{enumerate}
\end{proposition}
\begin{proof}
    These linear estimates are interpolated variants of a Kato-type local smoothing estimate (for~\eqref{eq:kato1} and~\eqref{eq:kato2}) and a maximal function estimate (for~\eqref{eq:max-fct}).

    From~\cite{KPV1991}*{Theorem~4.1} we know for large frequencies
    \begin{equation}\label{eq:high-freq-kato}
        \norm{(\mathrm{id} - P_1)u}_{L^\infty_x L^2_t} \lesssim \norm{(\mathrm{id} - P_1)u_0}_{H^\sigma}, \quad\text{for $\sigma = -\frac{2j-1}{2}$.}
    \end{equation}
    For small frequencies we may use a Sobolev-embedding in the space variable, where we may ignore the loss of derivatives. So we also know~\eqref{eq:high-freq-kato} without the projector $(\mathrm{id} - P_1)$. Using the transfer principle on this bound and interpolating with the trivial bounds
    \begin{equation}\label{eq:trivial-bounds}
        \norm{u}_{L^\infty_{xt}} \lesssim \norm{u}_{X_{\1{2}+,b}} \quad\text{and}\quad \norm{u}_{L^2_{xt}} \lesssim \norm{u}_{X_{0,0}} \quad\text{for $b > \1{2}$}
    \end{equation}
    results in estimates~\eqref{eq:kato1} and~\eqref{eq:kato2} above respectively.

    For the maximal function estimate we cite~\cite{KPV1991}*{Theorem~2.5}, where
    \begin{equation*}
        \norm{(\mathrm{id} - P_1)u}_{L^4_x L^\infty_t} \lesssim \norm{(\mathrm{id} - P_1)u_0}_{H^\1{4}}
    \end{equation*}
    is established, again only for high frequencies. The same estimate was also independently found in~\cite{Sjolin1987}. Taking care of low frequencies as above and interpolating with the first bound in~\eqref{eq:trivial-bounds} results in estimate~\eqref{eq:max-fct} above.
\end{proof}

\subsection{Bilinear estimates}
Before we can go about proving any bilinear estimates, we will first define the bilinear operators which we will use. We will need two bilinear operators, the estimates for which will also differ if complex conjugation is applied to one of the factors, since our phase function is even. This is in contrast to the mKdV hierarchy in~\cite{AGTowers}, where the phase function is odd, and thus $X_{s, b}$ norms are invariant under complex conjugation.

So for $j \in\N$ and $1 \le p \le \infty$ define the pair of bilinear operators $I_{p, j}^{\pm}$ by their Fourier transform:
\begin{equation*}
\F_x I_{p,j}^\pm(f,g)(\xi) = c \int_* k_j^\pm(\xi_1, \xi_2)^\1{p} \hat{f}(\xi_1)\hat{g}(\xi_2) \d{\xi_1}.
\end{equation*}
Their symbol is given by
\begin{equation*}
k_j^\pm(\xi_1,\xi_2) = |\xi_1 \pm \xi_2|(|\xi_1|^{2j-2} + |\xi_2|^{2j-2}).
\end{equation*}

Comparing with the linear estimates above, this bilinear operator is a refinement in the sense that we now have access to the symbol of the not-quite-derivative $|\xi_1 \pm \xi_2|$. The following proposition establishes a corresponding estimate:

\begin{proposition}
    Let $j \ge 1$ and $1 \le q \le r_{1}, r_{2} \le p \le \infty$ with $\1{p} + \1{q} = \1{r_1} + \1{r_2}$. Then one finds
    \begin{equation*}
    \norm{\F_x I_{p,j}^\pm(u,v_\pm)(\xi,\cdot)}_{\widehat{L^p_t}} \lesssim (|\hat{u}_0|^{p^\prime} * |\hat{v}_0|^{p^\prime}(\xi))^\1{p^\prime}
    \end{equation*}
    and
    \begin{equation*}
    \norm{I_{p,j}^\pm(u,v_\pm)}_{\widehat{L^q_x}\widehat{L^p_t}} \lesssim \norm{u_0}_{\widehat{L^{r_1}_x}}\norm{v_0}_{\widehat{L^{r_2}_x}},
    \end{equation*}
    where $v_+ = \overline{v}$ and $v_- = v$.
\end{proposition}
\begin{proof}
    We will only write down the details for the $+$-case, the proof of the $-$-case is similar and we omit the details. Let us begin by calculating the Fourier transform only in the space-variable:
    \begin{equation*}
    \F_x I_{p,j}^+(u,\overline{v})(\xi,t) = \int_* k_j^+(\xi_1,\xi_2)^\1{p} e^{it(\xi_1^{2j} - \xi_2^{2j})} \hat{u}_0(\xi_1) \hat{\overline{v}}_0(\xi_2) \d\xi_1.
    \end{equation*}
    And now for the complete Fourier transform, substituting $x = \xi_1 - \frac{\xi}{2}$
    \begin{align}
    \F &I_{p,j}^+(u,\overline{v})(\xi,t) = \int_* k_j^+(\xi_1,\xi_2)^\1{p} \delta(\tau - \xi_1^{2j} + \xi_2^{2j}) \hat{u}_0(\xi_1) \hat{\overline{v}}_0(\xi_2) \d\xi_1\\
    &= \int_* k_j^+\left(\frac{\xi}{2} + x,\frac{\xi}{2} - x\right)^\1{p} \delta(\tau - g(x)) \hat{u}_0\left(\frac{\xi}{2} + x\right) \hat{\overline{v}}_0\left(\frac{\xi}{2} - x\right) \d{x}\\
    &= \int_* \left(\sum_{n} \frac{\delta(x - x_n)}{|g^\prime(x_n)|}\right) k_j^+\left(\frac{\xi}{2} + x,\frac{\xi}{2} - x\right)^\1{p} \hat{u}_0\left(\frac{\xi}{2} + x\right) \hat{\overline{v}}_0\left(\frac{\xi}{2} - x\right)\d{x},
    \label{eq:ft-bilin}
    \end{align}
    where the sum $\sum_{n}$ is over the simple solutions of the equation $\tau - g(x) = 0$ involving the function
    \begin{align*}
    g(x) &= \left(\frac{\xi}{2} + x\right)^{2j} - \left(\frac{\xi}{2} - x\right)^{2j} = \sum_{k=0}^{2j} \binom{2j}{k} \left(\frac{\xi}{2}\right)^{2j-k} (x^k - (-x)^k)\\
    &= 2\sum_{l=1}^j \binom{2j}{2l-1} \left(\frac{\xi}{2}\right)^{2(j-l) + 1} x^{2l-1} = \xi x \sum_{l=1}^j \binom{2j}{2l-1} \left(\frac{\xi}{2}\right)^{2(j-l)} x^{2(l-1)}.
    \end{align*}
    By our choice in the definition of the symbol of our bilinear operator we have the following lower bound on the absolute value of the derivative of $g(x)$
    \begin{equation*}
    |\partial_xg(x)| \sim |\xi| \sum_{l=1}^j \binom{2j}{2l-1} \left(\frac{\xi}{2}\right)^{2(j-l)} x^{2(l-1)} \gtrsim k^+_j\left(\frac{\xi}{2} + x,\frac{\xi}{2} - x\right).
    \end{equation*}
    Now $g(x)$, as a sum of monotone functions, only admits a single (real) solution of $\tau -g(x) = 0$. Calling this solution $y \in \R$ we can bound, except on a $\xi$ set of measure zero
    \begin{equation*}
    \eqref{eq:ft-bilin} \lesssim k_j^+\left(\frac{\xi}{2} + y,\frac{\xi}{2} - y\right)^{-\1{p^\prime}} \hat{u}_0\left(\frac{\xi}{2} + y\right) \hat{\overline{v}}_0\left(\frac{\xi}{2} - y\right).
    \end{equation*}
    In order to now calculate the $L^{p^\prime}_\tau$-norm of this expression we substitute the measure $\d{\tau} = |g^\prime(y)|\d{y}$, since we have $\tau = g(y)$, which causes the symbol of the operator to disappear and we arrive at
    \begin{equation*}
    \norm{\F_x I_{p,j}^+(u,\overline{v})(\xi,\cdot)}_{\widehat{L^p_t}} = \int_\R \left|\hat{u}_0\left(\frac{\xi}{2} + y\right) \hat{v}_0\left(\frac{\xi}{2} - y\right)\right|^{p^\prime} \d{y}
    = |\hat{u}_0|^{p^\prime}*|\hat{v}_0|^{p^\prime}(\xi).
    \end{equation*}
    This proves our first claim. In order to now extend this to an $\widehat{L^q_x}\widehat{L^p_t}$ result we make use of Young's convolution inequality. For this choose $\rho^\prime = \frac{q^\prime}{p^\prime}$, $\rho_{k} = \frac{r_{k}^\prime}{p^\prime}$ for $k\in\set{1,2}$, so that $\1{\rho} = \1{\rho_1} + \1{\rho_2}$. Then
    \begin{align*}
    \norm{I_{p,j}^+(u,\overline{v})}_{\widehat{L^q_x}\widehat{L^p_t}} &\lesssim \paren{\int \left||\hat{u}_0|^{p^\prime}*|\hat{v}_0|^{p^\prime}(\xi)\right|^\frac{q^\prime}{p^\prime} \d\xi}^\1{q^\prime} = \norm{|\hat{u}_0|^{p^\prime}*|\hat{v}_0|^{p^\prime}}_{L^{\rho^\prime}_\xi}\\
    &\lesssim \left[ \norm{|\hat{u}_0|^{p^\prime}}_{L^{\rho_1^\prime}_\xi} \norm{|\hat{v}_0|^{p^\prime}}_{L^{\rho_2^\prime}_\xi} \right]^\1{p^\prime} = \norm{u_0}_{\widehat{L^{r_1}_x}} \norm{v_0}_{\widehat{L^{r_2}_x}}
    \end{align*}
    as claimed and the proof is complete.
\end{proof}
Using the transfer principle for $\hat X_{s,b}^r$ spaces mentioned in Section~\ref{sec:notation} we may now conclude:
\begin{corollary}
    Let $1 \le q \le r_{1}, r_{2} \le p < \infty$ and $b_i > \1{r_i}$. Then we have
    \begin{equation}\label{eq:bilin-est}
    \norm{I_{p,j}^\pm(u,v_\pm)}_{\widehat{L^q_x}\widehat{L^p_t}} \lesssim \norm{u}_{\hat X_{0,b_1}^{r_1}} \norm{v}_{\hat X_{0,b_2}^{r_2}},
    \end{equation}
    where $v_+ = \overline{v}$ and $v_- = v$.
\end{corollary}

Throughout dealing with the cubic terms we will also make heavy use of inequalities that can be interpreted as the duals of those in~\eqref{eq:bilin-est}. For this view the bilinear operators as maps
\begin{equation*}
u \mapsto I^\pm_{p,j}(u, v_\pm), \qquad \hat X_{0,b_1}^{r_1} \to \widehat{L^q_x}\widehat{L^p_t}
\end{equation*}
i.e. as a multiplication with $v_\pm$ with operator norm $\lesssim \norm{v}_{X_{0,b_2}^{r_2}}$. By duality we also have the continuity, except in the endpoint case, of the map
\begin{equation*}
w \mapsto I^{\pm,*}_{p,j}(w, v_{\mp}), \qquad \widehat{L^{q^\prime}_x}\widehat{L^{p^\prime}_t} \to \hat X_{0,-b_1}^{r_1^\prime}
\end{equation*}
with the same upper bound for the operator norm. Note how we now multiply with $v_\mp$ instead of $v_\pm$. A straightforward calculation gives the associated symbols of the operators $I^{\pm,*}_{p,j}$ as
\begin{equation*}
k_j^{+,*}(\xi_1, \xi_2) = |\xi_1|(|\xi_1|^{2j-2} + |\xi_2|^{2j-2}), \quad k_j^{-,*}(\xi_1, \xi_2) = |\xi_1 + 2\xi_2|(|\xi_1|^{2j-2} + |\xi_2|^{2j-2}).
\end{equation*}
We collect the new estimates in the following
\begin{corollary}
    Let $1 < q \le r_{1},r_{2} \le p < \infty$ with $\1{p} + \1{q} = \1{r_1} + \1{r_2}$ and $b_i > \1{r_i}$. Then the estimate
    \begin{equation}\label{eq:dual-bilin-free-params}
    \norm{I^{\pm,*}_{p,j}(u,v_\mp)}_{\hat X_{0,-b_1}^{r_1^\prime}} \lesssim \norm{u}_{\widehat{L^{q^\prime}_x}\widehat{L^{p^\prime}_t}} \norm{v}_{\hat X_{0,b_2}^{r_2}}
    \end{equation}
    holds. If alternatively $0 \le \1{\rho^\prime} \le \1{r^\prime}$ and $\beta < -\1{\rho^\prime}$ we have
    \begin{equation}\label{eq:bilin-dual}
    \norm{I^{\pm,*}_{\rho^\prime,j}(u,v_\mp)}_{\hat X_{0,\beta}^r} \lesssim \norm{u}_{\widehat{L^r_{xt}}} \norm{v}_{\hat X_{0,-\beta}^{\rho^\prime}}.
    \end{equation}
    In both cases $v_+ = \overline{v}$ and $v_- = v$.
\end{corollary}
\begin{proof}
    The first estimate follows from above arguments, for the second inequality we first mention the endpoint of Young's convolution inequality
    \begin{equation*}
    \norm{uv_\mp}_{\widehat{L^r_{xt}}} \lesssim \norm{u}_{\widehat{L^r_{xt}}} \norm{v}_{\widehat{L^\infty_{xt}}}.
    \end{equation*}
    which we will use in the form
    \begin{equation}\label{eq:endpoint-stein1}
    \norm{I^{\pm,*}_{\infty,j}(u, v_\mp)}_{\hat X_{0,0}^r} \lesssim \norm{u}_{\widehat{L^r_{xt}}} \norm{v}_{\hat X_{0,0}^\infty}.
    \end{equation}
    Setting $q = r_1 = r_2 = p = r'$ in~\eqref{eq:dual-bilin-free-params} results in
    \begin{equation}\label{eq:endpoint-stein2}
    \norm{I^{\pm,*}_{r',j}(u, v_\mp)}_{\hat X_{0,-b}^r} \lesssim \norm{u}_{\widehat{L^r_{xt}}} \norm{v}_{\hat X_{0,b}^{r'}}\quad\text{for $b > \1{r'}$.}
    \end{equation}
    Now applying Stein's interpolation theorem between~\eqref{eq:endpoint-stein1} and~\eqref{eq:endpoint-stein2} results in the desired bound~\eqref{eq:bilin-dual}.
\end{proof}

\subsection{Fefferman-Stein estimate}
For later interpolation arguments we need a generalization of the Fefferman-Stein~\cite{Fefferman1970} inequality for higher-order phase functions.
\begin{proposition}
    Let $4 < q < \infty$ and $\1{r} = \1{2} + \1{q}$. For $\sigma = \frac{j-1}{2}$ one has
    \begin{equation*}
    \norm{I^\sigma u}_{L^4_t L^q_x} \lesssim \norm{u_0}_{\widehat{L^r_x}}.
    \end{equation*}
\end{proposition}
\begin{proof}
    We at first assume, that $\hat{u}_0(\xi) = \chi_{(0,\infty)}(\xi)\hat{u}_0(\xi)$. Furthermore let $v = I^\sigma u$, then
    \begin{equation*}
    \norm{I^\sigma u}_{L^4_t L^q_x}^4 = \norm{|v|^2}_{L^2_t L^{\frac{q}{2}}_x}^2 \lesssim \norm{I^\eps |v|^2}_{L^2_{xt}}^2 = \norm{\F I^\eps |v|^2}_{L^2_{xt}}^2,
    \end{equation*}
    where we have $\eps = \1{2} - \frac{2}{q}$. Calculating the Fourier transform and substituting $x = \xi_1 - \frac{\xi}{2}$ we get
    \begin{equation}\label{eq:fs-fourier-trans}
    \F (I^\eps v\overline{v})(\xi, \tau) \sim \int_\R |\xi|^\eps \delta(g(x) - \tau) \hat{u}_0\left(\frac{\xi}{2} + x\right)\hat{\overline{u}}_0\left(\frac{\xi}{2} - x\right) \d{x}.
    \end{equation}
    In order to rid ourselves of the Dirac delta present in the integral we derive a lower bound on the derivative of its argument:
    \begin{align}
    \nonumber g(x) = \xi_1^{2j} - \xi_2^{2j} = \sum_{l=1}^{j} \binom{2j}{2l-1} \left(\frac{\xi}{2}\right)^{2(j-l)+1} x^{2l-1}\\
    |g^\prime(y)| \sim |\xi| \sum_{l=1}^{j} (2l-1) \binom{2j}{2l-1} \left(\frac{\xi}{2}\right)^{2(j-l)} x^{2l} \gtrsim |\xi|y^{2(j-1)} \label{eq:dirac-fct-est}
    \end{align}
    In~\eqref{eq:dirac-fct-est} $y$ refers to the single real solution that $g(x) - \tau = 0$ admits, as a sum of monotone functions. With this we can simplify~\eqref{eq:fs-fourier-trans} to
    \begin{equation*}
    \F (I^\eps v\overline{v})(\xi, \tau) \lesssim |\xi|^{\eps-\frac{1}{2}}\frac{ y^{-(j-1)}}{\sqrt{|g^\prime(y)|}} \hat{u}_0\left(\frac{\xi}{2} + y\right)\hat{\overline{u}}_0\left(\frac{\xi}{2} - y\right).
    \end{equation*}
    Thanks to our assumed condition on the support of $u_0$ we only have a contribution if $\frac{\xi}{2} + y \ge 0$ and $-\frac{\xi}{2} + y \ge 0$ which allows us to write $2y = (\frac{\xi}{2} + y) + (-\frac{\xi}{2} + y) = |\frac{\xi}{2} + y| + |\frac{\xi}{2} - y|$. Thus we control the arguments of $\hat{u}_0$ and $\hat{\overline{u}}_0$ and with that the derivatives on these terms via $y$.
    \begin{equation*}
    \lesssim \frac{|\xi|^{\eps - \1{2}}}{\sqrt{|g^\prime(y)|}}\F_x(I^{-\frac{j-1}{2}}u_0)\left(\frac{\xi}{2} + y\right)(\F_x I^{-\frac{j-1}{2}}\overline{u}_0)\left(\frac{\xi}{2} - y\right)
    \end{equation*}
    Piecing the $L^2_{\xi\tau}$-norm together and substituting the measure $\d{\tau} = g(y)\d{y}$ and $z_\pm = y \pm \frac{\xi}{2}$ gives
    \begin{align*}
    \norm{\F I^\eps |v|^2}_{L^2_{xt}}^2 &\lesssim \int \frac{|\xi|^{2\eps - 1}}{|g^\prime(y)|}\left|\F_x(I^{-\frac{j-1}{2}}u_0)\left(\frac{\xi}{2} + y\right)(\F_x I^{-\frac{j-1}{2}}\overline{u}_0)\left(\frac{\xi}{2} - y\right)\right|^2 \d\xi\d\tau\\
    &\lesssim \int |z_+ - z_-|^{2\eps - 1} |\hat{u_0}(z_+)\hat{u}_0(z_-)|^2 \d{z_+}\d{z_-}.
    \end{align*}
    An application of the Hardy-Littlewood-Sobolev inequality requires $0 < 1 - 2\eps < 1$ and $\frac{4}{r^\prime} + 1 - 2\eps = 2$, which is equivalent to $4 < q < \infty$ and $\1{r} = \1{2} + \1{q}$. So HLS gives us the desired upper bound. The support condition on $\hat{u}_0$ can be lifted by noting that both norms on the left and right hand side of the inequality are invariant with respect to complex conjugation.
\end{proof}

Interpolating the above proposition with the endpoint of the Riemann-Lebesgue lemma $\norm{u}_{L^\infty_{xt}} \lesssim \norm{u}_{\widehat{L^\infty_{xt}}}$ gives
\begin{corollary}
    Let $\1{r} = \frac{2}{p} + \1{q}$, $0 < \1{q} < \1{4}$ and $0 \le \1{p} \le \1{4}$. Then one finds that
    \begin{equation*}
    \norm{I^{\frac{2(j-1)}{p}} u}_{L^p_t L^q_x} \lesssim \norm{u_0}_{\widehat{L^r_x}}.
    \end{equation*}
\end{corollary}
The diagonal case $p = q = 3r$ is of special interest and the only one we will make use of. Using the transfer principle we have the estimate
\begin{equation}\label{eq:fs-est}
\norm{I^{\frac{2(j-1)}{3r}} u}_{L^{3r}_{xt}} \lesssim \norm{u}_{\hat X_{0,b}^r}
\end{equation}
as long as $b > \1{r}$ and $0 \le \1{r} < \frac{3}{4}$.

\subsection{Trilinear estimates}

Particularly in the realm of $r \ll 2$ we rely on a trilinear refinement of a Strichartz type estimate in order to derive our local well-posedness result. Specifically we rely on it in proving the trilinear estimates leading to Theorems~\ref{thm:wp-hat} and~\ref{thm:wp-modulation}. Though in contrast to the mKdV hierarchy, we may prove our trilinear estimate in a more general setting, not relying on a specific frequency constellation; see~\cite{AGTowers}*{Section~3.2}. This parallels the $j=1$ case, see for example~\cite{Grünrock2005}.

\begin{proposition}
    Let $1 < p_1 < p < p_0 < \infty$, $p < p_0^\prime$, $\frac{3}{p} = \1{p_0} + \frac{2}{p_1}$ and $\frac{2}{p_1} < 1 + \1{p}$. Then we have the estimate
    \begin{equation}\label{eq:simple-trilin-est}
    \norm{u v \overline{w}}_{\widehat{L^p_{xt}}} \lesssim \norm{u_0}_{\widehat{L^{p_0}_x}} \norm{I^{-\frac{j-1}{p}} v_0}_{\widehat{L^{p_1}_x}}\norm{I^{-\frac{j-1}{p}} w_0}_{\widehat{L^{p_1}_x}}.
    \end{equation}
\end{proposition}
\begin{proof}
    We begin by taking the Fourier transform in both space- and time-variable of the product $uv\overline{w}$ and substituting $\xi_{2,3} = \frac{\xi - \xi_1}{2} \pm x$
    \begin{align*}
    \F (u v \overline{w})(\xi, \tau) \sim \int_* \delta(g(\xi_1; x)-\tau) \hat{u}_0(\xi_1)\hat{v}_0\left(\frac{\xi - \xi_1}{2} + x\right)\hat{\overline w}_0\left(\frac{\xi - \xi_1}{2} - x\right) \d\xi_1\d\xi_2,
    \end{align*}
    where in the argument of the Dirac delta
    \begin{align*}
    g(\xi_1; x) &= \xi_1^{2j} + \xi_2^{2j} - \xi_3^{2j}
    = \xi_1^{2j} + \sum_{k=0}^{2j} \binom{2j}{k} \left(\frac{\xi - \xi_1}{2}\right)^{2j-k} (x^k - (-x)^k)\\&= \xi_1^{2j} + (\xi-\xi_1)\sum_{l=1}^j \binom{2j}{2l-1} \left(\frac{\xi - \xi_1}{2}\right)^{2(j-l)} x^{2l-1}.
    \end{align*}
    As a sum of monotone functions $g(\xi_1; x)$ only admits a single (real) solution with respect to $x$ of $g(x)-\tau=0$, which we will call $y \in \R$. We can bound the derivative of $g$ from below at this root by
    \begin{align*}
    |g^\prime(\xi_1; y)| &= |\xi - \xi_1| \sum_{l=1}^j (2l-1)\binom{2j}{2l-1} \left(\frac{\xi-\xi_1}{2}\right)^{2(j-l)}y^{2(l-1)} \\&\gtrsim |\xi - \xi_1|(|\xi - \xi_1|^{2(j-1)} + y^{2(j-1)}).
    \end{align*}
    Having estimated $|g'(\xi_1, y)|$ we may move back to proving our trilinear estimate. An application of Hölder's inequality splits the integral into two parts:
    \begin{align}
    &\F (u v \overline{w})(\xi, \tau) = \int \frac{ \hat{u}_0(\xi_1) \hat{v}_0(\frac{\xi - \xi_1}{2} + y)\hat{\overline w}_0(\frac{\xi - \xi_1}{2} - y) }{|g^\prime(\xi_1; y)|}\d\xi_1 \\
   \le &\paren{\int \frac{|\hat{u}_0(\xi_1)|^p \d\xi_1}{|\xi-\xi_1|^{(1-\theta)p}}}^\1{p} \paren{\int  \frac{|\hat{v}_0(\frac{\xi - \xi_1}{2} + y)\hat{\overline w}_0(\frac{\xi - \xi_1}{2} - y)|^{p^\prime}|\xi-\xi_1|^{p^\prime}}{|\xi-\xi_1|^{\theta p^\prime}|g^\prime(\xi_1, y)|^{p^\prime}}\d\xi_1}^\1{p^\prime}. \label{eq:trilin-fourier-trans}
    \end{align}
    To estimate the first factor in~\eqref{eq:trilin-fourier-trans} we use the weak Young inequality to deal with the $L^{p^\prime}_\xi$-norm
    \begin{equation*}
    \norm{|\hat{u}_0|^p * |\cdot|^{(\theta - 1)p}}_{L_\xi^{\frac{p^\prime}{p}}}^\1{p} \lesssim \paren{\norm{|\hat{u}_0|^p}_{L_\xi^{\frac{p_0^\prime}{p}}} \norm{|\cdot|^{(\theta - 1)p}}_{L_\xi^{\1{(\theta - 1)p},\infty}}}^\1{p} \lesssim \norm{u_0}_{\widehat{L^{p_0}_x}}.
    \end{equation*}
    Its application calls for
    \begin{equation*}
    0 < (1-\theta)p < 1, \quad 1 < \frac{p_0^\prime}{p} < \1{1 - (1-\theta)p}, \quad \theta = \1{p_0^\prime}
    \end{equation*}
    which are all fulfilled thanks to our requirements for the Hölder exponents.

    Moving on to the second factor in~\eqref{eq:trilin-fourier-trans}, where we rely on our bound on the derivative~$|g'(\xi_1; y)| \gtrsim |\xi - \xi_1|(|\xi - \xi_1|^{2(j-1)} + y^{2(j-1)})$, we may estimate
    \begin{align}
    &\paren{\int  \frac{|\hat{v}_0(\frac{\xi - \xi_1}{2} + y)\hat{\overline w}_0(\frac{\xi - \xi_1}{2} - y)|^{p^\prime}|\xi-\xi_1|^{p^\prime}}{|\xi-\xi_1|^{\theta p^\prime}|g^\prime(\xi_1, y)|^{p^\prime}}\d\xi_1}^\1{p^\prime}\\
    \lesssim &\paren{\int \frac{|(\F_x I^{-\frac{j-1}{p}}v_0)(\frac{\xi - \xi_1}{2} + y)(\F_x I^{-\frac{j-1}{p}}w_0)(\frac{\xi - \xi_1}{2} - y)|^{p^\prime} \d\xi_1}{|\xi-\xi_1|^{\theta p^\prime - 1}|g^\prime(\xi_1, y)|}}^\1{p^\prime}. \label{eq:trilin-middle-est}
    \end{align}
    Now taking the $L^{p^\prime}_\tau$-norm of the preceding line and then substituting both the measure $\d\tau = g^\prime(\xi_1; y)\d{y}$ and $z_\pm = \frac{\xi - \xi_1}{2} \pm y$ we arrive at
    \begin{align}
    \paren{\int \frac{|(\F_x I^{-\frac{j-1}{p}}v_0)(z_+)(\F_x I^{-\frac{j-1}{p}}w_0)(z_-)|^{p^\prime} \d{z_+}\d{z_-}}{|z_+ + z_-|^{\theta p^\prime - 1}}}^\1{p^\prime}\\
    \lesssim \quad\norm{I^{-\frac{j-1}{p}} v_0}_{\widehat{L^{p_1}_x}} \norm{I^{-\frac{j-1}{p}} w_0}_{\widehat{L^{p_1}_x}},
    \end{align}
    where we used the Hardy-Littlewood-Sobolev inequality, noting that $\theta = \frac{3}{p^\prime} - \frac{2}{p_1^\prime} \in (0,1)$ by our conditions on the Hölder exponents and thus that $\theta p^\prime - 1 \in (0,1)$, $\frac{2}{p_1^\prime} + \theta p^\prime - 1 = 2$ and $p_1^\prime > 1$.
    This concludes the proof of the trilinear estimate.
\end{proof}

In order for this trilinear estimate to actually be useful (we want the same $\widehat{L^r_x}$-norm on all factors) we must interpolate this estimate with the Fefferman-Stein inequality from the previous subsection.

\begin{corollary}\label{cor:interpolated-trilinear}
    Let $1 < r \le 2$, then there exist $s_{0},s_{1} \ge 0$ such that $s_0 + 2s_1 = \frac{2(j-1)}{r}$ and
    \begin{equation*}
    \norm{u v \overline{w}}_{\widehat{L^r_{xt}}} \lesssim \norm{I^{-s_0} u_0}_{\widehat{L^r_x}} \norm{I^{-s_1} v_0}_{\widehat{L^r_x}} \norm{I^{-s_1} w_0}_{\widehat{L^r_x}}.
    \end{equation*}
    In addition, if $b > \1{r}$ then
    \begin{equation}\label{eq:trilin-est}
    \norm{u v \overline{w}}_{\widehat{L^r_{xt}}} \lesssim \norm{I^{-s_0} u}_{\hat X^r_{0,b}} \norm{I^{-s_1} v}_{\hat X^r_{0,b}} \norm{I^{-s_1} w}_{\hat X^r_{0,b}}.
    \end{equation}
\end{corollary}
\begin{proof}
    Using Hölder's inequality we derive
    \begin{align}
    \norm{uvw}_{L^2_{xt}} &\lesssim \norm{u}_{L^{3q_0}_{xt}} \norm{v}_{L^{3q_1}_{xt}} \norm{w}_{L^{3q_1}_{xt}} \\&\lesssim \norm{I^{-\frac{2(j-1)}{3q_0}}u_0}_{\widehat{L^{q_0}_x}} \norm{I^{-\frac{2(j-1)}{3q_1}}v_0}_{\widehat{L^{q_1}_x}} \norm{I^{-\frac{2(j-1)}{3q_1}}w_0}_{\widehat{L^{q_1}_x}},
    \end{align}
    where $q_{0},q_{1} > \frac{4}{3}$ are chosen such that $\frac{1}{2} = \1{3q_0} + \frac{2}{3q_1}$. Furthermore interpolating with the trilinear estimate~\eqref{eq:simple-trilin-est} leads to the additional constraints $\1{r} = \frac{1-\theta}{p} + \frac{\theta}{2} = \frac{1-\theta}{p_0} + \frac{\theta}{q_0} = \frac{1-\theta}{p_1} + \frac{\theta}{q_1}$. The derivative gain on the factors is thus $s_0 = \frac{2(j-1\theta)}{3q_0}$ on the first and $s_1 = 2(j-1)(\frac{1-\theta}{2p} + \frac{\theta}{3q_1})$ on the other two, for a grand total of $s_0 + 2s_1 = \frac{2(j-1)}{r}$ as claimed.
\end{proof}

\begin{remark}
    It is at this point we would like to discuss the applicability of our estimates, particularly Corollary~\ref{cor:interpolated-trilinear}, to other problems only tangentially related to NLS-like equations. We refer to the recently published work~\cite{BrunLiLiuZine2023}, in which the cubic fractional Schrödinger equation (fNLS)
    \begin{equation}\label{eq:fNLS}
        i\partial_t u = I^\alpha u + |u|^2 u
    \end{equation}
    was studied on both the real line and the torus\footnote{On the torus the equation stated above~\eqref{eq:fNLS} is in fact not well-behaved at negative Sobolev regularities $s < 0$. In order to achieve positive results on the circle the equation has to be renormalised to $$ i\partial_t u = I^\alpha u + \left(|u|^2 - \1{\pi} \int_\T |u|^2 \d{x}\right) u $$ using a gauge-transformation to eliminate a certain set of resonant interactions.}. There, the local well-posedness in $H^s(\R)$ for $\frac{2-\alpha}{4} \le s < 0$ with $\alpha > 2$ and in $H^s(\T)$ for the same range of regularities was established. The local solutions could be extended globally in time for the range $\frac{2 - \alpha}{4} \le s < 0$ on the line and for $\frac{2 - \alpha}{6} \le s < 0$ on the circle.

    In~\cite{BrunLiLiuZine2023}*{Remark~1.12} the question of well-posedness of~\eqref{eq:fNLS} in Fourier-Lebesgue spaces was posed. Assuming, as is usual, the resonant interaction $\text{high}\times\text{high}\times\text{high}\to\text{low}$ is the culprit, our trilinear estimate from Corollary~\ref{cor:interpolated-trilinear} suggests that \eqref{eq:fNLS} is well-posed in $\hat H^s_r(\R)$ for $\frac{2 - \alpha}{3r} \le s$, $1 < r \le 2$ and $\frac{\alpha}{2} \in \N_{\ge 2}$. This would already cover a big chunk of the subcritical regime up to $s_c(r) = \frac{2 - \alpha r}{2r}$, where $r \to 1$.
\end{remark}


\section{Well-posedness results}\label{sec:multilinear-estimates}

Now we have all our smoothing estimates together we can deal with the necessary multilinear estimates that lead to Theorems~\ref{thm:wp-hat} and~\ref{thm:wp-modulation}. We separate out the cases dealing with Fourier-Lebesgue and modulation spaces.

For both families of spaces the cubic nonlinear terms are strictly less well behaved, so dealing with them requires separate analysis. In contrast the quintic and higher-order terms are more tame and we are thus able to prove a general multilinear estimate for these.

The latter estimates, specifically Corollaries~\ref{cor:wp-hat-estimate} and~\ref{cor:wp-modulation-estimate}, we establish by multilinear interpolation between an $X_{s,b}$ (corresponding to the case $r = 2$ or equivalently $p = 2$) and an (almost) endpoint estimate in the respective class of spaces.


\subsection{Multilinear estimates in $\hat X_{s,b}^r$ spaces}

\subsubsection{Estimates for cubic nonlinearities}

\begin{proposition}\label{prop:cubic-hat-estimate}
    Let $1 < r \le 2$, $s = \frac{j-1}{r'}$, $\alpha \in \N_0^3$ with $|\alpha| = 2(j-1)$ then there exist $b' > -\1{r'}$, $b > \1{r}$ and one has
    \begin{equation}
    \norm{\partial_x^{\alpha_1}u_1 \partial_x^{\alpha_2}\overline{u_2} \partial_x^{\alpha_3}u_3}_{\hat{X}_{s,b^\prime}^r} \lesssim \prod_{i=1}^{3} \norm{u_i}_{\hat{X}_{s,b}^r}.
    \end{equation}
\end{proposition}
\begin{proof}
    We divide the proof into different cases, depending on the size of the interacting frequencies.
    \begin{enumerate}[wide]
        \item \textbf{Low frequency case} $|\xi_{max}| \le 1$: Here, using the trivial estimate suffices, since $s \ge 0$:
        \begin{align*}
        \norm{\partial_x^{\alpha_1}u_1 \partial_x^{\alpha_2}\overline{u_2} \partial_x^{\alpha_3}u_3}_{\hat{X}_{s,b^\prime}^r} \lesssim \norm{u_1\overline{u_2}u_3}_{\widehat{L^{r}}_{xt}} \lesssim \prod_{i=1}^3 \norm{u_i}_{\widehat{L^{3r}}} \lesssim \prod_{i=1}^3 \norm{u_i}_{\hat{X}^r_{s,b}}.
        \end{align*}

        \item \textbf{Non-resonant interaction} $|\xi_{max}| \gg |\xi_{min}|$: If there is at least one \emph{small} frequency then without loss of generality we may assume that $|\xi_1 + \xi_2| \gtrsim |\xi_1|$ (otherwise swap the factors $u_1$ and $u_3$). This in turn allows us to estimate $k^+_j(\xi_1, \xi_2) \gtrsim |\xi_1|^{2j-1}$ and $k^{+,*}_j(\xi_1 + \xi_2, \xi_3) \gtrsim |\xi_1|^{2j-1}$.
        Applied to the quantity to be estimated this gives
        \begin{align*}
        \norm{\partial_x^{\alpha_1}u_1 \partial_x^{\alpha_2}\overline{u_2} \partial_x^{\alpha_3}u_3}_{\hat X_{s,b^\prime}^r}
        &\lesssim \norm{(J^{s + 2(j-1)} u_1) \overline{u_2} u_3}_{\hat X_{0,b^\prime}^r}\\
        &\lesssim \norm{I_{r,j}^+(J^{s + \frac{2j-1}{r'} - 1} u_1, \overline{u_2}) u_3}_{\hat X_{0,b^\prime}^r}\\
        &\lesssim \norm{I_{\rho',j}^{+,*} (I_{r,j}^+(J^{s + (2j-1)(\frac{1}{r'}-\frac{1}{\rho'}) - 1} u_1, \overline{u_2}), u_3)}_{\hat X_{0,b^\prime}^r},
        \intertext{where $\rho'$ is to be chosen later, according to the constraints set forth in the following. First, we want to assume $(2j-1)(\frac{1}{r'}-\frac{1}{\rho'}) - 1 \le 0$, which allows us to reshuffle the derivatives and apply estimate~\eqref{eq:bilin-dual}:}
        &\lesssim \norm{I_{\rho',j}^{+,*} (I_{r,j}^+(J^{s} u_1, \overline{u_2}), J^{(2j-1)(\frac{1}{r'}-\frac{1}{\rho'}) - 1}u_3)}_{\hat X_{0,b^\prime}^r}\\
        &\lesssim \norm{I_{r,j}^+(J^{s} u_1, \overline{u_2})}_{\widehat{L^{r}_{xt}}} \norm{J^{(2j-1)(\frac{1}{r'}-\frac{1}{\rho'}) - 1}u_3}_{\hat X_{0,-b^\prime}^{\rho'}}
        \end{align*}
        For this to hold we must have $1 < r < \infty$, $\infty \ge \rho' \ge r'$ and $b' < -\1{\rho'}$. Now for the first factor we may apply estimate~\eqref{eq:bilin-est} on the condition that $b > \1{r}$
        and for the second factor we use a Sobolev-embedding style estimate
        assuming that $b' + b > - \1{\rho'}$ and $\frac{2(j-1)}{r'} - \frac{2(j-1)}{\rho'} < s$. This is also the point where our argument breaks down for the classic cubic NLS, with $s = 0$. After choosing $\rho'$ appropriately the proof for this case is complete.

        \item \textbf{Resonant interaction} $|\xi_1| \sim |\xi_2| \sim |\xi_3| \gtrsim 1$: Now we may utilise our trilinear smoothing estimate. As is mentioned above we do not rely on a specific frequency constellation (their signs, see~\cite{AGTowers}*{Section~3.2}) for its application, so choosing $s_0, s_1 \ge 0$ so that~\eqref{eq:trilin-est} is applicable we may directly estimate
        \begin{align*}
        \norm{\partial_x^{\alpha_1}u_1 \partial_x^{\alpha_2}\overline{u_2} \partial_x^{\alpha_3}u_3}_{\hat X_{s,b^\prime}^r} &\lesssim \norm{(J^{s+s_0}u_1) (J^{s+s_1}\overline{u_2}) (J^{s+s_1}u_3)}_{\widehat{L^{r}_{xt}}}
        \lesssim \prod_{i=1}^3 \norm{u_i}_{\hat X^r_{s,b}},
        \end{align*}
        which concludes the proof.\qedhere
    \end{enumerate}
\end{proof}

\subsubsection{Estimates for quintic and higher-order nonlinearities}

The following proposition is the $X_{s,b}$ estimate we will later interpolate with, as mentioned in the beginning of this section. Because its proof does not rely on the specific number of factors that are complex conjugates it is responsible for the remark following Theorem~\ref{thm:wp-hat}.

\begin{proposition}\label{prop:quintic-estimate-sobolev}
    Let $2 \le k \le j$, $s > -\1{2}$, $\alpha \in \N_0^{2k+1}$ with $|\alpha| = 2(j-k)$. Then there exists a $b' > -\1{2}$ such that for all $b > \1{2}$ with $b' + 1 > b$ one has
    \begin{equation}
    \norm{\prod_{i=1}^{2k+1} \partial_x^{\alpha_i}u_i}_{X_{s,b^\prime}} \lesssim \prod_{i=1}^{2k+1} \norm{u_i}_{X_{s,b}}.
    \end{equation}
    Additionally for an arbitrary subset of the factors on the left hand side these may be replaced with their complex conjugates.
\end{proposition}
\begin{proof}
    Without loss of generality assume that the frequencies are sorted in descending order of magnitude i.e., $|\xi_1| \ge |\xi_2| \ge \ldots \ge |\xi_{2k+1}|$. We distinguish two cases for the magnitude of the resulting frequency $|\xi|$.
    \begin{enumerate}[wide]
        \item $|\xi| \sim |\xi_1|$. Here we can make proper use of the $-\1{2}+$ derivatives that lie on the product. First we apply the dual form of Kato's smoothing estimate~\eqref{eq:kato2} and redistribute derivatives, introducing $\delta > 0$, in order to at a later point use the maximal function estimate~\eqref{eq:max-fct}. After using Hölder's inequality, we make use of~\eqref{eq:kato2} again (this time literally). Finally we apply the maximal function estimate, where the magnitude of $\delta$ ensures we had previously gained enough derivatives:
        \begin{align*}
        \norm{\prod_{i=1}^{2k+1} \partial_x^{\alpha_i}u_i}_{X_{s,b^\prime}} &\lesssim \norm{J^{\frac{2j-1}{2}-}(J^{2(j-k) + s - \frac{2j-1}{2} + \delta +} u_1) \prod_{i=2}^{2k+1} J^{-\frac{\delta}{2k}} u_i}_{X_{0,b^\prime}}\\
        &\lesssim \norm{(J^{2(j-k)+s-\frac{2j-1}{2} + \delta +} u_1) \prod_{i=2}^{2k+1} J^{-\frac{\delta}{2k}} u_i}_{L^{1+}_x L^2_t}\\
        &\lesssim \norm{J^{\frac{2j+1}{2} - 2k + s + \delta +} u_1}_{L^\infty_x L^2_t} \prod_{i=2}^{2k+1} \norm{J^{-\frac{\delta}{2k}} u_i}_{L^{2k(1+\eps)}_x L^\infty_t}\\
        &\lesssim \norm{u_1}_{X_{s,b}} \prod_{i=2}^{2k+1} \norm{J^{\1{2} - \1{2k(1+\eps)} - \frac{\delta}{2k}} u_i}_{X_{0,b}} \lesssim \prod_{i=1}^{2k+1} \norm{u_i}_{X_{s,b}}
        \end{align*}
        This holds as long as $\delta + 1 < 2k$ and $\1{2} - \1{2k(1+\eps)} - \frac{\delta}{2k} < s = -\1{2}+$, which can be achieved by choosing $\eps > 0$ sufficiently small.

        \item $|\xi| \ll |\xi_1|$. In this case we argue there must be at least one factor that also has large frequency magnitude compared to $\xi_1$, since $|\xi|$ is small. Thus we know $|\xi_1| \sim |\xi_2|$. Though there must also be another factor with comparatively small frequency magnitude, because if all frequencies had comparable magnitude the resulting frequency $\xi$ must also be large since we have an uneven number of factors. Hence also $|\xi_1| \gg |\xi_{2k+1}|$. We now argue
        \begin{align*}
        \norm{\prod_{i=1}^{2k+1} \partial_x^{\alpha_i}u_i}_{X_{s,b^\prime}} &\lesssim \norm{(J^{j-1} u_1) (J^{-\1{2}} u_{2k+1}) (J^{j-1} u_2) \prod_{i=3}^{2k} J^{-1 + \1{4(k-1)}} u_i}_{X_{s,b^\prime}}\\
        &\lesssim \norm{I^\pm_{2,j}(J^{-\1{2}} u_1, J^{-\1{2}} u_{2k+1}) (J^{j-1} u_2) \prod_{i=3}^{2k} J^{-1 + \1{4(k-1)}} u_i}_{L^{1+}_{xt}},
        \end{align*}
        where we used the Sobolev embedding theorem and may freely make use of the bilinear operator $I^+_{2,j}$ since $|\xi_1 \pm \xi_{2k+1}| \sim |\xi_1|$. Next, setting $r = 2(k-1)(2+\eps)$, we use Hölder's inequality
        \begin{align*}
        &\lesssim \norm{I^+_{2,j}(J^{-\1{2}} u_1, J^{-\1{2}} u_{2k+1})}_{L^2_{xt}} \norm{J^{j-1} u_2}_{L^\infty_x L^{2+}_t} \prod_{i=3}^{2k} \norm{J^{-1 + \1{4(k-1)}} u_i}_{L^r_x L^\infty_t}
        \end{align*}
        For the first factor we used the bilinear estimate~\eqref{eq:bilin-est}, for the second the interpolated Kato's smoothing~\eqref{eq:kato2} and for the rest the maximal function estimate~\eqref{eq:max-fct}, in order to arrive at our desired bound.

        For the latter estimate to lead us into the correct $X_{s,b}$-space we need
        \begin{equation*}
        -1 + \1{4(k-1)} + \1{2} - \1{4} = -\1{2} + \1{4(k-1)} - \1{2(k-1)(2+\eps)} < s = -\1{2}+
        \end{equation*}
        which can be achieved by choosing $\eps > 0$ small enough.

        In both cases every factor passes through a norm that is invariant under complex conjugation, or we have the freedom to use $I^-_{2,j}$ over $I^+_{2,j}$, so fulfilling the additional claim that an arbitrary number of the factors can be complex conjugated is also dealt with.\qedhere
    \end{enumerate}
\end{proof}

Unfortunately, when transitioning to Fourier-Lebesgue spaces, one loses the freedom to choose arbitrarily the number of factors in the nonlinearity that may be complex conjugates of the solution $u$.

\begin{proposition}\label{prop:quintic-estimate-hat}
    Let $2 \le k \le j$ and $\alpha \in \N_0^{2k+1}$ with $|\alpha| = 2(j-k)$. Then there exists an $r_0 > 1$ such that for all $1 < r < r_0$ and $s > \frac{j-k}{kr^\prime}$ there exists a $b\prime > -\1{r^\prime}$ such that for all $b > \1{r}$ with $b' + 1 > b$ one has
    \begin{equation}
    \norm{\prod_{i=1}^{2k+1} \partial_x^{\alpha_i}v_i}_{\hat X_{s,b^\prime}^r} \lesssim \prod_{i=1}^{2k+1} \norm{u_i}_{\hat X_{s,b}^r},
    \end{equation}
    where exactly $k$ of $v_1,v_2,\ldots,v_{2k+1}$ are equal to the complex conjugate of $u_i$ and otherwise just equal to $u_i$.
\end{proposition}
\begin{proof}
    We assume, without loss of generality, that the magnitudes of the frequencies are sorted i.e., $|\xi_1| \ge |\xi_2| \ge \ldots |\xi_{2k+1}|$. Distinguish cases based on the number of high-frequency factors that are present in the product:
    \begin{enumerate}[wide]
        \item $|\xi_4| \gtrsim |\xi_1|$. So we have at least four high-frequency factors which is enough for us to make use of the Fefferman-Stein estimate~\eqref{eq:fs-est}. We start by choosing $r_0 > 1$ such that $s < \1{r}$. Next fix $s_1 > \1{4}(2(j-k) + s + (2k-3)(\1{r} - s))$ and $s_2 < s - \1{r} < 0$ fulfilling $4s_1 + (2k-3)s_2 = 2(j-k)+s$. Then we can estimate using the Hausdorff-Young inequality
        \begin{align*}
        \norm{\prod_{i=1}^{2k+1} \partial_x^{\alpha_i}v_i}_{\hat X_{s,b^\prime}^r} &\lesssim \norm{\prod_{i=1}^{4} J^{s_1}v_i \prod_{i=5}^{2k+1} J^{s_2}v_i}_{\widehat{L^r_{xt}}} \lesssim \norm{\prod_{i=1}^{4} J^{s_1}v_i \prod_{i=5}^{2k+1} J^{s_2}v_i}_{L^r_{xt}}\\
        &\lesssim \prod_{i=1}^{4} \norm{J^{s_1}u_i}_{L^{4r}_{xt}} \prod_{i=5}^{2k+1} \norm{J^{s_2}u_i}_{L^\infty_{xt}}
        \end{align*}
        For every factor in the second product we can now use $\norm{f}_{L^\infty_{xt}} \lesssim \norm{f}_{\widehat{L^\infty_{xt}}}$ followed by a Sobolev style embedding, where we end up with $s_2 + \1{r} - \1{\infty} +$ space- and $\1{r}+$ time-derivatives. The first four factors can be dealt with by the diagonal case of the Fefferman-Stein inequality~\eqref{eq:fs-est}. So that we end up in the correct $\hat X_{s,b}^r$-norm we need $s > s_1 + \frac{1 - 2(j-1)}{4r}$, which we can achieve for every $s > \frac{j-k}{kr^\prime}$ (by choosing $s_1$ near enough $\1{4}(2(j-k) + s + (2k-3)(\1{r} - s))$) as claimed.

        \item $|\xi| \sim |\xi_1| \gg |\xi_2|$. With only a single high-frequency factor $v_i$ we must distinguish if it is a complex conjugate or not. Without loss of generality we assume $v_1 = \overline{u_1}$ and that (since we know exactly $k$ of the factor are complex conjugates) we are dealing with a product of the form $\overline{u_1} (\prod_{i = 2}^{2k-3}v_i) u_{2k-2} u_{2k-1} u_{2k} \overline{u_{2k+1}}$ (omitting the derivatives). The arguments for the alternate cases is similar, we omit the details. Having only $|\xi_1|$ large gives us control over the symbols of $I^{-,*}_{\rho^\prime,j}$ and $I^+_{r,j}$ when applied as in
        \begin{align*}
        &\norm{\prod_{i=1}^{2k+1} \partial_x^{\alpha_i}v_i}_{\hat X_{0,b^\prime}^r}\\
        &\lesssim \norm{I^+_{r,j}(J^{2(j-k) - \frac{2j-1}{r}}\overline{u_1}, u_{2k})\overline{u_{2k+1}} \prod_{i=2}^{2k-1} v_i }_{\hat X_{0,b^\prime}^r}\\
        &\lesssim \norm{I^{-,*}_{\rho^\prime,j}(I^+_{r,j}(J^{2(j-k) - \frac{2j-1}{r} - \frac{2j-1}{\rho^\prime}}\overline{u_1}, u_{2k})\prod_{i=2}^{2k-1} v_i, \overline{u_{2k+1}}) }_{\hat X_{0,b^\prime}^r}\\
        &\lesssim \norm{I^{-,*}_{\rho^\prime,j}(I^+_{r,j}(J^{\frac{2(j-k)}{r^\prime} - \frac{2(j-1)}{\rho^\prime}}\overline{u_1}, u_{2k})\prod_{i=2}^{2k-1} J^{-\1{r} -}v_i, J^{-\1{r} + \1{\rho^\prime} -}\overline{u_{2k+1}}) }_{\hat X_{0,b^\prime}^r}
        \intertext{Now choosing $\rho \sim r$ such that $\frac{2(j-k)}{r^\prime} - \frac{2(j-1)}{\rho^\prime} \le 0$ we get for a $b^\prime < - \1{\rho^\prime}$}
        &\lesssim \norm{I^+_{r,j}(\overline{u_1}, u_{2k})\prod_{i=2}^{2k-1} J^{-\1{r} -}v_i}_{\widehat{L^r_{xt}}} \norm{J^{-\1{r} + \1{\rho^\prime} -} u_{2k+1}}_{\hat X_{0,-b^\prime}^{\rho^\prime}}\\
        &\lesssim \norm{I^+_{r,j}(\overline{u_1}, u_{2k})}_{\widehat{L^r_{xt}}} \prod_{i=2}^{2k-1} \norm{J^{-\1{r} -}v_i}_{\widehat{L^\infty_{xt}}} \norm{J^{-\1{r} + \1{\rho^\prime} -} u_{2k+1}}_{\hat X_{0,-b^\prime}^{\rho^\prime}}
        \end{align*}
        Using a the bilinear estimate~\eqref{eq:bilin-est}, a Sobolev style embedding and Young's inequality we arrive at the desired upper bound, at least in the case $s = 0$.

        \item $|\xi_1| \sim |\xi_2| \gg |\xi_3|$ or $|\xi_1| \sim |\xi_3| \gg |\xi_4|$.
        \subparagraph{subcase: $v_1 = u_1$ and $v_2 = u_2$} If there are two or three high-frequency factors we proceed similarly as to the case where there is only a single one, though parenthesizing differently with the bilinear operators. Here further cases can be made depending on if the high-frequency factors are complex conjugates or not, though these are remedied by using $I^-_{r,j}$ rather than $I^+_{r,j}$ and vice versa (dito for the dual operators). The arguments are very similar to the preceding cases, so we omit the details.

        We proved the inequality for $s = 0$  in the latter two cases, thus it also holds for every $s \ge 0$.
    \end{enumerate}
\end{proof}

\begin{remark}\label{rem:hat-space-complex-conj-proof}
    Let us discuss what influence the distribution of complex conjugates has on the estimate proven in Proposition~\ref{prop:quintic-estimate-hat}.
    In the first case, where we have `enough', that is four or more, high-frequency factors, whether the terms in the nonlinearity are complex conjugates or not is irrelevant.
    Inspecting the proof for the subsequent cases, where there are three or fewer high-frequency factors, we point out that $2k-2$ of the factors pass through a $\widehat{L^{\infty}_{xt}}$ norm and thus, if these are complex conjugates or not is irrelevant.

    Also in these cases, since $u_1$ is a high-frequency factor and $u_{2k}$ has low frequency, which of the symbols of either bilinear operators $I_{r,j}^\pm$ we gain does not matter. Hence we are not restricted in the sense that the `partner' of $u_1$ in the application of $I_{r,j}^\pm$ has a complex conjugate or not. (This is also independent of whether $u_1$ is a complex conjugate, because $I_{r,j}^\pm$ passed through a $\widehat{L^{r}_{xt}}$ norm.)

    What would remain to argue is why one also has free choice to apply either of the dual bilinear operators $I_{\rho', j}^{\pm, *}$ and hence again, that if the `partner' ($u_{2k+1}$ in the argument given in the proof) is a complex conjugate or not, is irrelevant. This is slightly more delicate and one must vary the `partner' in application of the dual bilinear operator between $u_{2k+1}$ and one of the other high-frequency factors, if the total frequency of $I^+_{r,j}(v_1, v_{2k})\prod_{i=2}^{2k-1}v_i$ (ignoring derivatives) is small. (This product having small frequency can only happen in case there are multiple (but fewer than four) high-frequency factors.) In such a case the symbol of, say, $I_{\rho', j}^{+, *}$ is small and one can thus not fully exploit the gain in derivative this operator would offer. To remedy this one can swap out $u_{2k+1}$ with one of the high-frequency factors besides $u_1$ to ensure the symbol of both bilinear operators is large again.

    We deem adding such a case by case analysis to the proof of Proposition~\ref{prop:quintic-estimate-hat} would distract from the overall argument, so we leave working out further details to the reader.
\end{remark}

Finally we may use multilinear interpolation to interpolate between the estimates in Propositions~\ref{prop:quintic-estimate-sobolev} and~\ref{prop:quintic-estimate-hat} in order to establish the corollary from which Theorem~\ref{thm:wp-hat} follows.

\begin{corollary}\label{cor:wp-hat-estimate}
    Let $2 \le k \le j$ and $\alpha \in \N_0^{2k+1}$ with $|\alpha| = 2(j-k)$. Then for $1 < r \le 2$ and $s > -\frac{1}{r'}$ there exists a $b' > -\frac{1}{r'}$ such that for all $b > \frac{1}{r}$ we have
    \begin{equation}
    \norm{\prod_{i=1}^{2k+1} \partial_x^{\alpha_i}v_i}_{\hat X_{s,b^\prime}^r} \lesssim \prod_{i=1}^{2k+1} \norm{u_i}_{\hat X_{s,b}^r},
    \end{equation}
    where exactly $k$ of $v_1, v_2, \ldots, v_{2k+1}$ are equal to the complex conjugate of $u_i$ and otherwise just equal to $u_i$.
\end{corollary}

\subsection{Multilinear estimates in $X_{s,b}^p$ spaces}

Before we dive into the proofs that will lead to Proposition~\ref{prop:cubic-mod-estimate} and Corollary~\ref{cor:wp-modulation-estimate}, which in turn imply Theorem~\ref{thm:wp-modulation}, we would like to give the reader a run down of extra conventions we will be using when dealing with estimates of frequency localised functions. As in the previous section, we will be proving our estimates separately for different frequency constellations on a case by case basis.

Let us first mention that, even though we are in modulation spaces, we will not need the added control the associated uniform frequency localisation may give us. In particular we will only rely on this additional control in the resonant case for the cubic nonlinear term. For all other cases a more common dyadic frequency decomposition will suffice, which we may sum to arrive in the correct modulation space using, for example,~\eqref{eq:bernstein}.

Furthermore, in order to save vertical space and give a more compact presentation of our estimates, we will play loose with the description of the set over which we will be summing in some cases. Implicitly it is understood that we are always summing over all dyadic frequencies $N, N_1, N_2, \ldots$ or integer frequencies $n, n_1, n_2,\ldots$ that appear in the expression we want to estimate, subject to the restrictions implied by the case we are currently estimating. An example of the suppression of information in a sum, would be the following two sums being equivalent
\begin{equation*}
    \sum_{N_1 \gtrsim N_3} \int_{\R^2} u_{N_1} \overline{u_{N_2}} u_{N_3} \overline{v_{N}} \d{x}\d{t} = \sum_{\substack{N, N_1, N_2, N_3 \ge 1 \\ N_1 \gtrsim N_3}} \int_{\R^2} P_{N_1}u_1 P_{N_2}\overline{u_2} P_{N_3}u_3 P_N\overline{v} \d{x}\d{t},
\end{equation*}
where additional we have made clear the convention mentioned in Section~\ref{sec:notation} that indices denoting frequency decomposition may suppress other indices.

We also introduce the notation $\xi_{max}$, $\xi_{min}$ and $N_{max}$, $N_{min}$ referring to the largest and smallest element of the sets of all frequencies $\set{|\xi_i| \mid 1 \le i \le 2k+1 }$ and of all dyadic frequencies $\set{N_i \mid 1 \le i \le 2k+1 }$, where $2k+1$ is the total number of factors in a nonlinear term.

One last ingredient: the following lemma will help us piece together uniform-frequency localized functions. It had previously appeared in~\cite{OhWang2021}*{eq.~(2.7)}, without proof, but we include its proof here for the reader's convenience.
\begin{lemma}\label{summation-lemma}
    Let $(a_m)_{m\in\Z}$ and $(b_n)_{n\in\Z}$ be two sequences. Then for $1 \le p \le \infty$ and every $\eps > 0$ one has
    \begin{equation*}
    \sum_{\substack{m, n \in \Z\\m \not= n}} \frac{a_m b_n}{|m-n| \JBX[n]^\eps} \lesssim_\eps \norm{a_m}_{\ell_m^p(\Z)} \norm{b_n}_{\ell_n^{p'}(\Z)}.
    \end{equation*}
\end{lemma}
\begin{proof}
    We apply Hölder's inequality and Young's convolution inequality
    \begin{gather*}
    \sum_{\substack{m, n \in \Z\\m \not= n}} \frac{a_m b_n}{|m-n| \JBX[n]^\eps} = \sum_{m \in \Z} a_m \sum_{n \in \Z} \frac{b_n}{\JBX[n]^\eps} \cdot \frac{\chi_{m\not=n}}{|m-n|}\\
    \lesssim \norm{a_m}_{\ell_m^p} \bignorm{\frac{b_\cdot}{\JBX[\cdot]^\eps} * \frac{\chi_{\cdot \not= 0}}{|\cdot|}}_{\ell^{p'}} \lesssim \norm{a_m}_{\ell_m^p} \norm{b_n\JBX[n]^{-\eps}}_{\ell_n^{q}} \norm{\chi_{n \not= 0}|n|^{-1}}_{\ell_n^r}\\
    \lesssim \norm{a_m}_{\ell_m^p} \norm{b_n}_{\ell_n^{p'}} \norm{\JBX[n]^{-\eps}}_{\ell_n^{\tilde{q}}} \norm{\chi_{n \not= 0}|n|^{-1}}_{\ell_n^r} \lesssim_\eps \norm{a_m}_{\ell_m^p} \norm{b_n}_{\ell_n^{p'}},
    \end{gather*}
    where $1 + \1{p'} = \1{q} + \1{r}$ and $\1{q} = \1{p'} + \1{\tilde{q}}$. The last inequality becomes true, if we choose $\tilde{\eps} > 0$ small enough and then set $\1{r} = \1{1 + \tilde{\eps}}$, as well as $\1{\tilde{q}} = \frac{\tilde{\eps}}{1 + \tilde{\eps}}$.
\end{proof}

\subsubsection{Estimates for cubic nonlinearities}

In the proof of the following Proposition~\ref{prop:cubic-mod-estimate} we assume $s = \frac{j-1}{2}$, though because of the inequality $\JBX[\xi] \lesssim \JBX[\xi_1]\JBX[\xi_2]\JBX[\xi_3]$ for $\xi = \xi_1 + \xi_2 + \xi_3$ the derived estimate also holds true for $s > \frac{j-1}{2}$.

\begin{proposition}\label{prop:cubic-mod-estimate}
    Let $j \ge 2$, $2 \le p < \infty$, $s = \frac{j-1}{2}$, $\alpha \in \N_0^3$ with $|\alpha| = 2(j-1)$. Then there exist $b' < 0$ and $b' + 1 > b > \1{2}$ such that one has
    \begin{equation*}
    \norm{\partial_x^{\alpha_1}u_1 \partial_x^{\alpha_2}\overline{u_2} \partial_x^{\alpha_3}u_3}_{X_{s,b^\prime}^p} \lesssim \prod_{i=1}^{3} \norm{u_i}_{X_{s,b}^p}.
    \end{equation*}
\end{proposition}
\begin{proof}
    Again, the proof is a case by case analysis of different frequency interactions. We prove the estimate in each case by duality:
    \begin{enumerate}[wide]
        \item \textbf{Low frequency case} $|N_{max}| \lesssim 1$: In this case we may deduce that the frequency of the product $N$ is also small. So we use Hölder's inequality, Sobolev embeddings and~\eqref{eq:bernstein} for the sum
        \begin{align*}
        &\sum_{N_{max}, N \lesssim 1} \int_{\R^2} \partial_x^{\alpha_1}u_{N_1} \partial_x^{\alpha_2}\overline{u_{N_2}} \partial_x^{\alpha_3}u_{N_3} N^s \overline{v_N} \d{x}\d{t}\\
        \lesssim &\sum_{N_{max}, N \lesssim 1} N_{max}^{s + 2(j-1)+1+} \norm{u_{N_1}}_{L^2_{xt}} \norm{u_{N_2}}_{L^\infty_{t}L^2_x} \norm{u_{N_3}}_{L^\infty_{t}L^2_x} \norm{v_{N}}_{L^2_{xt}}\\
        \lesssim &\sum_{N_{max}, N \lesssim 1} N_{max}^{s + 2(j-1)+1+\frac{3}{2}-\frac{3}{p}+} \norm{v_{N}}_{X_{0,-b'}^{p'}} \prod_{i=1}^3 \norm{u_{N_i}}_{X_{s,b}^p}
        \end{align*}
        This is a finite sum (remember, our dyadic frequencies are $N_i \in 2^\N$), so we may bound the final expression by our desired $\norm{v}_{X_{0,-b'}^{p'}} \prod_{i=1}^{3} \norm{u_i}_{X_{s,b}^p}$.

        \item \textbf{Non-/Semi-resonant interaction} $N_{max} \gg N_{min}$: Here there are two subcases to be dealt with, depending on which frequencies are of similar magnitude to $N_{max}$, but with opposite sign, if any. The arguments in both cases are the same (just with the roles of some of the factors interchanged), so we will only present one of the cases.

        Say we have $|\xi_{max}| = |\xi_1| \gg |\xi_3| = |\xi_{min}|$. Then either $|\xi_1 + \xi_2| \sim |\xi_1|$ or $|\xi_1 + \xi| \sim |\xi_1|$. In the former case, both $|\xi_1 + \xi_2|$ and $|\xi_3 + \xi|$ are comparable to $|\xi_{max}|$ and in the latter it is both $|\xi_1 + \xi|$ and $|\xi_2 + \xi_3|$ that are comparable. For other choices of $\xi_{max}$ and $\xi_{min}$ one may argue similarly.

        Observe the argument for the case with $|\xi_{max}| = |\xi_1| \gg |\xi_3| = |\xi_{min}|$ and $|\xi_1 + \xi_2| \sim |\xi_1|$: first we use Hölder's inequality
        \begin{align}
        &\sum_{N_1 \gg N_3} \int_{\R^2} \partial_x^{\alpha_1}u_{N_1} \partial_x^{\alpha_2}\overline{u_{N_2}} \partial_x^{\alpha_3}u_{N_3} N^s \overline{v_N} \d{x}\d{t}\\
        \label{eq:pre-bilin-interpolation} \lesssim &\sum_{N_1 \gg N_3} N_{max}^{s + 2(j-1)} \norm{u_{N_1} \overline{u_{N_2}}}_{L^2_{xt}} \norm{u_{N_3} \overline{v_{N}}}_{L^2_{xt}}
        \end{align}
        Next we would like to apply our bilinear estimate~\eqref{eq:bilin-est} with $q = p = 2$ to both terms in the $L^2$ norm. Though because we are estimating by duality simply using~\eqref{eq:bilin-est} as-is would leave us with $v_N$ in the wrong space $X_{0, b}$ for $b > \1{2}$. To remedy this we interpolate~\eqref{eq:bilin-est} with the much simpler bound
        \begin{align}
        \norm{I^+_{2,j}(u, \overline{v})}_{L^2_{xt}} &\lesssim \norm{(J^\sigma u)(J^\sigma v)}_{L^2_{xt}}\\ &\lesssim \norm{J^\sigma u}_{L^\infty_t L^2_x} \norm{J^\sigma v}_{L^2_t L^\infty_x} \lesssim \norm{u}_{X_{\sigma, \1{2} +}} \norm{v}_{X_{\sigma + \1{2} +, 0}}, \label{eq:simple-bilin}
        \end{align}
        where $\sigma = 2j-1$ and we used Hölder's inequality and Sobolev embeddings.
        Using our interpolated bound we may proceed with estimating~\eqref{eq:pre-bilin-interpolation}:
        \begin{align*}
        \lesssim &\sum_{N_1 \gg N_3} N_{max}^{s - 1} \norm{I_{2,j}^+(u_{N_1}, \overline{u_{N_2}})}_{L^2_{xt}} \norm{I_{2,j}^+(u_{N_3} \overline{v_{N}})}_{L^2_{xt}}\\
        \lesssim &\sum_{N_1 \gg N_3} N_{max}^{s - 1 +} (N_1N_2N_3)^{-s+} \norm{v_N}_{X_{0,-b'}^{p'}} \prod_{i=1}^3 N_i^{0-} \norm{u_{N_i}}_{X_{s,b}}\\
        \lesssim &\sum_{N_1 \gg N_3} N_1^{- \1{2} - \1{p}+} (N_2N_3)^{-s+\1{2}-\1{p}+} N^{0-}\norm{v_N}_{X_{0,-b'}^{p'}} \prod_{i=1}^3 N_i^{0-} \norm{u_{N_i}}_{X_{s,b}^p}.
        \end{align*}
        At this point it becomes important, that $j \not= 1$, because otherwise $s = 0$ and we wouldn't be able to sum up. For $j \ge 2$ though, one has $s \ge \1{2}$ so that $N_1^{- \1{2} - \1{p}+} (N_2N_3)^{-s+\1{2}-\1{p}+} \lesssim 1$ and we can close our argument with a final application of~\eqref{eq:bernstein}.

        \item \textbf{Resonant interaction} $N_{max} \sim N_{min}$: Here we will have to utilize the added control modulation spaces give us with the unit cube decomposition. We distinguish between the following subcases:
        \begin{enumerate}[label=\arabic*., wide]
            \item $\forall (i, j): |\xi_i + \xi_j| \gtrsim |\xi_i - \xi_j|$: This means that all frequencies have the same sign. Since we have separate control over the symbols $|\xi_i + \xi_j|$ and $|\xi_i - \xi_j|$ we may argue simpler than in~\cite{OhWang2021}. The estimate in this subcase may be proven analogously to the non-/semi-resonant case.
            \item $|\xi_1 - \xi_2| \ge |\xi_1 + \xi_2|$:
            \begin{enumerate}[label*=\arabic*., wide]
                \item $|\xi_1 + \xi_2| \lesssim 1$ and $\min(|\xi_2 + \xi_3|, |\xi_2 - \xi_3|) \lesssim 1$: Without loss of generality we will assume $|\xi_2 - \xi_3| \lesssim 1$, the other case may be argued analogously. So here we have the following frequencies for the individual factors and their product
                \begin{equation*}
                n_1 = -\ell + \mathcal{O}(1), \quad n_2 = \ell, \quad n_3 = \ell + \mathcal{O}(1), \quad n = \ell + \mathcal{O}(1)
                \end{equation*}
                for a fixed $\ell \in \Z$. We may restrict ourselves to proving the diagonal case, where $- n_1 = n_2 = n_3 = n = \ell$ hold exactly. This is because after having established the inequality for the diagonal case, the general case may be proven by switching to a different family of isometric decomposition operators $(\tilde{\Box}_n)_{n\in\Z}$ and using the inequality for the diagonal case. We omit the details.

                After using Hölder's inequality we use our trilinear estimate~\eqref{eq:trilin-est} to bound the contribution in this case:

                \begin{align*}
                &\sum_{\ell\in\Z} \JBX[\ell]^{s + 2(j-1)} \int_{\R^2} u_{-\ell} \overline{u}_{\ell} u_{\ell} \overline{v}_{\ell} \d{x}\d{t} \lesssim \sum_{\ell\in\Z} \JBX[\ell]^{s + 2(j-1)} \norm{u_{-\ell} \overline{u}_{\ell} u_{\ell}}_{L^2_{xt}} \norm{v_{-\ell}}_{L^2_{xt}}\\
                \lesssim &\sum_{\ell\in\Z} \JBX[\ell]^{s + 2(j-1)} \norm{u_{-\ell} \overline{u}_{\ell} u_{\ell}}_{L^2_{xt}} \norm{v_{-\ell}}_{L^2_{xt}} \lesssim \sum_{\ell\in\Z} \norm{u_{-\ell}}_{X_{s,b}}^3 \norm{v_{-\ell}}_{X_{0,-b'}}
                \end{align*}
                Using the trivial embeddings $\ell^2 \supset \ell^{p'}$ and $\ell^{3p} \supset \ell^p$ we arrive at our desired bound:
                \begin{equation*}
                \lesssim \norm{v_{-\ell}}_{X_{0,-b'}^{p'}} \prod_{i=1}^3 \norm{u_i}_{X_{s,b}^{3p}} \lesssim \norm{v}_{X_{0,-b'}^{p'}} \prod_{i=1}^3 \norm{u_i}_{X_{s,b}^{p}}
                \end{equation*}

                \item $|\xi_1 + \xi_2| \lesssim 1$ and $|\xi_2 \pm \xi_3| \gg 1$: In this case we have the following frequencies:
                \begin{equation*}
                n_1 = -\ell + \mathcal{O}(1), \quad n_2 = \ell, \quad n_3 = m + \mathcal{O}(1), \quad n = m + \mathcal{O}(1).
                \end{equation*}
                for fixed $\ell, m \in\Z$. Also we may note, that $|\ell \pm m| \gtrsim 1$, as well as $|m| \sim |\ell|$ because we are in a resonant case. By symmetry we may additionally assume $|m + \ell| \ge |m - \ell|$. Again it suffices to deal with the diagonal case, where $- n_1 = n_2 = \ell$ and $n_3 = n = m$ exactly. As usual we begin with an application of Hölder's inequality:
                \begin{equation*}
                \sum_{\ell,m\in\Z} \JBX[m]^{s + 2(j-1)} \int_{\R^2} u_{-\ell} \overline{u}_{\ell} u_{m} \overline{v}_{m} \d{x}\d{t}
                \lesssim \sum_{\ell,m\in\Z} \JBX[m]^{s + 2(j-1)} \norm{u_{m} \overline{u}_{\ell}}_{L^2_{xt}} \norm{u_{-\ell} \overline{v}_{m}}_{L^2_{xt}}
                \end{equation*}
                Being left in a similar situation to~\eqref{eq:pre-bilin-interpolation}, we argue with the same interpolated inequality (between~\eqref{eq:bilin-est} and~\eqref{eq:simple-bilin}) to arrive at
                \begin{align*}
                \lesssim &\sum_{\ell,m\in\Z} \frac{\JBX[m]^{s + 2(j-1)} \norm{u_m}_{X_{0,b}} \norm{u_{-\ell}}_{X_{0,b}} \norm{u_{-\ell}}_{X_{0,b}} \norm{v_{-m}}_{X_{0,-b'}} }{\JBX[m]^{2j-2-}\sqrt{|m-\ell| \cdot |\ell + m|}}.
                \end{align*}
                Here we may use $|m+\ell| \ge |m-\ell|$ and then apply~\Cref{summation-lemma}, which is again reliant on the fact $s > 0$:
                \begin{align*}
                \lesssim &\sum_{\ell,m\in\Z} \1{|m - \ell| \JBX[\ell]^{2s-}} \norm{u_m}_{X_{s,b}} \norm{u_{-\ell}}_{X_{s,b}} \norm{u_{-\ell}}_{X_{s,b}} \norm{v_{-m}}_{X_{0,-b'}}\\
                \lesssim & \bignorm{\norm{u_{-\ell}}_{X_{s,b}} \norm{u_{-\ell}}_{X_{s,b}}}_{\ell_\ell^{\frac{p}{2}}} \cdot \bignorm{\norm{u_{m}}_{X_{s,b}} \norm{v_{-m}}_{X_{0,-b'}}}_{\ell_m^{\frac{p}{p-2}}} 
                \end{align*}
                Finally for the first factor we utilise Hölder's inequality, for the second we send $\norm{u_m}_{X_{s,b}}$ to $X_{s,b}^\infty$ and then use the embeddings $\ell^\infty \supset \ell^p$ and $\ell^{\frac{p}{p-2}} \supset \ell^{p'}$ to arrive at our desired bound for this case.

                \item $\forall i \not= j: |\xi_i \pm \xi_j| \gg 1$: This subcase starts similarly to the preceding one, where we first apply Hölder's inequality and then our interpolated bilinear estimate (between~\eqref{eq:bilin-est} and~\eqref{eq:simple-bilin}) in order to place $v_n$ in the correct space $X_{0, -b'}$.
                \begin{align*}
                &\sum_{n_1 + n_2 + n_3 = n} |n|^{s + 2(j-1)} \int_{\R^2} u_{n_1} \overline{u}_{n_2} u_{n_3} \overline{v}_n \d{x}\d{t}\\
                \lesssim &\sum_{n_1 + n_2 + n_3 = n} |n|^{s + 2(j-1)} \norm{u_{n_1} \overline{u}_{n_2}}_{L^2_{xt}} \norm{u_{n_3} \overline{v}_n}_{L^2_{xt}}\\
                \lesssim &\sum_{n_1 + n_2 + n_3 = n} \frac{|n|^{s+} \norm{u_{n_1}}_{X_{0,b}} \norm{u_{-n_2}}_{X_{0,b}} \norm{u_{n_3}}_{X_{0,b}} \norm{v_{-n}}_{X_{0,-b'}} }{\sqrt{|n_1 + n_2| \cdot |n_3 + n|}}
                \end{align*}
                Now at least one of $|n_1 + n_2|$ or $|n_3 + n|$ is comparable to $|n|$, so assuming without loss, that $|n_3 + n| \sim |n|$ we may split the factor $|n|^{s+}$ and apply Hölder's inequality:
                \begin{align*}
                \lesssim &\sum_{n_1 + n_2 + n_3 = n} \frac{\norm{u_{n_1}}_{X_{s,b}} \norm{u_{-n_2}}_{X_{s,b}}}{\sqrt{|n_1 + n_2|} |n|^{\1{2}-}} \frac{\norm{u_{n_3}}_{X_{s,b}} \norm{v_{-n}}_{X_{0,-b'}}}{|n|^{j - 1}}\\
                \lesssim &\sup_{n, n_3} \paren{\sum_{n_2} \JBX[n_2]^{-1+} \norm{u_{n_1}}_{X_{s,b}} \norm{u_{-n_2}}_{X_{s,b}}} \cdot \sum_{n, n_3} \frac{\norm{u_{n_3}}_{X_{s,b}} \norm{v_{-n}}_{X_{0,-b'}}}{|n_3 \pm n| \JBX[n]^{0+}}\\
                \lesssim &\norm{v}_{X_{0,-b'}^{p'}} \prod_{i=1}^3 \norm{u_i}_{X_{s,b}^{p}},
                \end{align*}
                where in the final step we used~\Cref{summation-lemma} again.
            \end{enumerate}
            \item $|\xi_2 - \xi_3| \ge |\xi_2 + \xi_3|$: One can deal with this case in the same way as the previous with the roles of $\xi_1$ and $\xi_3$ swapped. \qedhere
        \end{enumerate}
    \end{enumerate}
\end{proof}

\subsubsection{Estimates for quintic and higher-order nonlinearities}

\begin{proposition}\label{prop:quintic-estimate-modulation}
    Let $2 \le k \le j$, $s > \1{4k}$, $\alpha \in \N_0^{2k+1}$ with $|\alpha| = 2(j-k)$. Then there exist $b' < 0$ and $b' + 1 > b > \1{2}$ such that one has
    \begin{equation}
    \norm{\prod_{i=1}^{2k+1} \partial_x^{\alpha_i}u_i}_{X_{s,b^\prime}^\infty} \lesssim \prod_{i=1}^{2k+1} \norm{u_i}_{X_{s,b}^\infty}.
    \end{equation}
    Additionally for an arbitrary subset of the factors on the left hand side these may be replaced with their complex conjugates.
\end{proposition}
\begin{proof}
    In the proof of this proposition we again assume that the frequencies of the factors in the nonlinearity are ordered in decending order $|\xi_1| \ge |\xi_2| \ge \cdots \ge |\xi_{2k+1}|$. There are essentially two cases to be dealt with, depending on if $\xi_1$ is cancelled out by $\xi_2$ or not. We estimate both cases by duality:

    \begin{enumerate}[wide]
        \item $|\xi_1| \sim |\xi|$: Here $\xi_1$ is not cancelled by $\xi_2$, but the factor corresponding to the product $v_N$ must thus have high frequency. The contribution from this case may be bounded by first using Hölder's inequality
        \begin{align*}
        &\sum_{N \sim N_1} \int_{\R^2} N^s \overline{v}_N N_1^{2(j-k)} \prod_{i=1}^{2k+1} u_{N_i} \d{x}\d{t}\\
        \lesssim &\sum_{N \sim N_1} N_1^{s + 2(j-k)} \norm{v_N}_{L^{\infty-}_xL^2_t} \norm{u_{N_1}}_{L^\infty_xL^2_t} \prod_{i = 2}^{2k+1} \norm{u_{N_i}}_{L^{2k+}_xL^\infty_t}\\
        \intertext{Now we use Kato's inequality~\eqref{eq:kato2} for both $v_N$ and $u_{N_1}$ and the maximal function estimate~\eqref{eq:max-fct} $2k$ times for the remaining $u_{N_i}$.}
        \lesssim &\sum N_1^{s + 1 - 2k+} \norm{v_N}_{X_{0,-b'}} \norm{u_{N_1}}_{X_{0,b}} \prod_{i = 2}^{2k+1} N_i^{\1{2} - \1{2k+}+} \norm{u_{N_i}}_{X_{0,b}}\\
        \lesssim &\paren{\prod_{i = 1}^{2k+1} \norm{u}_{X_{s,b}^\infty}} \norm{v}_{X_{0,-b'}^1} \sum N_1^{\frac{3}{2} - 2k+} \prod_{i = 2}^{2k+1} N_i^{1 - \1{2k} - s +}.
        \end{align*}
        Finally we make use of the embedding $\ell^2 \supset \ell^1$ and~\eqref{eq:bernstein}, where we lose half a derivative using the endpoint estimate. The last term is summable, since we may distribute the $2k - \frac{3}{2}-$ derivatives gain from the first factor and $1 - \1{2k} - (1 - \frac{3}{4k}-) - s+ < 0$ can be achieved for $s > \1{4k}$.
        \item $|\xi_1| \gg |\xi|$: In this case we must have $|\xi_1| \sim |\xi_2|$. To bound this case's contribution we use a Sobolev-embedding for the factor $v_N$ and Kato's inequality~\eqref{eq:kato2} for the two high frequency factors $u_{N_1}$ and $v_N$ after an application of Hölder's inequality.
        \begin{align*}
        &\sum_{N \ll N_1} \int_{\R^2} N^s \overline{v}_N N_1^{2(j-k)} \prod_{i=1}^{2k+1} u_{N_i} \d{x}\d{t}\\
        \lesssim &\sum N_1^{s+2(j-k)} \norm{v_N}_{L^{2}_xL^{\infty-}_t} \norm{u_{N_1}}_{L^\infty_xL^{2+}_t} \norm{u_{N_2}}_{L^\infty_xL^{2+}_t} \prod_{i = 3}^{2k+1} \norm{u_{N_i}}_{L^{2(2k-1)}_xL^\infty_t}\\
        \lesssim &\sum N_1^{s + 1 - 2k +}\norm{v_N}_{X_{0,-b'}} \norm{u_{N_1}}_{X_{0,b}} \norm{u_{N_2}}_{X_{0,b}} \prod_{i = 3}^{2k+1} N_i^{\1{2} - \1{2(2k-1)}+} \norm{u_{N_i}}_{X_{0,b}}\\
        \end{align*}
        For all other factors we applied the maximal function estimate~\eqref{eq:max-fct}. We close this case by~\eqref{eq:bernstein} for the $u_{N_i}$ and using the embedding $\ell^2 \supset \ell^1$ for the factor $v_N$.
        \begin{equation*}
        \lesssim \paren{\prod_{i = 1}^{2k+1} \norm{u}_{X_{s,b}^\infty}} \norm{v}_{X_{0,-b'}^1} \sum N_1^{2 - 2k - s +} \prod_{i = 3}^{2k+1} N_i^{1 - \1{2(2k-1)} - s +}
        \end{equation*}

        The final sums converge, because for every $i = 3,4,\ldots,2k+1$ we have an additional gain of $\frac{2(k-1) + s}{2k-1}-$ derivatives and one can easily check that $$1 - \1{2(2k-1)} - s - \frac{2(k-1) + s}{2k-1}+ < 0 \iff s > \1{4k}.$$
    \end{enumerate}
\end{proof}

Again, as in Corollary~\ref{cor:wp-hat-estimate}, the following corollary is derived from a multilinear interpolation between Proposition~\ref{prop:quintic-estimate-sobolev} and the endpoint estimate in Proposition~\ref{prop:quintic-estimate-modulation} we just proved.

\begin{corollary}\label{cor:wp-modulation-estimate}
    Let $2 \le k \le j$ and $\alpha \in \N_0^{2k+1}$ with $|\alpha| = 2(j-k)$. Then for $2 \le p \le \infty$, $s > \frac{1}{4k} - \frac{2k+1}{2kp}$, and  $b' + 1 > b > \frac{1}{2}$ we have
    \begin{equation}
    \norm{\prod_{i=1}^{2k+1} \partial_x^{\alpha_i}u_i}_{X_{s,b^\prime}^p} \lesssim \prod_{i=1}^{2k+1} \norm{u_i}_{X_{s,b}^p}.
    \end{equation}
    Additionally for an arbitrary subset of the factors on the left hand side these may be replaced with their complex conjugates.
\end{corollary}


\section{Ill-posedness results on $\R$ and $\T$}\label{sec:illposedness}

After now dealing with the positive results regarding the NLS hierarchy in this paper, let us now move focus to negative results. First we will establish Theorems~\ref{thm:illposed-hat} and~\ref{thm:illposedness-modulation}, that shows our Theorems~\ref{thm:wp-hat} and~\ref{thm:wp-modulation} to be optimal in the framework we are using. To do so we first exhibit a family of solutions to equations of type~\eqref{eq:general-nls-hierarchy-eqn}.

\begin{lemma}\label{solution-family}
    For $j \ge 2$ let us choose
    \begin{equation*}
    \delta_0 = \sum_{n=0}^{j} (-1)^{n+1} N^{2(j-n)} \binom{2j}{2n} \quad\text{and}\quad c_0 = \sum_{n=0}^{j-1} (-1)^n N^{2(j-n)-1} \binom{2j}{2n+1}
    \end{equation*}
    and set $u_N(x, t) = \exp(i(Nx + \delta_0 t)) \sech(x-c_0 t)$. Then for every $N > 0$ the function $u_N$ is a solution of a higher-order NLS-like equation~\eqref{eq:nls-like}.
\end{lemma}

Before we prove this Lemma, let us note that the one-parameter family $u_{N}$ of solutions will not suffice for our ill-posedness argument. Luckily, due to the scaling invariances of the equations we are looking at, we can extend this family:

\begin{corollary}\label{cor:ip-solution-family}
    The family of solutions in Lemma~\ref{solution-family} can be extended to a two-parameter family $v_{N, \omega}$ of solutions by setting $v_{N,\omega}(x, t) = \omega u_{\frac{N}{\omega}}(\omega x, \omega^{2j} t)$.
\end{corollary}

\begin{proof}[Proof of Lemma~\ref{solution-family}]
    To simplify notation in the forthcoming proof we will use $f = \sech$. Similarly to the argument in~\cite{AGTowers} we begin with calculating the time derivatives of our supposed solution:
    \begin{equation*}
    i\partial_t u_N(x, t) = \exp(i(Nx + \delta_0 t))(-\delta_0 f - ic_0 f')
    \end{equation*}
    Turning to the space derivatives, a slightly more lengthy calculation yields
    \begin{align*}
    \partial_x^{2j} u_N(x, t) = (-1)^j \exp(i(Nx + \delta_0 t)) &\sum_{m=0}^{j} f^{2m} \sum_{n=m}^{j} (-1)^n c_{n,m} N^{2(j-n)}\cdots\\ &\cdots\left[ \binom{2j}{2n}f - \frac{i}{N} \binom{2j}{2n+1} (2m+1) f' \right],
    \end{align*}
    where we have omitted the arguments to $f$ (which are always equal to $x - c_0 t$) and the coefficients $c_{n,m}$ are taken from the identities
    \begin{equation}\label{eq:sech-identities}
    f^{(2n)}(x) = \sum_{m=0}^{n} c_{n,m} f^{2m+1}(x) \;\;\text{and}\;\; f^{2n+1}(x) = \sum_{m=0}^{n} c_{n,m} (2m+1) f^{2m} f'.
    \end{equation}
    Of these coefficients we will only need to know the exact value $c_{n,0} = 1$. One may easily derive these identities from the well-known fact $f'^2 = f^2 - f^4$ and $f'' = f - 2f^3$.

    Now the parameters $\delta_0$ and $c_0$ were chosen specifically such that the linear part of the equation~\eqref{eq:nls-like} would vanish, so
    \begin{align}\label{eq:nonlinear-rhs}
    (i \partial_t + (-1)^{j+1} \partial_x^{2j})u_N(x, t) = \exp(i(Nx + \delta_0 t)) \left( -\sum_{m=1}^{j} f^{2m} \Sigma_m \right),
    \end{align}
    where we set
    \begin{equation*}
    \Sigma_m = \sum_{n=m}^j (-1)^n c_{n,m} N^{2(j-n)} \left[ \binom{2j}{2n} f - \frac{i}{N} \binom{2j}{2n+1} (2m+1)f' \right]
    \end{equation*}
    for readability.

    What is left to argue now, is that the right-hand side of~\eqref{eq:nonlinear-rhs} can in fact be expressed by inserting our supposed solution $u_N$ into a nonlinear term, that is part of the family described by~\eqref{eq:nls-like}.

    Though this can be achieved by the same argument that is used at the end of the proof of~\cite{AGTowers}*{Lemma~8}. We merely give the two tables of (nonlinear) terms appearing in the double sum~\eqref{eq:nonlinear-rhs}. The rest of the details are left to the reader.

    In~\eqref{eq:nonlinear-rhs} one may notice ``that the last term is missing'', i.~e. there are only $2(j-m)+1$ terms per line, for a total of $j^2$ in the whole table (as opposed to $(j+1)^2$ terms in~\cite{AGTowers}):
    \begin{equation*}
    \begin{matrix}
    & n=1           & n=2               & n=3           & n = 4             & \cdots & n = j    \\
    m=1    & N^{2(j-1)}f^3 & N^{2(j-1)-1}f^2f' & N^{2(j-2)}f^3 & N^{2(j-2)-1}f^2f' & \cdots & f^3      \\
    m=2    &               &                   & N^{2(j-2)}f^5 & N^{2(j-2)-1}f^4f' & \cdots & f^5      \\
    \vdots &               &                   &               &                   & \ddots & \vdots   \\
    m = j  &               &                   &               &                   &        & f^{2m+1}
    \end{matrix}
    \end{equation*}
    Finally the nonlinear terms of the resulting equation that $u_N$ will solve is given:
    \noindent
    \makebox[\textwidth]{\parbox{1.15\textwidth}{%
    \begin{equation*}
    \begin{matrix}
    |u|^2 \partial_x^{2(j-1)} u & (\partial_x |u|^2) \partial_x^{2(j-1)-1} u & (\partial_x^2 |u|^2) \partial_x^{2(j-2)} u & (\partial_x^3 |u|^2) \partial_x^{2(j-2)-1} u & \cdots & (\partial_x^{2(j-1)} |u|^2) u \\
    &                                            & |u|^4 \partial_x^{2(j-2)} u                & (\partial_x |u|^4) \partial_x^{2(j-2)-1} u   & \cdots & (\partial_x^{2(j-2)} |u|^4) u \\
    &                                            &                                            &                                              & \ddots & \vdots                        \\
    &                                            &                                            &                                              &        & |u|^{2j} u
    \end{matrix}
    \end{equation*}
    }}
    Note that these align with the expectation of the equation $u_N$ solves belonging to the family described in~\eqref{eq:nls-like}.
\end{proof}

Now with knowledge of our family of solutions from Corollary~\ref{cor:ip-solution-family} we may reuse an argument given in~\cite{AGTowers}*{Proposition~1}, based upon~\cite{KPV2001}, in order to prove Theorem~\ref{thm:illposed-hat}.

\begin{proof}[Proof of Theorem~\ref{thm:illposed-hat}]
    The same argument as given in~\cite{AGTowers}*{Proposition~1} works here, just that one has to modify the choices made at the start of the proof. We choose $N_1, N_2 \sim N$ but fulfilling $|N_1 - N_2| = \frac{C}{T} N^{sr' - 2(j-1)}$ for a constant $C > 0$. (We keep $N \to \infty$ and $\omega = N^{-sr'}$.)

    When checking the details the astute reader should note, that we have the bound $-\1{r'} < s < \frac{j-1}{r'}$ on the regularity of the data and the propagation speed of a solution is of the order of $N_k^{2j-1}$ (instead of $N_k^{2j}$), for $k = 1, 2$.
\end{proof}

Though our family of solutions is not just useful for proving ill-posedness in Fourier-Lebesgue spaces. We may reuse it again for the proof of Theorem~\ref{thm:illposedness-modulation}. We adapt an argument from~\cite{OhWang2021}*{Lemma~4.1}, which is also based on~\cite{KPV2001}, to our situation.

\begin{proof}[Proof of Theorem~\ref{thm:illposedness-modulation}]
    The proof of this theorem is similar in spirit to that of Theorem~\ref{thm:illposed-hat}, only that one has to be more careful in estimating the difference of solutions at a time $T > 0$. This is due to the fact, that the argument relies on the separation of (essential) support of two solutions in physical space, but this ``conflicts'' with the isometric decomposition used in the definition of modulation spaces.

    Let us begin by stating some parameter choices that we will use down the line. Since $s < \frac{j-1}{2}$ we can fix a $\theta > 0$ such that $4s - 2(j-1) + 2\theta < 0$. Let $N \gg 1$ and $N_1, N_2 \sim N$ but fulfilling the separation condition $|N_1 - N_2| = \frac{C}{T} N^{2s - 2(j-1) + 2\theta}$ for a positive time $T > 0$ and constant $C > 0$. Finally let $\omega = N^{-2s}$. Later we will look at the limiting behaviour $N \to \infty$.

    The next step is establishing bounds on our family of solutions in modulation spaces. We reuse the same arguments as in~\cite{OhWang2021}*{eqns.~(4.7) through~(4.10)} establishing
    $\norm{v_{N_k, \omega}(\cdot, t)}_{M^s_{2,p}} \sim 1$ uniformly in $t \in\R$ and $N, N_1, N_2 \ge 1$.

    For the bound on the difference of solutions at time $t = 0$, we may use the embedding $M^s_{s,p} \supset H^s$ and~\cite{KPV2001}*{eqn.~(3.5)} to estimate
    \begin{equation*}
        \norm{v_{N_1, \omega}(\cdot, 0) - v_{N_2, \omega}(\cdot, 0)}_{M^s_{2,p}} \lesssim N^{2s}|N_1 - N_2| \sim T^{-1}N^{4s - 2(j-1) + 2\theta},
    \end{equation*}
    which converges to zero, for $N \to \infty$.

    Next up is bounding the difference of solutions at a positive time $T > 0$. This is the point where an extra argument is necessary in the modulation space setting. One resorts to looking at frequency contributions to the norm in the vicinity of $N$; in $|\xi - N| \ll N^\theta$ to be precise.

    Noting our increased propagation speed of the solutions, we may argue analogously to~\cite{OhWang2021}*{eqn.~(4.12)} and establish
    \begin{equation}\label{eq:ip-mod-bound}
        |\langle \Box_n v_{N_1, \omega}(\cdot, T), \Box_n v_{N_2, \omega}(\cdot, T) \rangle| \lesssim \1{N^{2(j-1)} |N_1 - N_2|T} \lesssim T^{-1}N^{-2s-2\theta},
    \end{equation}
    which we now utilise in said bound on the difference of solutions at $T > 0$. Following along the lines of~\cite{OhWang2021}*{eqn.~(4.14)}, but using our new bound~\eqref{eq:ip-mod-bound}, we may establish
    \begin{equation*}
        \norm{v_{N_1, \omega}(\cdot, T) - v_{N_2, \omega}(\cdot, T)}_{M^s_{2,p}} \gtrsim 1 - T^{-1}N^{\frac{2}{p}\theta+2s}N^{-2\theta - 2s} = 1 - T^{-1} N^{-2\theta \1{p'}}.
    \end{equation*}
    Letting $N \to \infty$ we have thus established the theorem.
\end{proof}

As mentioned above in the discussion of results in the introduction, the equations leading to ill-posedness on $\R$ are not in general the NLS hierarchy equations. This is of course reflected in the statement of Theorem~\ref{thm:illposed-hat}.

For the interested reader though we give the family of fourth-order equations ($j = 2$) for which a solution was constructed in Lemma~\ref{solution-family}. Let $\lambda \in\R$, then the solution for $j=2$ that was constructed in Lemma~\ref{solution-family} solves the equations
\begin{align}\label{eq:illposedness-equation-example}
i\partial_t u - \partial_x^4 u = &\lambda |u|^2\partial_x^2u + (44-3\lambda)u^2\partial_x^2\overline{u} + (6\lambda - 80)|\partial_x^2 u|u \\&+ (56-4\lambda)(\partial_x u)^2\overline{u} + (40-2\lambda)|u|^4 u .\nonumber
\end{align}

Next we may deal with the Propositions leading to forms of ill-posedness on the torus $\T$, i.e. Theorems~\ref{thm:illposedness-torus-C3} and~\ref{thm:illposedness-torus-C0}.

\begin{proposition}
    The flow $S: \hat H^s_r(\T) \times (-T, T) \to \hat H^s_r(\T)$ of the fourth-order equation ($j = 2$) in the NLS hierarchy
    \begin{equation*}
    iu_t - \partial_x^4 u = -2u^2 \partial_x^2 \overline{u} - 8|u|^2 \partial_x^2 u - 4 |\partial_x u|^2 u - 6 (\partial_x u)^2 \overline{u} + 6 |u|^4 u
    \end{equation*}
    cannot be $C^3$ for any $1 \le r \le \infty$ and $s \in \R$.
\end{proposition}
\begin{proof}
    Following an argument by Bourgain~\cite{Bourgain1997}, assume the flow is indeed thrice continuously differentiable. For a datum $u_0(x) = \delta \phi(x)$, where $\delta > 0$ and $\phi \in H^s(\T)$ for any $s \in \R$ are to be chosen later, we will evaluate the third derivative of the flow at the origin. So let $u$ denote the corresponding solution to $u_0$, then
    \begin{equation*}
    \left. \frac{\partial^3 u}{\partial \delta^3} \right|_{\delta = 0} \sim \int_0^t U(t-t') N_3(U(t')u_0) dt',
    \end{equation*}
    where we have used the notation $N_3(u)$ to denote solely the cubic terms of the nonlinearity and $U(t)$ the linear propagator of the equation.

    We may now write the integrand as its Fourier series to arrive at
    \begin{align*}
    &= \int_0^t \sum_{\substack{k \in \Z\\k_1+k_2+k_3 = k}} e^{ikx} e^{i(t-t')k^4} e^{it'(k_1^4 - k_2^4 + k_3^4)} \hat{\phi}(k_1) \overline{\hat{\phi}(-k_2)} \hat{\phi}(k_3) n_3(k_1, k_2, k_3) dt'\\
    &= \sum_{\substack{k \in \Z\\k_1+k_2+k_3 = k}} e^{ikx + itk^4} \hat{\phi}(k_1) \overline{\hat{\phi}(-k_2)} \hat{\phi}(k_3) n_3(k_1, k_2, k_3) \int_0^t e^{-it'(k^4 - k_1^4 + k_2^4 - k_3^4)} dt'.
    \end{align*}
    Here $n_3(k_1, k_2, k_3) = (k_1 + k_2)^2 + \frac{3}{2} (k_1 + k_3)^2$ is the symbol corresponding to the terms in $N_3$. We may now choose $\hat{\phi}(k) = k^{-s}(\delta_{k, N} + \delta_{k, N_0})$, where $N_0 \ll N$. The choice of the $N_0$ parameter is not important as long as it is, say, fixed. For all further calculations the reader may assume $N_0=1$. We then observe $\norm{\phi}_{\hat H^s_r} \sim 1$ independent of the two parameters.

    Inserting this into the above expression we note that it suffices to look at the terms that produce a resulting frequency of $k = N$. There are three such choices for the tuple $(k_1, k_2, k_3)$, namely $(N, -N, N)$, $(N, -N_0, N_0)$ and $(N_0, -N_0, N)$. Note that for each of these three choices the resonance relation $k^4 - k_1^4 + k_2^4 - k_3^4$ cancels and the integral in the formula above is equal to $t$ and the symbol of our nonlinearity has size on the order of $N^2$.

    These frequency choices thus produce Fourier coefficients (at frequency $N$) on the order of $tN^{2-3s}$ (for the first one) and $tN^{2-s}$ (for the second and third). The remaining five frequency constellations cannot cancel these contributions as they are of lower order in $N$.

    This leaves us with the following lower bound for the Sobolev norm of the operator that is the derivative of the flow:
    \begin{equation*}
    \bignorm{\left. \frac{\partial^3 u}{\partial \delta^3} \right|_{\delta = 0}}_{\hat H^s_r}^{r'} \gtrsim N^{sr'} \cdot t^{r'} N^{(2-s)r'}(1 + N^{-2sr'}) \ge t^{r'}N^{2r'}
    \end{equation*}
    for $1 < r \le \infty$. If $r = 1$ we still have a lower bound of $tN^{2}$ though with a simpler argument.
    Letting $N \to \infty$ we can now see, that the flow cannot be $C^3$ for any $s \in \R$.
\end{proof}

The previous proposition shows that an approach with (just) a fixed-point theorem to prove well-posedness must fail at any regularity in $\hat H^s_r(\T)$. As is stated in Theorem~\ref{thm:illposedness-torus-C0} the situation is much more dire at lower regularities. Its proof lies in the following proposition.

\begin{proposition}\label{prop:ip-torus-low-regularity}
    Let $j\in\N$, $1 \le r \le \infty$ and $s < j - 1$. The flow $S : \hat H^s_r(\T) \times (-T, T) \to \hat H^s_r(\T)$ of the Cauchy problem
    \begin{equation}\label{eq:ip-torus-low-regularity}
    i\partial_t u + (-1)^{j+1}\partial_x^{2j} u = |u|^2 \partial_x^{2j-2} u \quad\text{with}\quad u(t = 0) = u_0 \in \hat H^s_r(\T),
    \end{equation}
    cannot be uniformly continuous on bounded sets.
\end{proposition}

We want to point out, that equation~\eqref{eq:ip-torus-low-regularity} is in fact a higher-order NLS-like equation according to~\eqref{eq:nls-like}. More so it even fits the structure of an NLS hierarchy equation~\eqref{eq:nls-hierarchy-general-eq}, though it is unlikely to be one because of its simple nonlinearity.

\begin{proof}[Proof of Proposition~\ref{prop:ip-torus-low-regularity}]
    We follow a similar argument to the one used in, for example,~\cite{OhTzvetkov2017}*{Appendix~A.2}.

    The reader may verify that our equation~\eqref{eq:ip-torus-low-regularity} has the two-parameter family of solutions
    \begin{equation*}
    u_{N, a}(x, t) = N^{-s} a \exp(i(Nx - N^{2j} t + N^{2j-2-2s} |a|^2 t)).
    \end{equation*}
    We fix $a \in \R$ at two different values and will only deal with the two solution families $u_{n}(x, t) = u_{N_n, 1}(x, t)$ and $\tilde{u}_{n}(x, t) = u_{N_n, 1+\frac{1}{n}}(x, t)$ depending on $n \in \N$. $N_n$ will be chosen later. We find that
    \begin{equation*}
    \norm{u_n(\cdot, 0)}_{\hat H^s_r}, \norm{\tilde{u}_n(\cdot, 0)}_{\hat H^s_r} \lesssim 1 \qquad\text{and}\qquad \norm{u_n(\cdot, 0) - \tilde{u}_n(\cdot, 0)}_{\hat H^s_r} \sim \frac{1}{n},
    \end{equation*}
    where the implicit constant is independent of $n \in \N$. Now choosing
    \begin{equation*}
    t_n = \frac{\pi N_n^{2s+2-2j}}{(1 + \frac{1}{n})^2 - 1}
    \end{equation*}
    and $N_n$ large enough, such that $t_n \le \frac{1}{n}$, we may then observe that
    \begin{equation*}
    \norm{u_n(\cdot, t_n) - \tilde{u}_n(\cdot, t_n)}_{\hat H^s_r} = \left| \exp(i N_n^{2j-2-2s} (1 - (1 + \frac{1}{n})^2) t_n) - (1 + \frac{1}{n}) \right| = 2 + \frac{1}{n}.
    \end{equation*}
    Letting $n \to \infty$ this shows that the flow is not uniformly continuous. Such a choice is possible, if $2s + 2 - 2j < 0$ or equivalently $s < j -1$ as stated.
\end{proof}

    \appendix
\section{The first few NLS hierarchy equations}\label{appendix:nls-hierarchy}

For the reader's convenience and future reference we will list the first few conserved quantities $I_k$ derived from~\eqref{eq:recursion-flux} and their associated nonlinear evolution equations~\eqref{eq:general-nls-hierarchy-eqn} in terms of the potentials $q$ and $r$. In this form both the focusing and defocusing variants of the (NLS) hierarchy can be derived by the identifications $r = +\overline{q}$ or $r = -\overline{q}$ respectively.

Though we will not just give the even numbered equations, corresponding to the NLS hierarchy, but also those corresponding to the mKdV hierarchy. Using the identification $r = q$ one arrives at the real mKdV hierarchy discussed in~\cite{AGTowers}. Deriving \emph{a} complex mKdV hierarchy (of which again there is a defocusing and focusing variant) is also possible (again using the identifications $r = \pm \overline{q}$). But there are two problems:
\begin{enumerate}
    \item Identifying $r = \pm \overline{q}$ for the equation induced by $I_4$, see~\eqref{eq:mKdV-appendix}, does not lead to the well known form of the complex mKdV equation given in~\eqref{eq:mKdV}. Rather the nonlinearity is replaced by $\pm 6 |u|^2 \partial_x u$, up to a choice of $\alpha_3$. For our local well-posedness theory this does not make a difference, as we are able to estimate both nonlinearities equally well. Though for a treatment relying more on the structure of the equation (e.g. for cancellation properties) this may be a relevant difference.

    When looking at the real mKdV hierarchy, i.e. using $r = q$, this problem does not present itself.

    \item If one wishes to use the identification $r = -\overline{q}$ the compatibility condition for the coefficients $\alpha_{2j+1}$ reads $\alpha_{2j+1} = - \overline{(\alpha_{2j+1})}$, as in~\eqref{eq:coeff-constraint}. Meaning $\alpha_{2j+1}$ is imaginary\footnote{It is non-zero, as otherwise this would lead to a trivial equation.} and thus introducing a factor $i$ that is usually not present in complex mKdV-like equations.

    Again, looking at the real mKdV hierarchy this is a non-issue, see also~\cite{Alberty1982-i}*{Section~3.2.2}.
\end{enumerate}

Not choosing an identification $r = \pm \overline{q}$ or $r = q$ also has the advantage, that we may derive the equations in the KdV hierarchy by setting $r = -1$, see~\cite{Alberty1982-i}*{Section~3.2.1}

Finally we note that our conserved quantities may differ from those given elsewhere in the literature, as these are only determined up to (repeated) partial integration and simplification. The equations though only differ up to a choice of~$\alpha_k$.

A similar listing is given in~\cite{Koch2018}*{Appendix~C} and~\cite{KochKlaus2023}*{Appendix~C}.

\begin{enumerate}[wide]
\item $n = 1, 2$. Phase shifts \& Group of translations
\begin{align*}
&I_1 = -\1{2i} \int qr \d{x} &\text{and} &&I_2 = - \paren{\1{2i}}^2 \int q r_x \d{x}\\
&q_t = 2\alpha_0 q &\text{and} &&q_t = i\alpha_1 q_x
\end{align*}

\item $n = 3$. cubic nonlinear Schrödinger equation
\begin{align*}
I_3 = \paren{\1{2i}}^3 \int q_x r_x + q^2 r^2 \d{x}\\
q_t = \frac{\alpha_2}{2} (- q_{xx} + 2q^2 r)
\end{align*}

\item $n = 4$. modified Korteweg-de-Vries equation
\begin{align}
\nonumber I_4 = \paren{\1{2i}}^4 \int q_x r_{xx} + q q_x r^2 + 4 q^2 r r_x \d{x}\\
q_t = \frac{-\alpha_3}{4} (q_{xxx} - 6 q q_x r) \label{eq:mKdV-appendix}
\end{align}

\item $n = 5$. fourth order NLS hierarchy equation
\begin{align*}
I_5 = \paren{\1{2i}}^5 \int -q_{xx} r_{xx} + q_{xx}r^2 + 6q q_x r r_x + 5 q^2 r_x^2 + 6 q^2 r r_{xx} - 2 q^3 r^3 \d{x}\\
q_t = \frac{-\alpha_4}{8} (-q_{xxxx} + 8 q q_{xx} r + 2 q^2 r_{xx} + 4 q q_x r_x + 6 q_x^2 r - 6 q^3 r^2)
\end{align*}

\item $n = 6$. fifth order mKdV hierarchy equation
\begin{align*}
I_6 = \paren{\1{2i}}^6 \int &- q r_{xxxxx} + q q_{xxx} r^2 + 8 q q_{xx} r r_{x} + 11 q q_x r_x^2 + 12 q q_x r r_{xx}\\ &+ 18 q^2 r_x r_{xx} + 8 q^2 r r_{xxx} - 6 q^2 q_x r^3 - 16 q^3 r^2 r_x \d{x}
\end{align*}
\begin{align*}
q_t = \frac{i\alpha_5}{2^4} (q_{xxxxx} &- 10 q q_{xxx} r - 10 q q_{xx} r_x - 10 q q_x r_{xx} - 20 q_x q_{xx} r - 10 q_x^2 r_x\\ &+ 30 q^2 q_x r^2)
\end{align*}

\item $n = 7$. sixth order NLS hierarchy equation
\begin{align*}
I_7 = \paren{\1{2i}}^7 \int &-q r_{xxxxxx} + q q_{xxxxx}r^2 + 10qq_{xxx}rr_x + 19 qq_{xx}r_x^2 \\&+ 52 q q_x r_xr_{xx} + 20 qq_{xx}rr_{xx} + 20qq_xrr_{xxx} + 19 q^2 r_{xx}^2 \\&+ 28 q^2 r_xr_{xxx} + 10 q^2 r r_{xxxx} + 5 q^4 r^4 \\&- 6 qq_x^2 r^3 - 8 q^2 q_{xx} r^3 - 64 q^2 q_x r^2 r_x - 50 q^3 rr_x^2 - 30 q^3r^2r_{xx}\d{x}
\end{align*}
\begin{align*}
q_t = \frac{\alpha_6}{2^5} (-q_{xxxxxx} &+ 12qq_{xxxx}r + 2q^2r_{xxxx} + 18 qq_{xxx}r_x + 22 qq_{xx}r_{xx} + 8 qq_xr_{xxx} \\&+ 30 q_xq_{xxx}r + 20 q_x^2r_{xx} + 20 q_{xx}^2r + 50q_xq_{xx}r_x + 20 q^4r^3 \\&- 20q^3rr_{xx} - 50 q^2q_{xx}r^2 - 10q^3r_x^2 - 60 q^2q_xrr_x - 70 qq_x^2r^2)
\end{align*}
\end{enumerate}

    \subsection*{Data availability statement}
    No data was used for the research described in this article.

    \subsection*{Conflict of interest statement}
    The author declares there to be no conflict of interest associated with this article.

	\bibliography{bibliography.bib}

\begin{bibdiv}
\begin{biblist}

\bib{AKNS1974}{article}{
      author={Ablowitz, Mark~J.},
      author={Kaup, David~J.},
      author={Newell, Alan~C.},
      author={Segur, Harvey},
       title={The inverse scattering transform-{F}ourier analysis for nonlinear
  problems},
        date={1974},
        ISSN={0022-2526,1467-9590},
     journal={Studies in Appl. Math.},
      volume={53},
      number={4},
       pages={249\ndash 315},
         url={https://doi.org/10.1002/sapm1974534249},
      review={\MR{450815}},
}

\bib{Alberty1982-i}{article}{
      author={Alberty, J.~M.},
      author={Koikawa, T.},
      author={Sasaki, R.},
       title={Canonical structure of soliton equations. {I}},
        date={1982},
        ISSN={0167-2789,1872-8022},
     journal={Phys. D},
      volume={5},
      number={1},
       pages={43\ndash 65},
         url={https://doi.org/10.1016/0167-2789(82)90049-5},
      review={\MR{666528}},
}

\bib{BanicaVega2024}{article}{
      author={Banica, Valeria},
      author={Luc\`a, Renato},
      author={Tzvetkov, Nikolay},
      author={Vega, Luis},
       title={Blow-{U}p for the 1{D} {C}ubic {NLS}},
        date={2024},
        ISSN={0010-3616,1432-0916},
     journal={Comm. Math. Phys.},
      volume={405},
      number={1},
       pages={11},
         url={https://doi.org/10.1007/s00220-023-04906-3},
      review={\MR{4693911}},
}

\bib{Benyi2020}{book}{
      author={B\'{e}nyi, \'{A}rp\'{a}d},
      author={Okoudjou, Kasso~A.},
       title={Modulation spaces},
      series={Applied and Numerical Harmonic Analysis},
   publisher={Birkh\"{a}user/Springer, New York},
        date={2020},
        ISBN={978-1-0716-0330-7; 978-1-0716-0332-1},
         url={https://doi.org/10.1007/978-1-0716-0332-1},
      review={\MR{4286055}},
}

\bib{Bourgain1993-1}{article}{
      author={Bourgain, J.},
       title={Fourier transform restriction phenomena for certain lattice
  subsets and applications to nonlinear evolution equations. {I}.
  {S}chr\"{o}dinger equations},
        date={1993},
        ISSN={1016-443X,1420-8970},
     journal={Geom. Funct. Anal.},
      volume={3},
      number={2},
       pages={107\ndash 156},
         url={https://doi.org/10.1007/BF01896020},
      review={\MR{1209299}},
}

\bib{Bourgain1993-2}{article}{
      author={Bourgain, J.},
       title={Fourier transform restriction phenomena for certain lattice
  subsets and applications to nonlinear evolution equations. {II}. {T}he
  {K}d{V}-equation},
        date={1993},
        ISSN={1016-443X,1420-8970},
     journal={Geom. Funct. Anal.},
      volume={3},
      number={3},
       pages={209\ndash 262},
         url={https://doi.org/10.1007/BF01895688},
      review={\MR{1215780}},
}

\bib{Bourgain1997}{article}{
      author={Bourgain, J.},
       title={Periodic {K}orteweg de {V}ries equation with measures as initial
  data},
        date={1997},
        ISSN={1022-1824,1420-9020},
     journal={Selecta Math. (N.S.)},
      volume={3},
      number={2},
       pages={115\ndash 159},
         url={https://doi.org/10.1007/s000290050008},
      review={\MR{1466164}},
}

\bib{BrunLiLiuZine2023}{article}{
      author={Brun, Enguerrand},
      author={Li, Guopeng},
      author={Liu, Ruoyuan},
      author={Zine, Younes},
       title={{Global well-posedness of one-dimensional cubic fractional
  nonlinear Schrödinger equations in negative Sobolev spaces}},
        date={2023},
      eprint={arXiv:2311.13370},
}

\bib{Cazenave2003}{book}{
      author={Cazenave, Thierry},
       title={Semilinear {S}chr\"{o}dinger equations},
      series={Courant Lecture Notes in Mathematics},
   publisher={New York University, Courant Institute of Mathematical Sciences,
  New York; American Mathematical Society, Providence, RI},
        date={2003},
      volume={10},
        ISBN={0-8218-3399-5},
         url={https://doi.org/10.1090/cln/010},
      review={\MR{2002047}},
}

\bib{Chen2020}{article}{
      author={Chen, Mingjuan},
      author={Guo, Boling},
       title={Local well and ill posedness for the modified {K}d{V} equations
  in subcritical modulation spaces},
        date={2020},
        ISSN={1539-6746,1945-0796},
     journal={Commun. Math. Sci.},
      volume={18},
      number={4},
       pages={909\ndash 946},
         url={https://doi.org/10.4310/CMS.2020.v18.n4.a2},
      review={\MR{4129658}},
}

\bib{FabFive-mKdV-2003}{article}{
      author={Colliander, J.},
      author={Keel, M.},
      author={Staffilani, G.},
      author={Takaoka, H.},
      author={Tao, T.},
       title={Sharp global well-posedness for {K}d{V} and modified {K}d{V} on
  {$\Bbb R$} and {$\Bbb T$}},
        date={2003},
        ISSN={0894-0347,1088-6834},
     journal={J. Amer. Math. Soc.},
      volume={16},
      number={3},
       pages={705\ndash 749},
         url={https://doi.org/10.1090/S0894-0347-03-00421-1},
      review={\MR{1969209}},
}

\bib{DNY2021}{article}{
      author={Deng, Yu},
      author={Nahmod, Andrea~R.},
      author={Yue, Haitian},
       title={Optimal local well-posedness for the periodic derivative
  nonlinear {S}chr\"{o}dinger equation},
        date={2021},
        ISSN={0010-3616,1432-0916},
     journal={Comm. Math. Phys.},
      volume={384},
      number={2},
       pages={1061\ndash 1107},
         url={https://doi.org/10.1007/s00220-020-03898-8},
      review={\MR{4259382}},
}

\bib{Faddeev1987}{book}{
      author={Faddeev, L.~D.},
      author={Takhtajan, L.~A.},
       title={Hamiltonian methods in the theory of solitons},
      series={Springer Series in Soviet Mathematics},
   publisher={Springer-Verlag, Berlin},
        date={1987},
        ISBN={3-540-15579-1},
         url={https://doi.org/10.1007/978-3-540-69969-9},
        note={Translated from the Russian by A. G. Reyman [A. G. Re\u{\i}man]},
      review={\MR{905674}},
}

\bib{Fefferman1970}{article}{
      author={Fefferman, Charles},
       title={Inequalities for strongly singular convolution operators},
        date={1970},
        ISSN={0001-5962,1871-2509},
     journal={Acta Math.},
      volume={124},
       pages={9\ndash 36},
         url={https://doi.org/10.1007/BF02394567},
      review={\MR{257819}},
}

\bib{Feichtinger1983}{techreport}{
      author={Feichtinger, Hans~Georg},
       title={{Modulation Spaces on Locally Compact Abelian Groups}},
 institution={University of Vienna},
        date={1983},
}

\bib{Feng1994}{article}{
      author={Feng, Xue~Shang},
       title={The global {C}auchy problem for a higher-order nonlinear
  {S}chr\"{o}dinger equation},
        date={1994},
        ISSN={0019-5588,0975-7465},
     journal={Indian J. Pure Appl. Math.},
      volume={25},
      number={6},
       pages={583\ndash 606},
      review={\MR{1285221}},
}

\bib{Feng1995}{article}{
      author={Feng, Xue~Shang},
       title={The global solutions for a fourth order nonlinear
  {S}chr\"{o}dinger equation},
        date={1995},
        ISSN={0252-9602},
     journal={Acta Math. Sci. (English Ed.)},
      volume={15},
      number={2},
       pages={196\ndash 206},
         url={https://doi.org/10.1016/S0252-9602(18)30040-7},
      review={\MR{1350362}},
}

\bib{GTV1997}{article}{
      author={Ginibre, J.},
      author={Tsutsumi, Y.},
      author={Velo, G.},
       title={On the {C}auchy problem for the {Z}akharov system},
        date={1997},
        ISSN={0022-1236,1096-0783},
     journal={J. Funct. Anal.},
      volume={151},
      number={2},
       pages={384\ndash 436},
         url={https://doi.org/10.1006/jfan.1997.3148},
      review={\MR{1491547}},
}

\bib{Ginibre1996}{incollection}{
      author={Ginibre, Jean},
       title={Le probl\`eme de {C}auchy pour des {EDP} semi-lin\'{e}aires
  p\'{e}riodiques en variables d'espace (d'apr\`es {B}ourgain)},
        date={1996},
       pages={Exp. No. 796, 4, 163\ndash 187},
        note={S\'{e}minaire Bourbaki, Vol. 1994/95},
      review={\MR{1423623}},
}

\bib{Grafakos-classic}{book}{
      author={Grafakos, Loukas},
       title={Classical {F}ourier analysis},
     edition={Third},
      series={Graduate Texts in Mathematics},
   publisher={Springer, New York},
        date={2014},
      volume={249},
        ISBN={978-1-4939-1193-6; 978-1-4939-1194-3},
         url={https://doi.org/10.1007/978-1-4939-1194-3},
      review={\MR{3243734}},
}

\bib{AGDiss}{thesis}{
      author={Gr\"{u}nrock, Axel},
       title={{New applications of the Fourier restriction norm method to
  wellposedness problems for nonlinear evolution equations}},
        type={Ph.D. Thesis},
        date={2002},
}

\bib{Grünrock2004}{article}{
      author={Gr\"{u}nrock, Axel},
       title={An improved local well-posedness result for the modified {K}d{V}
  equation},
        date={2004},
        ISSN={1073-7928,1687-0247},
     journal={Int. Math. Res. Not.},
      number={61},
       pages={3287\ndash 3308},
         url={https://doi.org/10.1155/S1073792804140981},
      review={\MR{2096258}},
}

\bib{Grünrock2005}{article}{
      author={Gr\"{u}nrock, Axel},
       title={Bi- and trilinear {S}chr\"{o}dinger estimates in one space
  dimension with applications to cubic {NLS} and {DNLS}},
        date={2005},
        ISSN={1073-7928,1687-0247},
     journal={Int. Math. Res. Not.},
      number={41},
       pages={2525\ndash 2558},
         url={https://doi.org/10.1155/IMRN.2005.2525},
      review={\MR{2181058}},
}

\bib{AGTowers}{article}{
      author={Gr\"{u}nrock, Axel},
       title={On the hierarchies of higher-order m{K}d{V} and {K}d{V}
  equations},
        date={2010},
        ISSN={1895-1074,1644-3616},
     journal={Cent. Eur. J. Math.},
      volume={8},
      number={3},
       pages={500\ndash 536},
         url={https://doi.org/10.2478/s11533-010-0024-5},
      review={\MR{2653659}},
}

\bib{GrünrockHerr2008}{article}{
      author={Gr\"{u}nrock, Axel},
      author={Herr, Sebastian},
       title={Low regularity local well-posedness of the derivative nonlinear
  {S}chr\"{o}dinger equation with periodic initial data},
        date={2008},
        ISSN={0036-1410,1095-7154},
     journal={SIAM J. Math. Anal.},
      volume={39},
      number={6},
       pages={1890\ndash 1920},
         url={https://doi.org/10.1137/070689139},
      review={\MR{2390318}},
}

\bib{GrünrockVega2009}{article}{
      author={Gr\"{u}nrock, Axel},
      author={Vega, Luis},
       title={Local well-posedness for the modified {K}d{V} equation in almost
  critical {$\widehat{H^r_s}$}-spaces},
        date={2009},
        ISSN={0002-9947,1088-6850},
     journal={Trans. Amer. Math. Soc.},
      volume={361},
      number={11},
       pages={5681\ndash 5694},
         url={https://doi.org/10.1090/S0002-9947-09-04611-X},
      review={\MR{2529909}},
}

\bib{Guo2017}{article}{
      author={Guo, Shaoming},
       title={On the 1{D} cubic nonlinear {S}chr\"{o}dinger equation in an
  almost critical space},
        date={2017},
        ISSN={1069-5869,1531-5851},
     journal={J. Fourier Anal. Appl.},
      volume={23},
      number={1},
       pages={91\ndash 124},
         url={https://doi.org/10.1007/s00041-016-9464-z},
      review={\MR{3602811}},
}

\bib{HKV2023}{article}{
      author={Harrop-Griffiths, Benjamin},
      author={Killip, Rowan},
      author={Vi\c{s}an, Monica},
       title={Large-data equicontinuity for the derivative {NLS}},
        date={2023},
        ISSN={1073-7928,1687-0247},
     journal={Int. Math. Res. Not. IMRN},
      number={6},
       pages={4601\ndash 4642},
         url={https://doi.org/10.1093/imrn/rnab374},
      review={\MR{4565673}},
}

\bib{HO1994}{article}{
      author={Hayashi, Nakao},
      author={Ozawa, Tohru},
       title={Finite energy solutions of nonlinear {S}chr\"{o}dinger equations
  of derivative type},
        date={1994},
        ISSN={0036-1410},
     journal={SIAM J. Math. Anal.},
      volume={25},
      number={6},
       pages={1488\ndash 1503},
         url={https://doi.org/10.1137/S0036141093246129},
      review={\MR{1302158}},
}

\bib{Ikeda2021}{article}{
      author={Hirayama, Hiroyuki},
      author={Ikeda, Masahiro},
      author={Tanaka, Tomoyuki},
       title={Well-posedness for the fourth-order {S}chr\"{o}dinger equation
  with third order derivative nonlinearities},
        date={2021},
        ISSN={1021-9722,1420-9004},
     journal={NoDEA Nonlinear Differential Equations Appl.},
      volume={28},
      number={5},
       pages={Paper No. 46, 72},
         url={https://doi.org/10.1007/s00030-021-00707-6},
      review={\MR{4274694}},
}

\bib{HuoJia2005}{article}{
      author={Huo, Zhaohui},
      author={Jia, Yueling},
       title={The {C}auchy problem for the fourth-order nonlinear
  {S}chr\"{o}dinger equation related to the vortex filament},
        date={2005},
        ISSN={0022-0396,1090-2732},
     journal={J. Differential Equations},
      volume={214},
      number={1},
       pages={1\ndash 35},
         url={https://doi.org/10.1016/j.jde.2004.09.005},
      review={\MR{2143510}},
}

\bib{KenigPilod2016}{article}{
      author={Kenig, Carlos~E.},
      author={Pilod, Didier},
       title={Local well-posedness for the {K}d{V} hierarchy at high
  regularity},
        date={2016},
        ISSN={1079-9389},
     journal={Adv. Differential Equations},
      volume={21},
      number={9-10},
       pages={801\ndash 836},
         url={http://projecteuclid.org/euclid.ade/1465912584},
      review={\MR{3513119}},
}

\bib{KPV1991}{article}{
      author={Kenig, Carlos~E.},
      author={Ponce, Gustavo},
      author={Vega, Luis},
       title={Oscillatory integrals and regularity of dispersive equations},
        date={1991},
        ISSN={0022-2518,1943-5258},
     journal={Indiana Univ. Math. J.},
      volume={40},
      number={1},
       pages={33\ndash 69},
         url={https://doi.org/10.1512/iumj.1991.40.40003},
      review={\MR{1101221}},
}

\bib{KPV1993}{article}{
      author={Kenig, Carlos~E.},
      author={Ponce, Gustavo},
      author={Vega, Luis},
       title={Well-posedness and scattering results for the generalized
  {K}orteweg-de {V}ries equation via the contraction principle},
        date={1993},
        ISSN={0010-3640,1097-0312},
     journal={Comm. Pure Appl. Math.},
      volume={46},
      number={4},
       pages={527\ndash 620},
         url={https://doi.org/10.1002/cpa.3160460405},
      review={\MR{1211741}},
}

\bib{KPV1994-ii}{article}{
      author={Kenig, Carlos~E.},
      author={Ponce, Gustavo},
      author={Vega, Luis},
       title={Higher-order nonlinear dispersive equations},
        date={1994},
        ISSN={0002-9939,1088-6826},
     journal={Proc. Amer. Math. Soc.},
      volume={122},
      number={1},
       pages={157\ndash 166},
         url={https://doi.org/10.2307/2160855},
      review={\MR{1195480}},
}

\bib{KPV1994-i}{incollection}{
      author={Kenig, Carlos~E.},
      author={Ponce, Gustavo},
      author={Vega, Luis},
       title={On the hierarchy of the generalized {K}d{V} equations},
        date={1994},
   booktitle={Singular limits of dispersive waves ({L}yon, 1991)},
      series={NATO Adv. Sci. Inst. Ser. B: Phys.},
      volume={320},
   publisher={Plenum, New York},
       pages={347\ndash 356},
      review={\MR{1321214}},
}

\bib{KPV1996-i}{article}{
      author={Kenig, Carlos~E.},
      author={Ponce, Gustavo},
      author={Vega, Luis},
       title={A bilinear estimate with applications to the {K}d{V} equation},
        date={1996},
        ISSN={0894-0347,1088-6834},
     journal={J. Amer. Math. Soc.},
      volume={9},
      number={2},
       pages={573\ndash 603},
         url={https://doi.org/10.1090/S0894-0347-96-00200-7},
      review={\MR{1329387}},
}

\bib{KPV1996-ii}{article}{
      author={Kenig, Carlos~E.},
      author={Ponce, Gustavo},
      author={Vega, Luis},
       title={Quadratic forms for the {$1$}-{D} semilinear {S}chr\"{o}dinger
  equation},
        date={1996},
        ISSN={0002-9947,1088-6850},
     journal={Trans. Amer. Math. Soc.},
      volume={348},
      number={8},
       pages={3323\ndash 3353},
         url={https://doi.org/10.1090/S0002-9947-96-01645-5},
      review={\MR{1357398}},
}

\bib{KPV2001}{article}{
      author={Kenig, Carlos~E.},
      author={Ponce, Gustavo},
      author={Vega, Luis},
       title={On the ill-posedness of some canonical dispersive equations},
        date={2001},
        ISSN={0012-7094,1547-7398},
     journal={Duke Math. J.},
      volume={106},
      number={3},
       pages={617\ndash 633},
         url={https://doi.org/10.1215/S0012-7094-01-10638-8},
      review={\MR{1813239}},
}

\bib{KNV2023}{article}{
      author={Killip, Rowan},
      author={Ntekoume, Maria},
      author={Vi\c{s}an, Monica},
       title={On the well-posedness problem for the derivative nonlinear
  {S}chr\"{o}dinger equation},
        date={2023},
        ISSN={2157-5045,1948-206X},
     journal={Anal. PDE},
      volume={16},
      number={5},
       pages={1245\ndash 1270},
         url={https://doi.org/10.2140/apde.2023.16.1245},
      review={\MR{4628747}},
}

\bib{Killip2018}{article}{
      author={Killip, Rowan},
      author={Vi\c{s}an, Monica},
      author={Zhang, Xiaoyi},
       title={Low regularity conservation laws for integrable {PDE}},
        date={2018},
        ISSN={1016-443X,1420-8970},
     journal={Geom. Funct. Anal.},
      volume={28},
      number={4},
       pages={1062\ndash 1090},
         url={https://doi.org/10.1007/s00039-018-0444-0},
      review={\MR{3820439}},
}

\bib{Klaus2023}{article}{
      author={Klaus, Friedrich},
       title={Wellposedness of {NLS} in modulation spaces},
        date={2023},
        ISSN={1069-5869,1531-5851},
     journal={J. Fourier Anal. Appl.},
      volume={29},
      number={1},
       pages={Paper No. 9, 37},
         url={https://doi.org/10.1007/s00041-022-09985-9},
      review={\MR{4534495}},
}

\bib{KochKlaus2023}{article}{
      author={Klaus, Friedrich},
      author={Koch, Herbert},
      author={Liu, Baoping},
       title={{Wellposedness for the KdV hierarchy}},
        date={2023},
      eprint={arXiv:2309.12773},
}

\bib{KlausSchippa2022}{article}{
      author={Klaus, Friedrich},
      author={Schippa, Robert},
       title={A priori estimates for the derivative nonlinear {S}chr\"{o}dinger
  equation},
        date={2022},
        ISSN={0532-8721},
     journal={Funkcial. Ekvac.},
      volume={65},
      number={3},
       pages={329\ndash 346},
      review={\MR{4515528}},
}

\bib{KochTataru2007}{article}{
      author={Koch, Herbert},
      author={Tataru, Daniel},
       title={A priori bounds for the 1{D} cubic {NLS} in negative {S}obolev
  spaces},
        date={2007},
        ISSN={1073-7928,1687-0247},
     journal={Int. Math. Res. Not. IMRN},
      number={16},
       pages={Art. ID rnm053, 36},
         url={https://doi.org/10.1093/imrn/rnm053},
      review={\MR{2353092}},
}

\bib{KochTataru2012}{article}{
      author={Koch, Herbert},
      author={Tataru, Daniel},
       title={Energy and local energy bounds for the 1-d cubic {NLS} equation
  in {$H^{-\frac14}$}},
        date={2012},
        ISSN={0294-1449,1873-1430},
     journal={Ann. Inst. H. Poincar\'{e} C Anal. Non Lin\'{e}aire},
      volume={29},
      number={6},
       pages={955\ndash 988},
         url={https://doi.org/10.1016/j.anihpc.2012.05.006},
      review={\MR{2995102}},
}

\bib{Koch2018}{article}{
      author={Koch, Herbert},
      author={Tataru, Daniel},
       title={Conserved energies for the cubic nonlinear {S}chr\"{o}dinger
  equation in one dimension},
        date={2018},
        ISSN={0012-7094,1547-7398},
     journal={Duke Math. J.},
      volume={167},
      number={17},
       pages={3207\ndash 3313},
         url={https://doi.org/10.1215/00127094-2018-0033},
      review={\MR{3874652}},
}

\bib{OhSulem2012}{article}{
      author={Oh, Tadahiro},
      author={Sulem, Catherine},
       title={On the one-dimensional cubic nonlinear {S}chr\"{o}dinger equation
  below {$L^2$}},
        date={2012},
        ISSN={2156-2261,2154-3321},
     journal={Kyoto J. Math.},
      volume={52},
      number={1},
       pages={99\ndash 115},
         url={https://doi.org/10.1215/21562261-1503772},
      review={\MR{2892769}},
}

\bib{OhTzvetkov2017}{article}{
      author={Oh, Tadahiro},
      author={Tzvetkov, Nikolay},
       title={Quasi-invariant {G}aussian measures for the cubic fourth order
  nonlinear {S}chr\"{o}dinger equation},
        date={2017},
        ISSN={0178-8051,1432-2064},
     journal={Probab. Theory Related Fields},
      volume={169},
      number={3-4},
       pages={1121\ndash 1168},
         url={https://doi.org/10.1007/s00440-016-0748-7},
      review={\MR{3719064}},
}

\bib{OhWang2020}{article}{
      author={Oh, Tadahiro},
      author={Wang, Yuzhao},
       title={Global well-posedness of the one-dimensional cubic nonlinear
  {S}chr\"odinger equation in almost critical spaces},
        date={2020},
        ISSN={0022-0396,1090-2732},
     journal={J. Differential Equations},
      volume={269},
      number={1},
       pages={612\ndash 640},
         url={https://doi.org/10.1016/j.jde.2019.12.017},
      review={\MR{4081534}},
}

\bib{OhWang2021}{article}{
      author={Oh, Tadahiro},
      author={Wang, Yuzhao},
       title={On global well-posedness of the modified {K}d{V} equation in
  modulation spaces},
        date={2021},
        ISSN={1078-0947,1553-5231},
     journal={Discrete Contin. Dyn. Syst.},
      volume={41},
      number={6},
       pages={2971\ndash 2992},
         url={https://doi.org/10.3934/dcds.2020393},
      review={\MR{4235636}},
}

\bib{Ozawa1996}{article}{
      author={Ozawa, T.},
       title={On the nonlinear {S}chr\"{o}dinger equations of derivative type},
        date={1996},
        ISSN={0022-2518,1943-5258},
     journal={Indiana Univ. Math. J.},
      volume={45},
      number={1},
       pages={137\ndash 163},
         url={https://doi.org/10.1512/iumj.1996.45.1962},
      review={\MR{1406687}},
}

\bib{Palais1997}{article}{
      author={Palais, Richard~S.},
       title={The symmetries of solitons},
        date={1997},
        ISSN={0273-0979,1088-9485},
     journal={Bull. Amer. Math. Soc. (N.S.)},
      volume={34},
      number={4},
       pages={339\ndash 403},
         url={https://doi.org/10.1090/S0273-0979-97-00732-5},
      review={\MR{1462745}},
}

\bib{Pattakos2019}{article}{
      author={Pattakos, N.},
       title={N{LS} in the modulation space {$M_{2,q}({\Bbb {R}})$}},
        date={2019},
        ISSN={1069-5869,1531-5851},
     journal={J. Fourier Anal. Appl.},
      volume={25},
      number={4},
       pages={1447\ndash 1486},
         url={https://doi.org/10.1007/s00041-018-09655-9},
      review={\MR{3977124}},
}

\bib{Pilod2008}{article}{
      author={Pilod, Didier},
       title={On the {C}auchy problem for higher-order nonlinear dispersive
  equations},
        date={2008},
        ISSN={0022-0396,1090-2732},
     journal={J. Differential Equations},
      volume={245},
      number={8},
       pages={2055\ndash 2077},
         url={https://doi.org/10.1016/j.jde.2008.07.017},
      review={\MR{2446185}},
}

\bib{Sasaki1982-ii}{article}{
      author={Sasaki, R.},
       title={Canonical structure of soliton equations. {II}. {T}he
  {K}aup-{N}ewell system},
        date={1982},
        ISSN={0167-2789,1872-8022},
     journal={Phys. D},
      volume={5},
      number={1},
       pages={66\ndash 74},
         url={https://doi.org/10.1016/0167-2789(82)90050-1},
      review={\MR{666529}},
}

\bib{Saut1979}{article}{
      author={Saut, J.-C.},
       title={Quelques g\'{e}n\'{e}ralisations de l'\'{e}quation de
  {K}orteweg-de\thinspace {V}ries. {II}},
        date={1979},
        ISSN={0022-0396,1090-2732},
     journal={J. Differential Equations},
      volume={33},
      number={3},
       pages={320\ndash 335},
         url={https://doi.org/10.1016/0022-0396(79)90068-8},
      review={\MR{543702}},
}

\bib{Sjolin1987}{article}{
      author={Sj\"{o}lin, Per},
       title={Regularity of solutions to the {S}chr\"{o}dinger equation},
        date={1987},
        ISSN={0012-7094,1547-7398},
     journal={Duke Math. J.},
      volume={55},
      number={3},
       pages={699\ndash 715},
         url={https://doi.org/10.1215/S0012-7094-87-05535-9},
      review={\MR{904948}},
}

\bib{Takaoka1999}{article}{
      author={Takaoka, Hideo},
       title={Well-posedness for the one-dimensional nonlinear
  {S}chr\"{o}dinger equation with the derivative nonlinearity},
        date={1999},
        ISSN={1079-9389},
     journal={Adv. Differential Equations},
      volume={4},
      number={4},
       pages={561\ndash 580},
      review={\MR{1693278}},
}

\bib{TaoBook2006}{book}{
      author={Tao, Terence},
       title={Nonlinear dispersive equations},
      series={CBMS Regional Conference Series in Mathematics},
   publisher={Conference Board of the Mathematical Sciences, Washington, DC; by
  the American Mathematical Society, Providence, RI},
        date={2006},
      volume={106},
        ISBN={0-8218-4143-2},
         url={https://doi.org/10.1090/cbms/106},
        note={Local and global analysis},
      review={\MR{2233925}},
}

\end{biblist}
\end{bibdiv}
\end{document}